%% file: main.tex
\documentclass[reqno,12pt]{amsart}
\input{preamble}

\author{Max Demirdilek} 

\title{Grothendieck--Verdier functors}

\begin{document}

\begin{abstract}
We introduce Grothendieck--Verdier functors between Grothendieck--Verdier, or $\ast$-autonomous, categories. Such functors are lax monoidal functors equipped with a morphism expressing compatibility with Grothendieck--Verdier duality. We show that the \mbox{resulting} \mbox{$2$-category} is $2$-equivalent to that of linearly distributive categories with negation and Frobenius linearly distributive functors. We further extend this $2$-equivalence to the braided setting.

We then establish a lifting theorem for Grothendieck--Verdier functors: given a conservative lax monoidal functor from a closed monoidal category $\cC$ to a Grothendieck--Verdier category $\cD$, we identify additional structure such that the Grothendieck--Verdier structure of $\cD$ lifts to $\cC$. This structure turns the functor into a Grothendieck--Verdier functor. As applications, we recover and extend conditions under which modules over Hopf monads and Hopf algebroids inherit Grothendieck--Verdier structures. We also characterize when categories of bimodules, modules, and local modules over (commutative) algebras internal to a Grothendieck--Verdier category admit such structures. Our results apply to quantales, smash product algebras, skew group algebras, and enveloping algebras of Lie--Rinehart algebras.
\end{abstract}

\maketitle

\vspace{-0.9cm}
\tableofcontents
\vspace{-0.9cm}

\section{Introduction}
The monoidal category of finite-dimensional modules over a Hopf algebra with invertible antipode has rigid duals. This is an instance of a lifting principle: if \(U\colon \cC \to \cD\) is a strong monoidal functor between closed monoidal categories that reflects isomorphisms and preserves internal homs, then the rigidity of \(\cD\) implies that of \(\cC\); see \cite{BLV}. In particular, taking \(U\) to be the fiber functor to finite-dimensional vector spaces, one recovers the rigidity of finite-dimensional modules over a Hopf algebra with invertible antipode.
 
 \medskip
 
 \textit{\textbf{Grothendieck--Verdier categories.}}
Rigid duality is often too restrictive; for example, it forces the tensor product to be exact. \emph{Grothendieck--Verdier (GV) categories}, also known as \(\ast\)-autonomous categories, provide a more flexible notion. Like rigid categories, they are necessarily closed (Remark~\ref{GV implies closed}). However, unlike in the rigid case, left and right duals in a GV-category are generally not determined by the monoidal structure alone; rather, they are defined relative to a choice of \emph{dualizing object} (Definition~\ref{def:GV-category}). GV-categories appear in mathematical physics and representation theory \cite{allen2021duality, allen2024hopfalgebroidsgrothendieckverdierduality, fuchs2024grothendieckverdierdualitycategoriesbimodules}, linear logic \cite{See89,MelCatSem}, functional analysis \cite{BarrStarAut}, and algebraic geometry \cite{BoDrinfeld}. They also serve as input for constructions in quantum topology \cite{CyclicFramedLittleDisk,CatOpenTopField}.

 \medskip

Motivated by the above lifting principle for rigid categories, we seek an analogous criterion for Grothendieck--Verdier (GV) duality:
 
 \medskip 

\textit{\textbf{Question.}} Given a functor \(U\colon \cC\to \cD\) between closed monoidal categories, where \(\cD\) is a GV-category, what additional structure on \(U\) ensures that the GV-duality on \(\cD\) lifts to \(\cC\)?

\medskip

To answer this question, we introduce a Frobenius-type structure on a lax monoidal functor.

\smallskip

\textit{\textbf{Frobenius forms}}.
Let \(\cC\) and \(\cD\) be closed monoidal categories, with distinguished objects \(K\in \cC\) and \(k\in \cD\). A \emph{Frobenius form} on a lax monoidal functor \(U\colon \cC \rightarrow \cD\) is a morphism 
\begin{align}\label{mor: Frob form}
	\upsilon^{0,U}\colon\: U(K) \:\longrightarrow\: k
\end{align} 	
such that the induced morphisms (Definition \ref{def: duality transformations})
\begin{align}\label{mor: induced mor}
	U(X\multimap K) \:\longrightarrow\: U(X) \multimap k \qquad \text{and} \qquad U(K \multimapinv X) \:\longrightarrow\: k \multimapinv U(X)\
\end{align} 
are invertible for all \(X\in \cC\). Here, \(\multimap\) and \(\multimapinv\) denote the left and right internal homs.

\sloppy When \(K\) and \(k\) are dualizing, the Frobenius form \eqref{mor: Frob form} expresses compatibility with Grothendieck--Verdier duality. Indeed, since the left and right duals in a GV-category are given by internal homs into the dualizing object, the invertibility of the morphisms \eqref{mor: induced mor} means that the functor preserves both left and right duals.

\medskip 

Using this notion, we answer the above question:

\begin{liftingthm}[Theorem~\ref{main thm}]\label{thm: A}
Let \(U\colon \cC\to \cD\) be a lax monoidal functor between closed monoidal categories, with distinguished objects \(K\in \cC\) and \(k\in \cD\). Assume that \(k\in \cD\) is dualizing. 

If \(U\) is conservative and admits a Frobenius form, then \(\cC\) is a Grothendieck--Verdier category with dualizing object \(K\).
\end{liftingthm}

This lifting theorem suggests the following definition, which gives the paper its title and, to the best of our knowledge, is new.

\smallskip

\textit{\textbf{Grothendieck--Verdier functors}}. A \emph{Grothendieck--Verdier (GV) functor} (Definition~\ref{def: GV-functor}) is a lax monoidal functor between Grothendieck--Verdier categories equipped with a Frobenius form relative to the dualizing objects.

Conceptually, GV-functors may be viewed as multi-object generalizations of Frobenius algebras. Indeed, just as lax monoidal functors generalize unital associative algebras, a GV-Frobenius algebra in the sense of \cite{fuchs2024grothendieckverdier, DeS} is exactly a GV-functor from the terminal category. In a similar vein, GV-categories themselves can be regarded as categorified Frobenius algebras; see \cite{QuantumCatStreetDay, star-autonomoFrob, CyclicFramedLittleDisk}.

\medskip

 \textbf{\textit{A \(2\)-equivalence.}} 
 Frobenius algebras in a monoidal category admit two equivalent characterizations: as algebras equipped with a Frobenius form, or as algebras endowed with a compatible coalgebra structure. We extend this equivalence to the doctrine of GV-categories and GV-functors:
 Let \(\mathsf{GV}\) be the \((2,1)\)-category of GV-categories, GV-functors, and their morphisms (Definition~\ref{def: morphism of GV-functors}). Let \(\mathsf{LDN}\) be the \((2,1)\)-category of linearly distributive categories with negation (Definition~\ref{def: LD-category with negation}), Frobenius linearly distributive functors (Definition~\ref{def: Frobenius LD-functor}), and their morphisms (Definition~\ref{def: morphism of Frob LD-functors}); see also \cite{WDC,DeS}.
 
 Our second main result is:

\begin{sndthm}[Theorem~\ref{thm: GV equiv LDN}]\label{thm: B}
	 \(\mathsf{GV}\) and \(\mathsf{LDN}\) are {\(2\)-equivalent}. The \(2\)-equivalence can be chosen to strictly commute with the forgetful \(2\)-functors to the \((2,1)\)-category \(\mathsf{MonCat}_g\) of monoidal categories, lax monoidal functors, and monoidal natural isomorphisms:
\[
\begin{tikzcd}
	\mathsf{GV} \arrow[rr, <->, "\simeq"] \arrow[rd, bend right=15, "{\text{forget}}"']
	&&
	\mathsf{LDN} \arrow[ld, bend left=15, "{\text{forget}}"]\\
	& \mathsf{MonCat}_g. &
\end{tikzcd}
\]
\end{sndthm}

\medskip

This extends \cite[Thm. 4.5]{WDC} and builds on \cite{fuchs2024grothendieckverdierdualitycategoriesbimodules, fuchs2024grothendieckverdier, DeS}. It also shows that the conservative functor \(U\) from Theorem~\hyperref[thm: A]{A} acquires the structure of a Frobenius linearly distributive functor and, in particular, preserves GV-Frobenius algebras (Remark~\ref{rem: Liftingt thm functor ist Frob LD}).

Moreover, upon specialization to \emph{rigid} monoidal categories, the \(2\)-equivalence provides a new characterization of Frobenius monoidal functors \(U\colon (\cC,\otimes,1) \to (\cD,\otimes,1)\), as defined in \cite{NoteOnFrobMonFunctors} and recalled in Definition~\ref{def: Frobenius monoidal functor}: namely, such functors are precisely lax monoidal functors equipped with a Frobenius form \(U(1)\to 1\) (Remark~\ref{rem: rigid case of 2-equivalence}).

\smallskip

\textit{\textbf{Adding braided structure.}}
The above \(2\)-equivalence lifts to braided categories and functors: Let \(\mathsf{BrGV}\) be the \((2,1)\)-category of braided GV-categories (Definition~\ref{def: braided GV}), braided GV-functors (Definition~\ref{def: braided GV functor}), and their morphisms. Let \(\mathsf{BrLDN}\) be the \((2,1)\)-category of braided linearly distributive categories with negation (Definition~\ref{def:braided LD-category}), in the sense of \cite{MelCatSem}, braided Frobenius linearly distributive functors (Definition~\ref{def: braided Frobenius LD-functor}), and their morphisms.

Our third main result is:
\begin{thdthm}[Theorem~\ref{thm: braided GV=braided LDneg}]\label{thm: C}
	\(\mathsf{BrGV}\) and \(\mathsf{BrLDN}\) are {\(2\)-equivalent}. The \(2\)-equivalence can be chosen to strictly commute with the forgetful \(2\)-functors to the \((2,1)\)-category \(\mathsf{BrMonCat}_g\) of braided monoidal categories, braided lax monoidal functors, and monoidal isomorphisms:
	\[
	\begin{tikzcd}
		\mathsf{BrGV} \arrow[rr, <->, "\simeq"] \arrow[rd, bend right=15, "{\text{forget}}"']
		&&
		\mathsf{BrLDN} \arrow[ld, bend left=15, "{\text{forget}}"]\\
		& \mathsf{BrMonCat}_g. &
	\end{tikzcd}
	\]
\end{thdthm}
The main difficulty in extending Theorem~\hyperref[thm: B]{B} to Theorem~\hyperref[thm: C]{C} is that braided linearly distributive categories and their functors must satisfy hexagon identities involving the coherence data of linearly distributive categories, which is intricate to construct from the GV-structure.

\smallskip

\textbf{\textit{Applications.}}
We now turn to applications of the lifting theorem (Theorem~\hyperref[thm: A]{A}). This theorem unifies various examples of Grothendieck--Verdier categories:

\begin{itemize}
	\item \textit{Modules and local modules (Propositions~\ref{prop: A-mod is GV} and~\ref{prop: local modules}).} 
	For a commutative algebra \(A\) in a braided GV-category \(\cC\) satisfying mild assumptions, the categories of \(A\)-modules and of local \(A\)-modules inherit GV-structures. This recovers results of Creutzig--McRae--Shimizu--Yadav \cite[Thm. 3.9, 3.11]{Creutzig_2025}.
	
	\item \textit{Hopf monads (Proposition~\ref{hopf monad on gv}).}
	We give a short proof of a result of Hasegawa--Lemay \cite[Thm. 5.9]{HaLe}: for a Hopf monad \(T\) on a GV-category, \(T\)-module structures on the dualizing object correspond to lifts of the GV-structure to the category of \(T\)-modules.
\end{itemize}	
The lifting theorem also allows us to extend existing examples:
\begin{itemize}
	\item \textit{Bimodules (Proposition~\ref{prop: category of bimodules is gv}).} For an algebra \(A\) in a GV-category \(\cC\) satisfying mild assumptions, the category of \(A\)-bimodules in \(\cC\) inherits a GV-structure from \(\cC\).
	\item \textit{Hopf algebroids (Proposition~\ref{Hopf algebroids lifting}).} 
	For a Hopf algebroid \(H\) over a finite-dimensional \(k\)-algebra \(R\), any \(H\)-module structure on the \(k\)-linear dual \(R\)-bimodule \(R^\ast:=\operatorname{Hom}_k(R,k)\) yields a GV-structure on the category of finite-dimensional \(H\)-modules. In particular, Hopf algebroids with an invertible antipode give rise to GV-categories (Corollary~\ref{cor: full Hopf gives GV}), recovering the main result of Allen \cite{allen2024hopfalgebroidsgrothendieckverdierduality}.
	\item \textit{Explicit examples.} 
	\sloppy We compute GV-structures for several families of Hopf algebroids, including smash product algebras and skew group algebras (Examples~\ref{ex: GV smash product algebras} and~\ref{ex: GV skew group}). This yields GV-structures on the category of finite-dimensional \(G\)-equivariant \mbox{\(R\)-modules}, where \(G\) is a group and \(R\) is a finite-dimensional commutative \(k\)-algebra with a \(G\)-action. We also consider universal enveloping algebras of Lie--Rinehart algebras, including the truncated modular Weyl algebra (Examples~\ref{ex: enveloping algebras} and~\ref{ex: trunctated mod Weyl algebra}). These generalize enveloping algebras of Lie algebras and, in appropriate cases, include algebras of differential operators as special instances.
\end{itemize}

In each example above, the lifting theorem is applied to the forgetful functor from a category of modules. The theorem then ensures that each such functor is a GV-functor. In the case of Hopf monads and Hopf algebroids, this is a strict monoidal functor, and the Frobenius form is the identity. In the other examples, the functor is only lax monoidal, and the Frobenius form is the GV-dual of the unit of the algebra \(A\).

\medskip
 
\textbf{\textit{Outline.}}
In Section \ref{sec: GV functors}, we introduce the notion of Grothendieck--Verdier functors and morphisms between them. Before that, we review closed monoidal categories (Subsection \ref{sec: closed monoidal categories}), their functors (Subsection \ref{sec: functors between closed monoidal categories}), and Grothendieck--Verdier categories (Subsection \ref{sec: GV-categories and closed monoidal categories}).

Section~\ref{sec: two 2-equivalence} establishes the \(2\)-equivalences of Theorem~\ref{thm: GV equiv LDN} and Theorem~\ref{thm: braided GV=braided LDneg}, after recalling (braided) LD-categories, (braided) Frobenius LD-functors, and their morphisms, and introducing braided GV-functors. Section~\ref{sec:liftingGVstructures} proves the lifting theorem. Section~\ref{sec: Applications} applies it to Hopf monads, Hopf algebroids, and categories of (bi)modules, and develops explicit examples. The relevant algebraic structures are reviewed beforehand in Section~\ref{sec: lifting theorem prep}. All technical proofs but the one of the lifting theorem are deferred to Appendix~\ref{sec:appendixproof}.

\addtocontents{toc}{\protect\setcounter{tocdepth}{1}}
\subsection{Notation and conventions}
\begin{itemize}
	\item All \(2\)-categories and \(2\)-functors appearing in this paper are strict. 
	\item A \emph{\((2,1)\)-category} is a \(2\)-category whose \(2\)-cells are invertible.
	\item A \emph{(strict) \(2\)-natural transformation} \(\zeta\colon F \to G\) between strict \(2\)-functors \(F,G\colon \mathsf{C}\to \mathsf{D}\) consists of, for each \(0\)-cell \(X\in \cC\), a \(1\)-morphism \(\zeta_X \colon {F(X) \to G(X)}\) in \(\mathsf{D}\), such that for every \(1\)-morphism \(f\colon X \to Y\) in \(\mathsf{C}\)
	\begin{equation}
	G(f)\circ \zeta_X\;=\;\zeta_Y\circ F(f).
	\end{equation}
	It is called a \emph{(strict) \(2\)-natural isomorphism} if each component \(\zeta_X\) is an isomorphism.
	\item A functor \(F\colon \cC\rightarrow \cD\) is called \emph{conservative} if any morphism \(f\in \operatorname{Hom}_{\cC}(X,Y)\) whose image \(F(f)\) is invertible, is an isomorphism.
	\item The terminal category is denoted by \(\ast\).
	\item For a morphism \(f\colon X \to Y\) in a category \(\cC\) and an object \(Z\in \cC\), we write
	\begin{align}
	f^{\ast} &\,:=\,\operatorname{Hom}_{\cC}(f,Z)\colon\, \operatorname{Hom}_{\cC}(Y,Z) \rightarrow \operatorname{Hom}_{\cC}(X,Z),\\ 
	f_{\ast}&\,:=\, \operatorname{Hom}_{\cC}(Y,f)\colon\, \operatorname{Hom}_{\cC}(Z,X) \rightarrow \operatorname{Hom}_{\cC}(Z,Y),
	\end{align}
	for the maps given by precomposition and postcomposition with \(f\), respectively.
\end{itemize}

\smallskip

\nid Let \(\cC=(\cC,\otimes,1)\) be a monoidal category. 
\begin{itemize}
	\item The \emph{reversed monoidal product} on \(\cC\) is defined by \({X\otimes^{\text{rev}}Y:=Y\otimes X}\) for \(X,Y\in \cC\). We write \(\cC^{\text{rev}}=(\cC,\otimes^{\text{rev}},1)\) for the resulting \emph{reversed monoidal category}.
	\item When clear from context, we omit the associator 
	\begin{align*}
		\ao \colon\; {\otimes} {\,\circ\,} {(\operatorname{id_\cC} \times \ot)} \,\sxlongrightarrow{\simeq}\, {\ot} {\,\circ\,} {(\otimes \times \operatorname{id_\cC})}
	\end{align*}
	and the left and right unitors  
	\begin{align*}
		\lou\colon\; {\ot} \circ (1\tim \idC) \,\sxlongrightarrow{\simeq}\, \idC, \qquad \qquad \rou \colon\; {\ot} \circ (\idC \tim 1)\,\sxlongrightarrow{\simeq}\, \idC.
	\end{align*}
	For readability, we may also suppress their indices.
	\item Algebras in a monoidal category are unital associative, and their morphisms are unital.
\end{itemize}
\addtocontents{toc}{\protect\setcounter{tocdepth}{2}}


\section{Grothendieck--Verdier functors}\label{sec: GV functors}

\addtocontents{toc}{\protect\setcounter{tocdepth}{1}}
\subsection{Closed monoidal categories}\label{sec: closed monoidal categories}

We recall standard categorical notions:
\begin{definition}
A monoidal category \(\cC=(\cC,\otimes,1)\) is \emph{left closed} if, for each \(X\in \cC\), the endofunctor \(X\otimes{?} \colon \,\cC\rightarrow \cC\) has a right adjoint \(X\multimap {?}\colon \, \cC\rightarrow \cC\). Similarly, it is \emph{right closed}, if, for each \(X\in \cC\), the endofunctor \({?}\otimes X\colon\, \cC\rightarrow \cC\) has a right adjoint \({?}\multimapinv X\colon \, \cC\rightarrow \cC\). If \(\cC\) is both left and right closed, it is called \emph{closed}.
\end{definition}

\begin{remark}[Terminology]\label{internal homs}
	The family of endofunctors \({\{X\multimap {?}\}_{X\in \cC}}\) extends to a bifunctor \(\multimap\colon\; \cC^{\mathrm{op}}\times \cC \,\longrightarrow\, \cC,\) called the \emph{left internal hom}, uniquely determined by
	\begin{equation}\label{AdjMor1}
		\Phi_{X,Y,Z}\colon \;\operatorname{Hom}_{\cC}(X\otimes Y,Z)\,\xlongrightarrow{\simeq}\, \operatorname{Hom}_{\cC}(Y,X\multimap Z)
	\end{equation}
	being natural in all three components \(X,Y,Z\in \cC\). Similarly, from an adjunction isomorphism
	\begin{align}\label{AdjMor2}
		\highoverline{\Phi}_{X,Y,Z}\colon \operatorname{Hom}_{\cC}(Y\otimes X,Z)\,\xlongrightarrow{\simeq}\,\operatorname{Hom}_{\cC}(Y,Z\multimapinv X),
	\end{align}
	one obtains a bifunctor \(\multimapinv\colon\; \cC\times \cC^{\mathrm{op}} \rightarrow \cC\), called the \emph{right internal hom}.
\end{remark}

\begin{remark}[Coclosedness]
	A monoidal category \(\cC\) is left (resp. right) \emph{coclosed} if its opposite \((\cC^{\operatorname{op}},\otimes,1)\) is left (resp. right) closed. Consequently, any result for closed monoidal categories dualizes to coclosed ones.
\end{remark}

\begin{remark}[Notation]
For a left closed monoidal category \(\cC\) and \(X\in \cC\), we denote the unit (`\emph{coevaluation}') and counit (`\emph{evaluation}') of the left \emph{tensor-hom adjunction} (\ref{AdjMor1}) by \(\operatorname{coev}^{X}\colon\, \operatorname{id}_{\cC}\rightarrow X\multimap (X\otimes {?})\) and \(\operatorname{ev}^{X}\colon \,X\otimes (X\multimap {?})\rightarrow \operatorname{id}_{\cC}\). For an object \(X\) in a right closed monoidal category \(\cC\), the unit and counit of the right tensor-hom adjunction are denoted by \(\higheroverline{\operatorname{coev}}^{X}\colon\, \operatorname{id}_{\cC}\rightarrow ({?}\otimes {X})\multimapinv X\) and \(\higheroverline{\operatorname{ev}}^{X}\colon \,({?}\multimapinv X) \otimes X \rightarrow \operatorname{id}_{\cC}\).
\end{remark}

Let \(\cC\) be a left closed monoidal category.
\begin{lemma}\label{lemma:extranaturality of eval and coeval}\emph{(Cf. \cite{EilenbergKellyFun}).}
    The evaluation \(\operatorname{ev}^X\) and coevaluation \(\operatorname{coev}^X\) are both extranatural in the component \(X\in \cC\). This means that for all \(X,Y,Z \in \cC\) and every \(f\in \operatorname{Hom}_{\cC}(X,Y)\),
    \begin{align}
        {\operatorname{ev}^X_Z} \circ {\big(X \otimes (f\multimap Z)\big)} & \;=\; {\operatorname{ev}^Y_Z} \circ {\big(f\otimes(Y \multimap Z)\big)},\label{eval is natural}\\
        {\big(X \multimap (f \otimes Z)\big)} \circ {\operatorname{coev}^X_Z} & \;=\; {\big(f\multimap (Y \otimes Z)\big)} \circ {\operatorname{coev}^Y_Z}. \label{coeval is natural}
    \end{align}
    Analogous identities hold for any right closed monoidal category.
\end{lemma}

\begin{remark}[Internal composition]\label{rem:internal composition}
The \emph{(left) internal composition} in \(\cC\) is the family 
\begin{gather*}
   \operatorname{comp}^l_{X,Y,Z}\colon \; (X\multimap Y) \otimes (Y\multimap Z) \longrightarrow X\multimap Z,\\
   \operatorname{comp}^l_{X,Y,Z}\;:=\;(X\multimap \operatorname{ev}^Y_Z) \circ \big(X\multimap({\operatorname{ev}^X_Y}\otimes {(Y\multimap Z)})\big)\circ {\big(X \multimap\ao\big)} \circ {\operatorname{coev}^X_{(X\multimap Y)\otimes (Y\multimap Z)}},
\end{gather*}
natural in \(X,Z\in \cC\) and extranatural in \(Y\in \cC\). The \emph{(right) internal composition} \(\operatorname{comp}^r\) in a right closed monoidal category is defined analogously.
\end{remark}

\begin{lemma}\label{lemma: assoc of composition}\emph{(See e.g.\ \cite[\S 7.9]{EGNO}).}
The internal composition \(\operatorname{comp}^l\) is associative:
\begin{equation*}
    {\operatorname{comp}^l_{W,X,Z}}\circ {\big((W\multimap X)\otimes \operatorname{comp}^l_{X,Y,Z}\big)} \;=\; {{\operatorname{comp}^l_{W,Y,Z}} \circ {\big({\operatorname{comp}^l_{W,X,Y}}\otimes {(Y\multimap Z)}\big)}\circ {\ao}},
\end{equation*}
for all \(W,X,Y,Z\in \cC\). The right internal composition satisfies an analogous identity.
\end{lemma}

\begin{remark}[Internal algebra]
	For \(X\in \cC\), the \emph{endomorphism object} $E_X:=X\multimap X$
	is an algebra (Lemma~\ref{lemma: assoc of composition}) with multiplication \(\operatorname{comp}^l_{X,X,X}\) and unit 	
	\begin{equation*}
		\operatorname{e}^X:=(X\multimap \rho_X) \circ \operatorname{coev}^X_1\in \operatorname{Hom}_{\cC}(1,E_X).
	\end{equation*} 
\end{remark}

\begin{lemma}\label{lem:module over interal algebra}
	For any $X,Y\in \cC$, the object \(X\multimap Y\) carries an \((E_X,E_Y)\)-bimodule structure via internal composition, by Lemma~\ref{lemma: assoc of composition}.
\end{lemma}

The following result is an immediate consequence of Lemma~\ref{lemma:extranaturality of eval and coeval}.

\begin{lemma}\label{lem:internal composition is extranatural}
    The internal composition \(\operatorname{comp}^l_{X,Y,Z}\) is extranatural in \(Y\), and the unit \(e^X\) is extranatural in \(X\). This means that for all \(W,X,Y,Z \in \cC\) and \(f\in \operatorname{Hom}_{\cC}(X,Y)\),
    \begin{align}
         {\operatorname{comp}^l_{W,X,Z}}\circ {\big((W\multimap X)\otimes (f\multimap Z)\big)} &\;=\; {\operatorname{comp}^l_{W,Y,Z}}\circ {\big((W\multimap f)\otimes (Y\multimap Z)\big)},\\
        {(f\multimap Y)\circ \operatorname{e}^Y} & \;=\;{(X\multimap f) \circ \operatorname{e}^X}.
    \end{align}
    Analogous identities hold for any right closed monoidal category.
\end{lemma}

Any left closed monoidal category $\cC$ comes with further canonical morphisms, which we recall for later use:

\begin{remark}[Internal hom tensorality]\label{rem: internal hom tensorality}
	For $X,Y,Z\in \cC$, define the morphism
	\begin{gather*}
		\underline{X\otimes }_{Y,Z}\colon\;Y\multimap Z \;\longrightarrow\; (X\otimes Y)\multimap (X\otimes Z),\\
		\underline{X\otimes }_{Y,Z} \;:=\; {\big((X\otimes Y)\multimap(X \otimes {\operatorname{ev}^Y_Z})\big)} \circ {\operatorname{coev}^{X\otimes Y}_{Y \multimap Z}}.
	\end{gather*}
	This is natural in \(Y,Z\in \cC\) and extranatural in \(X\in \cC\).
\end{remark}

\begin{remark}[Canonical isomorphisms]\label{rem:canonical isos}
    By Yoneda's lemma, the associator induces an isomorphism (`\emph{internal adjunction isomorphism}')
    \begin{align}
    {\beta_{X,Y,Z}\colon\; (X\otimes Y)\multimap Z} & \,\xlongrightarrow{\simeq}\, {Y \multimap (X\multimap Z)},
    \end{align}
    natural in $X,Y,Z\in \cC$.   
Similarly, the left unitor yields an isomorphism
\begin{align}\label{gamma unitor}
 	 {\gamma_{X}\colon\; 1\multimap X} & \,\xlongrightarrow{\simeq}\,{X},
    \end{align}
    natural in $X \in \cC$. For a right closed monoidal category, one obtains analogous isomorphisms
    \begin{align}
\higheroverline{\beta}_{X,Y,Z}\colon\; Z\multimapinv (X\otimes Y) & \,\xlongrightarrow{\simeq}\, (Z \multimapinv Y)\multimapinv X,\\
\higheroverline{\gamma}_{X}\colon\; X\multimapinv 1 & \,\xlongrightarrow{\simeq}\, X.
    \end{align}
\end{remark}

Using Mac Lane's coherence theorem together with Yoneda's lemma, we have:
\begin{lemma}\label{lemma: relation beta and gamma}
 The isomorphism \(\beta\) is compatible with unitors:
	\begin{align}
		(\lambda_{X} \multimap \gamma_Y) \;=\; \beta_{1,X,Y}^{-1},\label{eq: relation beta and gamma1}\\
		(\rho_{X} \multimap Y) \circ \gamma_{X\multimap Y} \;=\; \beta_{X,1,Y}^{-1},
	\end{align}
	for all $X,Y\in \cC$. Analogous identities hold in any right closed monoidal category.
\end{lemma}

The following lemma is a routine verification.

\begin{lemma}\label{lemma: beta  compatibility}
	The isomorphism \(\beta\) is compatible with (co)evaluations:
	\begin{align}
		{\operatorname{ev}^{X\otimes Y}_Z} \circ \big((X \otimes Y) \otimes \beta^{-1}_{X,Y,Z}\big) &\,=\, {\operatorname{ev}^X_Z} \circ {\big({X}\otimes {\operatorname{ev}^{Y}_{X \multimap Z}}\big)} \circ {\alpha^{-1}_{X,Y,Y\multimap (X\multimap Z)}},\label{eq: ev compatible with beta}\\[.5em]
		\beta_{X,Y,(X\otimes Y)\otimes Z} \circ \operatorname{coev}^{X\otimes Y}_Z &\,=\, \big(Y\multimap (X\multimap \alpha_{X,Y,Z})\big) \circ \big(Y \multimap \operatorname{coev}^{X}_{Y \otimes Z}\big)\circ \operatorname{coev}^Y_Z,\label{eq: coev compatible with beta}
	\end{align}
	for all \(X,Y,Z \in \cC\). Analogous equations hold in any right closed monoidal category.
\end{lemma}

Let $\cC$ now be a \emph{closed} monoidal category. We relate the left and right internal homs of $\cC$:

\begin{remark}[Another canonical isomorphism]\label{rem:another can iso}
    For $X,Y,Z\in \cC$, define
\begin{equation}\label{eq: associator for internal homs}
\iota_{X,Y,Z}\colon\; (X\multimap Y)\multimapinv Z \,\xlongrightarrow{\simeq}\, X\multimap (Y\multimapinv Z)
\end{equation}
as the unique isomorphism characterized by the property that, for all \(W\in \cC\), the map
\begin{align*}
	(\iota_{X,Y,Z})_{\ast}\colon\;\operatorname{Hom}_{\cC}(W,(X\multimap Y)\multimapinv Z)\,\xlongrightarrow{\simeq}\,\operatorname{Hom}_{\cC}(W,X \multimap (Y \multimapinv Z)),
\end{align*}
is equal to the composite 
\begin{align*}
  \operatorname{Hom}_{\cC}(W,(X\multimap Y)\multimapinv Z)\xrightarrow{\simeq}\operatorname{Hom}_{\cC}(W \otimes Z,X\multimap Y)\xrightarrow{\simeq} \operatorname{Hom}_{\cC}(X \otimes (W \otimes Z),Y) \xrightarrow{(\alpha^{-1}_{X,W,Z})^{\ast}}\\ 
  \operatorname{Hom}_{\cC}((X \otimes W) \otimes Z,Y)\xrightarrow{\simeq}\operatorname{Hom}_{\cC}(X \otimes W,Y \multimapinv Z)\xrightarrow{\simeq}\operatorname{Hom}_{\cC}(W,X \multimap (Y \multimapinv Z)),
\end{align*}
where the unlabeled morphisms are instances of adjunction isomorphisms.
\end{remark}

The following lemma is a routine verification.

\begin{lemma}\label{lemma: iota compatibility}
 Let \(\cC\) be a closed monoidal category. 
 \begin{enumerate}[label=(\roman*)]
     \item The isomorphism \(\iota\) is compatible with (co)evaluations:
    \begin{align}
   \higheroverline{\operatorname{ev}}^Z_Y \circ {\big({\operatorname{ev}^X_{Y \multimapinv Z}} \otimes {Z}\big)} \circ {\big((X \otimes \iota^{-1}_{X,Y,Z}) \otimes Z \big)}\circ {\alpha}^{-1}&\,=\, \operatorname{ev}^X_Y\circ {\big(X\otimes {\higheroverline{\operatorname{ev}}^Z_{X\multimap Y}}\big)},\label{eq: ev compatible with iota}\\[.5em]
  \iota \circ \big((X\multimap \alpha)\multimapinv Z \big) \circ \big(\operatorname{coev}^X_{Y \otimes Z}\multimapinv Z\big)\circ \higheroverline{\operatorname{coev}}_Y^Z &\,=\, \big(X\multimap \higheroverline{\operatorname{coev}}^Z_{X\otimes Y}\big)\circ \operatorname{coev}_Y^X,\label{eq: coev compatible with iota}
 \end{align}
 for all objects \(X,Y,Z \in \cC\).
 \item The isomorphism \(\iota\) is compatible with the isomorphisms \(\beta,\higheroverline{\beta},\gamma\) and \(\higheroverline{\gamma}\):
 \begin{align}
 \beta_{W,X,Y\multimapinv Z}\circ \iota_{W\otimes X,Y,Z} &\,=\, (X\multimap \iota_{W,Y,Z})\circ \iota_{X,W\multimap Y,Z}\circ(\beta_{W,X,Y}\multimapinv Z),\label{eq: pentagon iota 1}\\
 \higheroverline{\beta}_{Z,Y,W\multimap X}\circ \iota_{W,X,Y\otimes Z}^{-1} &\,=\, (\iota^{-1}_{W,X,Z}\multimapinv Y)\circ \iota^{-1}_{W,X\multimapinv Z,Y}\circ(W\multimap \higheroverline{\beta}_{Z,Y,X}),\label{eq: pentagon iota 2}\\
 \gamma_{X\multimapinv Y}\circ \iota_{1,X,Y} &\,=\, \gamma_X \multimapinv Y,\label{eq: triangle iota 1}\\
 \higheroverline{\gamma}_{X\multimap Y}\circ \iota_{X,Y,1}^{-1} &\,=\, X \multimap \higheroverline{\gamma}_Y,\label{eq: triangle iota 2}
 \end{align}
 for all \(W,X,Y,Z\in \cC\).
 \end{enumerate}
\end{lemma}

\smallskip

\subsection{Closed monoidal functors}\label{sec: functors between closed monoidal categories}
The following morphisms will play an important role.
\begin{definition}\label{def: closed monoidal functor}
    Let \(\cC=(\cC,\otimes,1)\) and \(\cD=(\cD,\otimes,1)\) be left closed monoidal categories, and let \({F\colon \cC\rightarrow \cD}\) be a lax monoidal functor with multiplication morphism \(\varphi^{2,F}\). 
    
    The \emph{(left internal hom) comparator} of \(F\) is the family of morphisms, natural in \(X,Y\in \cC\),
        \begin{align}\label{ClosedFunctor1} 
        \tau^{l,F} & \;=\;\{\tau^{l,F}_{X,Y} \colon \; F(X\multimap{Y})\,\longrightarrow\, F(X)\multimap {F(Y)}\}_{X,Y\in \cC},\\
        \tau^{l,F}_{X,Y} & \;:=\;\Big(F(X)\multimap \big(F(\operatorname{ev}^X_Y)\circ\varphi^{2,F}_{X,X\multimap\, {Y}}\big)\Big) \circ \operatorname{coev}^{F(X)}_{F(X\multimap\, {Y})}.
        \end{align}
        
        For right closed monoidal categories, the \emph{(right internal hom) comparator} of \(F\),
        \begin{align*}
        \tau^{r,F} \;=\;{{\{\tau^{r,F}_{Y,X} \colon \; F(Y\multimapinv{X}) \,\longrightarrow\, F(Y)\multimapinv {F(X)}\}_{X,Y\in \cC}}},
        \end{align*} 
        is defined similarly.
\end{definition}

The following lemmas record routine compatibilities of the comparators.

\begin{lemma}\label{lemma: compatibility of multi and left internal hom transfo}
    Let \(F\colon \cC \rightarrow \cD\) be a lax monoidal functor between left closed monoidal categories. Its left comparator \({\tau^{l,F}}\) is compatible with (co)evaluations:
\begin{align}
        {F(\operatorname{ev}^X_Y)} \circ {\varphi^{2,F}_{X,X\multimap Y}} &\;=\; {\operatorname{ev}^{F(X)}_{F(Y)}} \circ {\big(F(X) \otimes \tau^{l,F}_{X,Y}\big)}, \label{eval lax functor compatibility}\\
        {\tau^{l,F}_{Y,Y\otimes X}} \circ {F(\operatorname{coev}^Y_X)} &\;=\; {\big(F(Y)\multimap \varphi^{2,F}_{X,Y}\big)} \circ {\operatorname{coev}^{F(Y)}_{F(X)}},\label{coeval lax functor compatibility}
    \end{align}
for all \(X,Y \in \cC\). Analogous relations hold for right closed monoidal categories.
\end{lemma}

\begin{lemma}\label{lemma: compatibility of multi and beta}
	Let \(F\colon \cC \rightarrow \cD\) be a lax monoidal functor between left closed monoidal categories. Its left comparator \({\tau^{l,F}}\) is compatible with the isomorphism \(\beta\) from Remark~\ref{rem:canonical isos}:
	\begin{align}
		\beta_{F(X),F(Y),F(Z)} \circ {\big(\varphi^{2,F}_{X,Y} \multimap F(Z)\big)} \circ \tau^{l,F}_{X\otimes Y,Z} \;=\; 
		\big(F(Y)\multimap \tau^{l,F}_{X,Z}\big) \circ \tau^{l,F}_{Y,X\multimap Z}\circ F(\beta_{X,Y,Z}),
	\end{align}
	for all \(X,Y,Z \in \cC\). An analogous equation holds for right closed monoidal categories.
\end{lemma}

\begin{lemma}\label{lem:compatibility composition and internal hom transfo}
Let \(F\colon \cC \rightarrow \cD\) be a lax monoidal functor between left closed monoidal categories, with multiplication morphism $\varphi^{2,F}$ and unit morphism $\varphi^{0,F}$. Its left comparator \({\tau^{l,F}}\) is compatible with internal composition and internal unit:
    \begin{align}
        {\tau^{l,F}_{X,Z}\circ F(\operatorname{comp}^l_{X,Y,Z}) \circ \varphi^{2,F}_{X\multimap Y,Y\multimap Z}} &\;=\; {{\operatorname{comp}^l_{F(X),F(Y),F(Z)}}\circ {(\tau^{l,F}_{X,Y}\otimes \tau^{l,F}_{Y,Z})}},\label{eq:comp and inhom compatibility}\\
        \tau^{l,F}_{X,X}\circ F(\operatorname{e}^X)\circ \varphi^{0,F} & \;=\; \operatorname{e}^{F(X)},\label{eq:unit and inhom compatibility}
    \end{align}
 for all \(X,Y,Z \in \cC\). Analogous equations hold for right closed monoidal categories.
\end{lemma}

By specializing Lemma~\ref{lem:compatibility composition and internal hom transfo} to \(X=Y=Z\), we find:
\begin{corollary}
Let \(F\colon \cC \rightarrow \cD\) be a lax monoidal functor between left closed monoidal categories. For any \(X\in \cC\), the corresponding component of the left comparator 
 \begin{equation*}
     \tau^{l,F}_{X,X}\colon\; F(E_{X})\eqdef F(X\multimap X)\,\longrightarrow\, F(X)\multimap F(X) \eqdef E_{F(X)}
\end{equation*}
is a morphism of algebras. An analogous statement holds in the right closed case.
\end{corollary}

\begin{lemma}\label{lemma:taul and monoidal nat transfos}
	Let \(F,G\colon \cC \rightarrow \cD\) be two lax monoidal functors between left closed monoidal categories, and let \(f \colon F\to G\) be a monoidal natural transformation. Then \(f\) is compatible with the left comparators \(\tau^{l,F}\) and \(\tau^{l,G}\), in the sense that
	 \begin{equation}
	\big(F(X)\multimap f_Y\big) \circ \tau^{l,F}_{X,Y} \;=\; \big(f_X\multimap G(Y)\big)\circ \tau^{l,G}_{X,Y} \circ f_{X\multimap Y},
	\end{equation}
	for all \(X,Y\in \cC\). An analogous equation holds for right closed monoidal categories.
\end{lemma}

We use the comparators to define what it means for a functor to preserve internal homs:

\begin{definition}
    Let \({F\colon \cC\rightarrow \cD}\) be a lax monoidal functor between monoidal categories.
    If \(\cC\) and \(\cD\) are left closed and the left comparator \(\tau^{l,F}\) of \(F\) is invertible, we say \(F\) is \emph{left closed}. Similarly, if \(\cC\) and \(\cD\) are right closed and \(\tau^{r,F}\) is invertible, we say \(F\) is \emph{right closed}.
    
    Finally, if \(\cC\) and \(\cD\) are closed and \(F\) is both left and right closed, we say \(F\) is \emph{closed}.
\end{definition}

\begin{remark}[Comparing definitions]
    What we call left (resp. right) closed monoidal categories are called right (resp. left) closed in \cite{BLV}. Up to this reversal, our definitions of left/right closed lax monoidal functors agree with \cite[\S 3.2]{BLV}, and our notion of closed strong monoidal functor matches the definition of a monoidal functor preserving left and right inner homs in \cite[\S I.4.3]{CatTann} and \cite[\S 3.1]{DualDoubleHopfAlgSch}.
\end{remark}

\begin{remark}[Coclosed functors]
    The notion of a (left or right) coclosed oplax monoidal functor between (left or right) coclosed monoidal categories is defined analogously.
\end{remark}

\subsection{Grothendieck--Verdier categories}\label{sec: GV-categories and closed monoidal categories}
A key notion for this paper is the following: 
\begin{definition}\label{def:GV-category}
Let \(\cC=(\cC,\otimes,1)\) be a monoidal category.

\begin{myitemize}
	\vspace{-0.3em}
    \item A \emph{dualizing object} of \(\cC\) is an object \(K\in \cC\) such that for every object \(Y\in \cC\), the functor \(\operatorname{Hom}_{\cC}(-\otimes Y,K)\) is representable by some object \(DY\in \cC\) and the induced contravariant functor \(D\) on \(\cC\) is an antiequivalence. We call \(D\) the \emph{duality functor} associated to \(K\), and denote its quasi-inverse by \(D^{-1}\).
    \item A \emph{Grothendieck--Verdier (GV) structure} on \(\cC\) is a choice of a dualizing object.\\ A \emph{Grothendieck--Verdier (GV) category} is a monoidal category with a GV-structure.
    \item If the monoidal unit \(1\) is a dualizing object of \(\cC\), we call \(\cC\) an \emph{r-category}.
\end{myitemize}
\end{definition}

\begin{remark}[Terminology]
We follow the terminology of \cite{BoDrinfeld}; GV-categories are also known as (non-symmetric) \(\ast\)-autonomous categories \cite{BarrStarAut,Ba95}.
\end{remark}

\begin{remark}[Closedness]\label{GV implies closed}
    Every GV-category is closed, with internal homs given by
    \begin{equation}\label{eq:right internal hom}
        Y\multimapinv X := D\big(X\otimes D^{-1}(Y)\big),
    \end{equation}
    \begin{equation}\label{eq:left internal hom}
        X\multimap Y := D^{-1}\big(D(Y)\otimes X\big),
    \end{equation}
    see, e.g., \cite[\S6]{Ba95} or \cite[Rem. 1.1]{BoDrinfeld}.
\end{remark}

We recall an alternative definition of a GV-category (Definition~\ref{def: GV via closed}), which is equivalent to Definition~\ref{def:GV-category} (Proposition~\ref{CompDualizing}). It will be used in the proof of one of our main results, the lifting theorem (Theorem~\ref{main thm}). We begin by presenting the necessary background:

\begin{remark}[A specific adjunction]\label{rem: double dualization morphism}
Let \((\cC,\otimes,1)\) be closed. For \(K\in \cC\), define the functors
\begin{align}
	D_K:=\;(K\multimapinv {?}) \colon \;\cC \rightarrow \cC^{\mathrm{op}}, \qquad \qquad D'_K:=\;({?}\multimap K) \colon\; \cC^{\mathrm{op}} \rightarrow \cC.
\end{align} 
Then \(D_K\) is left adjoint to \(D'_K\), with unit and counit given by the double-transposes of the identity morphisms on \({D_K(X)}\) and \({D'_K(X)}\) for any $X\in \cC$:
   \begin{equation}\label{DoubDualMor1}
	d^K_X\colon \:X \,\xrightarrow{\operatorname{coev}_X^{D_K(X)}}\, D_K(X)\multimap \big(D_K(X)\otimes X\big) \, \xrightarrow{D_K(X)\;\multimap\;{\higheroverline{\operatorname{ev}}^X_K}}\,D'_KD_K(X),
\end{equation}
\begin{equation}\label{DoubDualMor2}
	\widetilde{d}^K_X\colon \: X \,\xrightarrow{\higheroverline{\operatorname{coev}}^{D'_K(X)}_X}\, \big(X\otimes D'_K(X)\big)\multimapinv D'_K(X)\,\xrightarrow{\operatorname{ev}^X_K\;\multimapinv\;D'_K(X)}\,D_K D'_K(X). 
\end{equation}
\end{remark}

\smallskip

\begin{definition}\label{def: GV via closed}
Let \((\cC,\otimes,1)\) be closed. An object \(K\in \cC\) is a \emph{dualizer} if the unit \(d^K\) and counit \(\widetilde{d}^K\) are both invertible.
\end{definition}

This notion has been considered by several authors, e.g. \cite{BarrStarAut, BluteScott}, including in a non-symmetric setting \cite[\S4.8]{MelCatSem}.

\begin{remark}[Notation]
When the dualizer $K$ is clear, we write \(d\) and \(\widetilde{d}\) for the unit and counit, and $D$ and $D'$ for the associated duality functors, omitting superscripts and subscripts.
\end{remark}

A GV-category is equivalently a closed monoidal category with a dualizer:

\begin{prop}\label{CompDualizing}\emph{(Cf. \cite[\S6]{Ba95}.)}
    Let \((\cC,\otimes,1)\) be a monoidal category. An object \(K\in \cC\) is dualizing in the sense of Definition~\ref{def:GV-category} if and only if \(\cC\) is closed and \(K\) is a dualizer. In this case, the duality functor associated to $K$ is given by \(({K\multimapinv {?}})\) and its quasi-inverse by \(({?\multimap K})\).
\end{prop}

\begin{remark}[Terminology]
 	From now on, we identify dualizers with dualizing objects, write $D^{-1}$ or $D'$ for the inverse duality functor, and use `GV-category' and `closed monoidal category with a dualizer' interchangeably.
\end{remark}

\subsection{Grothendieck--Verdier functors} \label{sec: Functors between GV-categories}

We now introduce a notion of functor between GV-categories, which to the best of our knowledge is new. Let \(\cC\) and \(\cD\) be monoidal categories with distinguished objects \(K\in \cC\) and \(k\in \cD\). At this stage, \(K\) and \(k\) are not assumed to be dualizing.

\begin{definition}
	A \emph{form} on a functor \(F\colon \cC \to \cD\) is a morphism ${\upsilon^{0,F}\in \operatorname{Hom}_{\cD}(F(K),k)}$.
\end{definition}	

From now on, we assume that \(\cC\) and \(\cD\) are closed. By abuse of notation, we write
\begin{align*}
D' &:= (- \multimap K) \colon \cC^{\mathrm{op}} \to \cC,
\qquad
&D' := (- \multimap k) \colon \cD^{\mathrm{op}} \to \cD,\\
D &:= (K \multimapinv -) \colon \cC^{\mathrm{op}} \to \cC,
\qquad
&D := (k \multimapinv -) \colon \cD^{\mathrm{op}} \to \cD.
\end{align*}

\begin{definition}\label{def: duality transformations}
 Let \(F\colon \cC \to \cD\) be a lax monoidal functor with a form $\upsilon^{0,F}$. Its \emph{duality transformations} are the families of morphisms
	\begin{align*}
		& \xi^{l,F}_X\, \colon \,  FD'(X) \eqdef F(X\multimap K) \,\xlongrightarrow{\tau^{l,F}_{X,K}}\, F(X)\multimap F(K) \,\xlongrightarrow{{F(X)}\,\multimap\, {\upsilon^{0,F}}}\, F(X)\multimap k \eqdef D'F(X),\\
		& \xi^{r,F}_X\, \colon \, FD(X) \eqdef F(K\multimapinv X) \,\xlongrightarrow{\tau^{r,F}_{K,X}}\, F(K) \multimapinv F(X) \,\xlongrightarrow{\upsilon^{0,F}\,\multimapinv\, F(X)}\, k \multimapinv F(X) \eqdef DF(X),
	\end{align*} 
	natural in \(X\in \cC\). Here, $\tau^{l,F}$ and $\tau^{r,F}$ denote the comparators from Definition~\ref{ClosedFunctor1}.
\end{definition}

\begin{remark}[Bijective correspondences]
	Forms on a lax monoidal functor \(F\), natural transformations \(FD'\to D'F\), and natural transformations \(FD\to DF\) are all in bijective correspondence. For example, any form on \(F\) determines a natural transformation \(FD'\to D'F\), as described in Definition~2.35. Conversely, such a natural transformation gives rise to a form by evaluating at the monoidal unit \(1\), and then using the unit morphism \(\varphi^{0,F}\colon 1 \to F(1)\) and the isomorphisms \eqref{gamma unitor} from Remark \ref{rem:canonical isos}.
\end{remark}

We record compatibility results of the duality transformations \(\xi^{l,F}\) and \(\xi^{r,F}\) for a lax monoidal functor \(F\colon \cC \to \cD\) with form $\upsilon^{0,F}$:

\begin{lemma}\label{lemma:left-right duality transfos relations}
	 \(\xi^{l,F}\) and \(\xi^{r,F}\) are compatible with the unit \(d\) and counit \(\widetilde{d}\) from Remark~\ref{rem: double dualization morphism}:
	\begin{align}
		\xi^{l,F}_{D(X)} \circ F(d_X) &\;=\; D'(\xi^{r,F}_X) \circ d_{F(X)},\label{eq:left-right duality transfos relations 1}\\
		\xi^{r,F}_{D'(X)} \circ F(\widetilde{d}_X) &\;=\; D(\xi^{l,F}_X) \circ \widetilde{d}_{F(X)},\label{eq:left-right duality transfos relations 2}
	\end{align}
	for all \(X \in \cC\).
\end{lemma}

\begin{lemma}\label{lemma:duality transfos relations}
	\(\xi^{l,F}\) and \(\xi^{r,F}\) are compatible with the isomorphisms \(\beta\) and \(\higheroverline{\beta}\) \mbox{from Remark~\ref{rem:canonical isos}:}
		\begin{align}
		\beta_{F(X),F(Y),K} \circ D'(\varphi^{2,F}_{X,Y})\circ \xi^{l,F}_{X\ot Y} &\;=\;\big({F(Y) \multimap \xi^{l,F}_X}\big) \circ {\tau^{l,F}_{Y,D'(X)}}\circ F(\beta_{X,Y,K}),\label{eq: GV functor 1} \\
		\higheroverline{\beta}_{F(X),F(Y),K} \circ D(\varphi^{2,F}_{X,Y})\circ \xi^{r,F}_{X\ot Y} &\;=\;\big({\xi^{r,F}_Y\multimapinv F(X)}\big) \circ {\tau^{r,F}_{D(Y),X}}\circ F(\higheroverline{\beta}_{X,Y,K}),\label{eq: GV functor 2}
	\end{align}
	for all \(X,Y \in \cC\).
\end{lemma}

Proofs appear in Appendix~\ref{sec:GV-functors-appendix}.

\begin{definition}\label{def: Frobenius form}
A form $\upsilon^{0,F}$ on a lax monoidal functor \(F\colon \cC \to \cD\) is called \emph{Frobenius} if its associated duality transformations \(\xi^{l,F}\) and \(\xi^{r,F}\) are invertible.
\end{definition}

From now on, we assume that \(K\in \cC\) and \(k\in \cD\) are dualizing, so that \((\cC,K)\) and \((\cD,k)\) are GV-categories. Under this assumption, Lemma \ref{lemma:left-right duality transfos relations} implies that the natural transformation \(\xi^{l,F}\) is invertible if and only if \(\xi^{r,F}\) is invertible.

\begin{definition}\label{def: GV-functor}
A \emph{GV-functor} is a lax monoidal functor \(F\colon \cC \to \cD\) equipped with a Frobenius form $\upsilon^{0,F}\in \operatorname{Hom}_{\cD}(F(K),k)$.
\end{definition}

\begin{example}[Non-conservative GV-functor]\label{ex: non-con GV}
	In Example \ref{ex: duality functors are GV} and Section \ref{sec: Applications}, we encounter many examples of conservative GV-functors. Of course, not every GV-functor is conservative. For instance, if \(\cC\) is a symmetric GV-category \(\cC\), its double gluing \(\mathbf{G}(\cC)\) is again a symmetric GV-category, and the forgetful functor \(U\colon \mathbf{G}(\cC)\to \cC\) is a strict monoidal GV-functor with identity Frobenius form; cf.~\cite[Prop.~41]{HySch03}. This functor is not conservative in general.
\end{example}

\begin{example}[Functor without a Frobenius form]\label{ex: fun wo Frob}
As an example of a strict monoidal functor between GV-categories that does not, in general, admit a Frobenius form, let \(\cC\) be a symmetric GV-category with finite products, and consider the GV-structure on the product \(\cC\times \cC^{\operatorname{op}}\) described in \cite[Prop.~10]{HySch03}. Then the projection functor \(\cC\times \cC^{\operatorname{op}}\to \cC\) onto the first factor is strict monoidal but does not, in general, admit a Frobenius form.
\end{example}

\begin{definition}\label{def: morphism of GV-functors}
	Let \((F,\upsilon^{0,F})\) and \((G,\upsilon^{0,G})\) be two GV-functors \((\cC,K) \ra (\cD,k)\).
	A \emph{morphism of GV-functors} is a monoidal natural transformation \(f\colon F\ra G\) such that 
	\begin{align}
		\upsilon^{0,F}\;=\;\upsilon^{0,G}\circ f_{K}.\label{eq: morphism of GV-functors}
	\end{align}
\end{definition}

\begin{lemma}\label{lemma:pre-GV-morph-duality-transfo}
	Any morphism \(f\colon F \to G\) of GV-functors is compatible with the associated right duality transformations \(\xi^{r,F}\) and \(\xi^{r,G}\); explicitly, for all \(X\in \cC\),
	\begin{align}
		\xi^{r,F}_X &\;=\; D(f_X) \circ \xi^{r,G}_X \circ f_{D(X)}. \label{eq:xirF-relation-xirG}
	\end{align}
	An analogous equation holds for \(\xi^{l,F}\) and \(\xi^{l,G}\).
\end{lemma}

See Appendix~\ref{sec:GV-functors-appendix} for a proof. Lemma~\ref{lemma:pre-GV-morph-duality-transfo} has the following immediate consequence.
\begin{corollary}\label{cor: morph of GV-functros compat with duality transfo}
Any morphism of GV-functors \(f\colon(F,\upsilon^{0,F})\to(G,\upsilon^{0,G})\) is invertible. For each \(X\in \cC\), the inverse of \(f_X\colon\, F(X)\longrightarrow G(X)\) is given by
\begin{align}
	(\xi^{r,F}_{D'(X)})^{-1} \circ D(f_{D'(X)}) \circ \xi^{r,G}_{D'(X)},
\end{align}
where we have suppressed the natural isomorphism \(DD' \simeq \idC\).
\end{corollary}

We record the following:

\begin{prop}\label{prop: 2Cat GV}
GV-categories, together with GV-functors and their morphisms, form a \((2,1)\)-category \(\mathsf{GV}\).
\end{prop}

\vspace{0.1pt}

\section{Two \texorpdfstring{$2$}{2}-equivalences}\label{sec: two 2-equivalence}

\addtocontents{toc}{\protect\setcounter{tocdepth}{2}}
\subsection{Relation between Grothendieck--Verdier and linearly distributive structures}\label{sec: GV vs LDN}
In this subsection, we show that the $2$-category $\mathsf{GV}$ from Proposition~\ref{prop: 2Cat GV} is $2$-equivalent to a $2$-category $\mathsf{LDN}$, which we now introduce.
\subsubsection{\normalfont\textbf{Linearly distributive structures}}\label{sec: linearly distributive categories} We begin by recalling standard notions.

 \begin{definition}\label{def: LD cat} (\cite[Def. 2.1]{WDC}).
 A \emph{linearly distributive (LD) category} is a category \(\cC\) together with two monoidal structures \((\otimes,1)\) and \((\parLL,K)\) and two natural transformations \(\distl \colon {\otimes} {\,\circ\,} {(\operatorname{id_\cC} \times \parLL)} \rightarrow {\parLL} {\,\circ\,} {(\otimes \times \operatorname{id_\cC})}\) and \(\distr \colon {\otimes}{\,\circ\,}{(\parLL \times \,{\operatorname{id_\cC}})} \rightarrow {\parLL} {\,\circ\,}{(\operatorname{id_\cC}\times \,{\otimes)}}\), called \emph{distributors}, satisfying the coherence axioms~\eqref{eq:A1} to~\eqref{eq:A10} of Appendix~\ref{app: coherence diagrams}.
\end{definition}

\begin{definition}\label{def:LD-with-negation}(\cite[Def. 4.1]{WDC}).
Let \(\cC\) be an LD-category. An object \(X \in \cC\) is \emph{right LD-dualizable} if there exists an object \(\Eev\, X \in \cC\), called the \emph{right LD-dual}, together with \emph{evaluation} and \emph{coevaluation} morphisms \(\epsilon^X\colon X \otimes {\Eev\, X} \to K\) and \(\eta^X\colon 1 \to {\Eev\, X} \parLL X\), satisfying the snake equations~\eqref{firstzigzag} and~\eqref{secondzigzag} from Appendix ~\ref{app: coherence diagrams}. \emph{Left LD-duals} \(X \Eev{}\) are defined analogously, with evaluation \(\underline{\epsilon}^X \colon X \Eev{} \otimes X \to K\) and coevaluation \(\underline{\eta}^X\colon 1 \to {X} \parLL X \Eev{}\).
\end{definition}

\begin{definition}\label{def: LD-category with negation}(\cite[Def. 4.1]{WDC}).
 An LD-category is an \emph{LD-category with negation} if every object is both left and right LD-dualizable.
\end{definition}

Variants of the next notion appear in \cite{LD-functors, NoteOnFrobMonFunctors}; we adopt the conventions of \cite{DeS}.

\begin{definition}\label{def: Frobenius LD-functor}
    A \emph{Frobenius LD-functor} \(F\colon \cC \rightarrow \cD\) is a functor \(F\) equipped with both a lax \(\otimes\)-monoidal structure \((\varphi^{2,F},\varphi^{0,F})\) and an oplax \(\parLL\)-monoidal structure \((\upsilon^{2,F},\upsilon^{0,F})\) satisfying the Frobenius relations~\eqref{eq:F1} and~\eqref{eq:F2} from Appendix~\ref{app: coherence diagrams}.
\end{definition}

\begin{definition}
A Frobenius LD-functor is \emph{strong} if both its monoidal structures are strong. It is a \emph{Frobenius LD-equivalence} if it is strong and its underlying functor is an equivalence.
\end{definition}

Treating a monoidal category as an LD-category with identical monoidal structures, we recover a familiar notion:
\begin{definition}(\cite{NoteOnFrobMonFunctors}).\label{def: Frobenius monoidal functor}
	A \emph{Frobenius monoidal functor} is a Frobenius LD-functor between monoidal categories.
\end{definition}

\begin{definition}\label{def: morphism of Frob LD-functors}
A \emph{morphism of Frobenius LD-functors} \(F,G\colon \cC \ra \cD\) is a natural transformation \(f\colon F\ra G\) that is \(\otimes\)-monoidal and \(\parLL\)-opmonoidal.
\end{definition}

\begin{remark}[Invertibility of \(2\)-cells]
	Any morphism of Frobenius LD-functors between LD-categories with negation is invertible \cite[Prop. 2.38]{DeS}.
\end{remark}

\begin{prop}\label{prop: 2Cat LDN}{\normalfont (\cite[Prop. 2.23]{DeS}).}
LD-categories with negation, together with Frobenius LD-functors and their morphisms,
form a \((2,1)\)-category \(\mathsf{LDN}\).   
\end{prop}

\subsubsection{\normalfont\textbf{A $2$-equivalence between \(\mathsf{GV}\) and \(\mathsf{LDN}\)}}\label{sec: 2-equivalence gv to ldn}
\begin{remark}[Notation]\label{rem: forgetful 2-functor to MonCatl}
	Let \(\mathsf{MonCat}_g\) denote the \((2,1)\)-category of monoidal categories, lax monoidal functors and monoidal natural isomorphisms. There are forgetful \(2\)-functors from \( \mathsf{LDN}\) to \(\mathsf{MonCat}_g\) (forgetting the \(\parLL\)-monoidal structure) and from \(\mathsf{GV}\) to \(\mathsf{MonCat}_g\) (forgetting the dualizing object at the level of \(0\)-cells and the Frobenius form at the level of \(1\)-cells).
\end{remark}

The first main result of this paper is the following theorem.

\begin{theorem}\label{thm: GV equiv LDN}
    The \((2,1)\)-categories \(\mathsf{GV}\) and \(\mathsf{LDN}\) are {\(2\)-equivalent}. The \(2\)-equivalence can be chosen to strictly commute with the forgetful \(2\)-functors to \(\mathsf{MonCat}_g\).
\end{theorem}

\begin{remark}[Terminology]
By Theorem~\ref{thm: GV equiv LDN}, it is meaningful to speak of `the LD-structure of a GV-category' and `the GV-structure of an LD-category with negation'. After this theorem, the terms `GV-category' and `LD-category with negation' will be used interchangeably.
\end{remark}

\begin{remark}[The rigid case]\label{rem: rigid case of 2-equivalence}
The $2$-equivalence specializes to the rigid setting: a Frobenius monoidal functor $F\colon \cC\to \cD$ between rigid monoidal categories $(\cC,\otimes,1)$ and $(\cD,\otimes,1)$ is equivalently a lax monoidal functor equipped with a Frobenius form $F(1)\to 1$. Moreover, a monoidal natural transformation between such functors is a morphism of Frobenius monoidal functors if and only if it is compatible with the counit morphisms (Definition~\ref{def: morphism of GV-functors}).
\end{remark}

\subsubsection{\normalfont\textbf{From \(\mathsf{GV}\) to \(\mathsf{LDN}\)}}
We now turn to the proof of Theorem~\ref{thm: GV equiv LDN}. We begin by defining a \(2\)-functor \(\mathsf{GV}\to \mathsf{LDN}\):

\begin{construction}[\(0\)-cells]
It is well-known (cf. \cite[\S 4]{WDC} and \cite[\S 3]{BoDrinfeld}) that a GV\nolinebreak-category \((\cC,\ot,1,K)\) carries a second monoidal product
\begin{align}\label{def: par monoidal product}
	X \parLL Y \;:=\; D\big(D'(Y)\ot D'(X)\big),
\end{align} 
 with unit \(K\) and \(\parLL\)-associator given by
\begin{align}\label{def: par associator}
	\ap_{X,Y,Z}\;:=\;D\big(\ao_{D'(Z),D'(Y),D'(X)}\big),
\end{align}
for \(X,Y,Z\in \cC\). Here, we have suppressed the natural isomorphism \(D'D\simeq \idC\). The \(\parLL\)-unitors are defined similarly.
\end{construction}

\begin{remark}[Internal homs and $\parLL$-monoidal structure]\label{rem:closed mon in LD}
    The \(\parLL\)-monoidal product can be expressed via internal homs, e.g. \cite[\S 3.3]{fuchs2024grothendieckverdierdualitycategoriesbimodules}:
\begin{align}
    X \multimapinv D'(Y) &\enspace\eqabove{\eqref{eq:right internal hom}}\enspace D\big(D'(Y)\otimes D'(X)\big) \;\cong\; X \parLL Y, \label{eq:rightInHom} \\
    D(X) \multimap Y &\enspace \eqabove{\eqref{eq:left internal hom}}\enspace D'\big(D(Y)\otimes D(X)\big) \;\cong\; X \parLL Y, \label{eq:leftInHom}
\end{align}
for all \(X, Y \in \cC\). 
Under these identifications, the canonical isomorphism
\begin{align}
    \iota_{X,Y,Z} \colon (X \multimap Y) \multimapinv Z \;\xlongrightarrow{\simeq}\; X \multimap (Y \multimapinv Z),
\end{align}
equals the inverse \(\parLL\)-associator \(\ap^{-1}_{D'(X), Y, D(Z)}\).
\end{remark}

\begin{remark}[Distributors via internal homs]\label{rem:InHoms and distributors}
    For each component, the left (resp. right) internal hom of \(\cC\) is a lax module endofunctor \(\cC_{\cC}\rightarrow \cC_{\cC}\) (resp. \(_{\cC}\cC\rightarrow {_{\cC}\cC}\)) by doctrinal adjunction; see, e.g., \cite[Lemm. 2.4--2.5]{Shi}.
    Here, \(\cC_{\cC}\) (resp. \(_{\cC}\cC\)) denotes \(\cC\) regarded as a right (resp. left) module category over \((\cC,\otimes,1)\). Explicitly, for \(X,Y,Z\in \cC\), the lax actions \begin{align}
    \widetilde{\distl}_{X,Y,Z}\colon\;X\ot (Y\multimapinv Z) &\;\longrightarrow\; (X\ot Y)\multimapinv Z,\\
    \widetilde{\distr}_{X,Y,Z}\colon\;(X\multimap Y) \ot Z &\;\longrightarrow\; X\multimap (Y \ot Z),\end{align} are obtained by conjugating (inverse) associators:
    \begin{align}
        \widetilde{\distl}_{X,Y,Z}\;:=\;&\big((X\ot \higheroverline{\operatorname{ev}}^Z_Y)\multimapinv Z\big)\circ \big(\aoi_{X,Y\multimapinv Z,Z}\multimapinv Z\big) \circ \higheroverline{\operatorname{coev}}^Z_{X\ot (Y \multimapinv Z)},\label{eq: distl via inhoms}\\
        \widetilde{\distr}_{X,Y,Z}\;:=\;&\big(X\multimap ({\operatorname{ev}^X_Y} \ot Z)\big)\circ \big(X \multimap \ao_{X,X\multimap Y,Z}\big) \circ \operatorname{coev}^X_{(X \multimap Y)\ot Z}.\label{eq: distr via inhoms}
    \end{align}
    These lax actions yield the distributors \(\distl_{X,Y,Z}\) and \(\distr_{X,Y,Z}\) of the associated LD-category:
    \begin{align}
        X\otimes (Y\parLL Z)\congabove{\eqref{eq:rightInHom}} X\otimes \big(Y\multimapinv D'(Z)\big)\,\xlongrightarrow{\widetilde{\distl}_{X,Y,D'(Z)}}\,  (X\otimes Y)\multimapinv D'(Z) \congabove{\eqref{eq:rightInHom}} (X\otimes Y)\parLL Z, \label{eq: def left distributor}\\
        (X\parLL Y) \otimes Z \congabove{\eqref{eq:leftInHom}} \big(D(X)\multimap Y\big)\otimes Z \,\xlongrightarrow{\widetilde{\distr}_{D(X),Y,Z}}\, D(X)\multimap (Y\otimes Z) \congabove{\eqref{eq:leftInHom}} X \parLL {(Y \otimes Z)}.\label{eq: def right distributor}
    \end{align}
See \cite[\S 4]{fuchs2024grothendieckverdierdualitycategoriesbimodules}, \cite[\S 3]{fuchs2024grothendieckverdier}, and \cite[Thm. 2.45]{DeS} for further details.
\end{remark}

\begin{construction}[$1$-cells]\label{constr: GV to LDN}
Let \(F\colon \cC \to \cD\) be a GV-functor with multiplication morphism \(\varphi^{2,F}\), unit morphism \(\varphi^{0,F}\), and Frobenius form \(\upsilon^{0,F}\). Define its \(\parLL\)-comultiplication morphism
 \begin{gather}
 F(X\parLL Y) \eqdef FD\big(D'(Y) \ot D'(X)\big)\xrightarrow{\upsilon^{2,F}_{X,Y}} D\big(D'F(Y) \ot D'F(X)\big) \eqdef F(X)\parLL F(Y), \label{eq: def of comultiplication morphism of GV functor}
\end{gather}
through the formula
\begin{gather}\label{eq: GV to LDN}
\upsilon^{2,F}_{X,Y}\,:=\,D\big((\xi^{l,F}_Y)^{-1}\ot (\xi^{l,F}_X)^{-1}\big)\circ {D\big(\varphi^{2,F}_{D'(Y),D'(X)}\big)}\circ \xi^{r,F}_{D'(Y)\ot D'(X)},
\end{gather}
for all \(X,Y\in \cC\). The counit morphism of \(F\) is defined as the Frobenius form \(\upsilon^{0,F}.\)
\end{construction}

\begin{lemma}\label{lemma: F is Frob LD}
    The coherence morphisms \((\varphi^{2,F},\varphi^{0,F},\upsilon^{2,F},\upsilon^{0,F})\) of Construction~\ref{constr: GV to LDN} endow the GV-functor \(F\) with a Frobenius LD-structure.
\end{lemma}

\begin{remark}[$2$-cells]
	We let \(\mathsf{GV}\to \mathsf{LDN}\) act as the identity on \(2\)-cells. This is justified by the following lemma.
\end{remark}

\begin{lemma}\label{lemma: morphism of GV is morphism of Frob}
   Let \(\cC,\cD\) be  GV-categories. Let \((F,\upsilon^{0,F}),(G,\upsilon^{0,G}) \colon \cC\to \cD\) be GV-functors. Any morphism of GV-functors \(f\colon F\to G\) is also a morphism of Frobenius LD-functors. Here, \(F\) and \(G\) are endowed with the Frobenius LD-structures from Construction~\ref{constr: GV to LDN}.
\end{lemma}

\begin{prop}\label{prop: GV to LDN is functorial}
    The assignment \(\mathsf{GV}\to \mathsf{LDN}\) defines a \(2\)-functor.
\end{prop}

Lemmas~\ref{lemma: F is Frob LD} and~\ref{lemma: morphism of GV is morphism of Frob} and Proposition~\ref{prop: GV to LDN is functorial} are proved in Appendix~\ref{sec:GV-cat-and-LD}.

\subsubsection{\normalfont\textbf{From \(\mathsf{LDN}\) to \(\mathsf{GV}\)}}
We define a \(2\)-functor \(\mathsf{LDN} \to \mathsf{GV}\): 

\begin{construction}[$0$-cells]\label{constr: LDN to GV 0-cells}
	On \(0\)-cells, the \(2\)-functor forgets the \(\parLL\)-monoidal structure, retaining only the \(\parLL\)-monoidal unit, which serves as the dualizing object \cite[Thm. 2.45]{DeS}. Given an LD-category with negation, the duality functors $D$ and $D'$ of its associated GV-category are induced by the assignments $X\mapsto X\Eev{}$ and $X\mapsto {\Eev\, X}$, respectively; see \cite[Prop. 2.51]{DeS}.
\end{construction}

To define  \(\mathsf{LDN} \to \mathsf{GV}\) on \(1\)-cells, we use the following lemma, proved in Appendix~\ref{sec:LD-cat-to-GV}.

\begin{lemma}\label{prop:LD-categories closed}
	Let \(\cC\) be an LD-category. If every object \(X\in \cC\) has a right LD-dual \(D'(X)\in \cC\), then the monoidal category \((\cC,\otimes,1)\) is left closed, with internal hom 
	\begin{equation}\label{eq: def inhom in LD-cat}
		 X \multimap Y \,:=\, D'(X) \parLL {Y},
	\end{equation}
	for \(Y\in \cC\). Similarly, if every object has a left LD-dual, \((\cC,\otimes,1)\) is right closed. 
\end{lemma}

Let \(F\colon \cC \to \cD\) be a Frobenius LD-functor between LD-categories with negation.

\begin{remark}[Frobenius LD-functors preserve LD-duals]\label{rem:Frobenius LD-functors preserve duals}
As in \cite[\S 4.2]{DeS}, there exists a unique natural isomorphism of functors 
\begin{equation}\label{eq: chi lF def}
{\chi}^{l,F}\colon \; F \circ \rD \; \xlongrightarrow{\simeq} \; \rD \circ F,
\end{equation} 
such that, for every object \(X\in \cC\), the following identity holds
\begin{equation}\label{charact property duality transfo}
	{\big({\chi}^{l,F}_X\parLL F(X)\big)}\circ \widetilde{\eta}^{F(X)}\;=\;\eta^{F(X)}.    
\end{equation}
Here, the morphism \(\widetilde{\eta}^{F(X)}\) is defined by 
\begin{equation}
\widetilde{\eta}^{F(X)}\;:=\;\upsilon^{2,F}_{\rD (X),X}\circ F(\eta^X) \circ \varphi^{0,F},
\end{equation} 
where \(\varphi^{0,F}\colon 1\rightarrow F(1)\) is the unit morphism of the lax \(\otimes\)-monoidal functor \(F\), and 
\begin{equation}
\eta^X\colon 1\,\longrightarrow\,{D'(X)}\parLL X \quad \text{and} \quad \eta^{F(X)}\colon 1\,\longrightarrow\,{D'F(X)}\parLL F(X)
\end{equation}
are coevaluation morphisms in the LD-categories with negation \(\cC\) and \(\cD\); see Definition~\ref{def:LD-with-negation}.
\end{remark}

\begin{remark}[A candidate for the left comparator $\tau^{l,F}$]
Using the natural isomorphism \({\chi}^{l,F}\) from Equation~\eqref{eq: chi lF def}, we define morphisms
\begin{equation}
	\Upsilon^{l,F}_{X,Y}\colon \; F(X\multimap{Y}) \,\;\eqabove{\tiny~\eqref{eq: def inhom in LD-cat}}\;\, F\big({\rD (X)} \parLL {Y}\big) \; \longrightarrow \; \rD F(X)\parLL {F(Y)} \,\;\eqabove{\tiny~\eqref{eq: def inhom in LD-cat}}\;\, F(X)\multimap {F(Y)},
\end{equation}
natural in \(X,Y\in \cC\), by 
\begin{equation}\label{eq: candidate comparator}
\Upsilon^{l,F}_{X,Y} \; := \; \big({\chi}^{l,F}_X \parLL {F(Y)}\big)\circ\upsilon^{2,F}_{\rD (X), Y}.
\end{equation}
\end{remark}

Using \(\Upsilon^{l,F}\), we can relate the left comparator $\tau^{l,F}$ of $F$ to the comultiplication morphism $\upsilon^{2,F}$ of $F$, as stated in the following lemma.
\begin{prop}\label{prop:FrobLDfunctors are closed}
	Let \(F\colon \cC \to \cD\) be a Frobenius LD-functor between LD-categories where every object is right LD-dualizable. Then, for all \(X,Y\in \cC\),
	\begin{align} \label{eq: Upsilon=tau}
		\tau^{l,F}_{X,Y} \;&=\; \Upsilon^{l,F}_{X,Y}.
	\end{align}
	An analogous statement holds if every object in \(\cC\) and \(\cD\) is left LD-dualizable.
\end{prop}

See Appendix~\ref{sec:LD-cat-to-GV} for a graphical proof.

\begin{corollary}\label{cor:strong monoidal implies closed}
Let \(F\colon \cC \rightarrow \cD\) be a Frobenius LD-functor between LD-categories where every object is right LD-dualizable. Then, by Equation~\eqref{eq: candidate comparator}, the $\parLL$-comultiplication morphism \(\upsilon^{2,F}\colon F\circ\parLL \,\longrightarrow\, {\parLL} \circ {(F\times F)}\) is invertible if and only if the underlying lax \(\otimes\)-monoidal functor of \(F\) is left closed. An analogous statement holds if every object is left LD-dualizable.
\end{corollary}

\begin{construction}[\(1\)-cells]\label{constr: LDN to GV 1-cells}
	On $1$-cells, the assignment \(\mathsf{LDN}\to \mathsf{GV}\) forgets the \(\parLL\)-monoidal structure, retaining only the counit morphism. By Proposition~\ref{prop:FrobLDfunctors are closed} and the counitality of the underlying oplax \(\parLL\)-monoidal structure of the Frobenius LD-functor,
	this counit morphism is a Frobenius form.
\end{construction}

\begin{construction}[$2$-cells]\label{constr: LDN to GV 2-cells}
	We let \(\mathsf{LDN}\to \mathsf{GV}\) act as the identity on \(2\)-cells. This is justified by the following direct lemma. 
\end{construction}

\begin{lemma}\label{lemma: morphism of Frob LD is morphism of GV}
	Let \(\cC,\cD\) be  LD-categories with negation, and let \(F,G \colon \cC\to \cD\) be Frobenius LD-functors. Any morphism  \(f\colon F\to G\) of Frobenius LD-functors is also a morphism of GV-functors, with \(F\) and \(G\) equipped with the GV-functor structures from Construction~\ref{constr: LDN to GV 1-cells}.
\end{lemma}

The next two results follow immediately.

\begin{proposition}\label{prop: LDN to GV is functorial}
The assignment \(\mathsf{LDN}\to \mathsf{GV}\) defines a \(2\)-functor.
\end{proposition}

\begin{lemma}\label{lemma: GV to GV is equivalent to identity}
The composite \(2\)-functor \(\mathsf{GV}\to \mathsf{LDN} \to \mathsf{GV}\) is \(2\)-naturally isomorphic to the identity \(2\)-functor on \(\mathsf{GV}\).
\end{lemma}

The following lemma is proved in Appendix~\ref{sec:LD-cat-to-GV}.

\begin{lemma}\label{lemma: LDN to LDN is equivalent to identity}
	The composite \(2\)-functor \(\mathsf{LDN}\to \mathsf{GV} \to \mathsf{LDN}\) is \(2\)-naturally isomorphic to the identity \(2\)-functor on \(\mathsf{LDN}\).
\end{lemma}

We collect our results.

\begin{proof}[Proof of Theorem~\ref{thm: GV equiv LDN}]
The \(2\)-functors \(\mathsf{GV}\to \mathsf{LDN}\) and \(\mathsf{LDN}\to \mathsf{GV}\) from Propositions~\ref{prop: GV to LDN is functorial} and~\ref{prop: LDN to GV is functorial} clearly commute with the forgetful \(2\)-functors to \(\mathsf{MonCat}_g\). 
The claim now follows from Lemmas~\ref{lemma: GV to GV is equivalent to identity} and~\ref{lemma: LDN to LDN is equivalent to identity}.
\end{proof}

\begin{remark}[Frobenius algebras]
	Applying Theorem~\ref{thm: GV equiv LDN} to the terminal category $\ast$, we obtain an equivalence of hom-categories
	\begin{equation}
	\mathsf{GV}(\ast, \cC) \,\cong\, \mathsf{LDN}(\ast, \cC),
	\end{equation}
	for any GV-category \(\cC\). 
	Together with Proposition~\ref{prop:FrobLDfunctors are closed}, this recovers the characterization of LD-Frobenius algebras , the objects of \(\mathsf{LDN}(\ast, \cC)\), as given in \cite[Def.~3.22]{DeS}.
\end{remark}

\subsubsection{\normalfont\textbf{Pivotality}} Using the characterization of GV-categories as linearly distributive categories, we present more examples of GV-functors:
\begin{example}[Duality functors]\label{ex: duality functors are GV}
	Let $(\cC,\otimes,1,K)$ be a GV-category. Its duality functors 
	\begin{align}
		D,D'\colon\; (\cC,\otimes,1,K) \; \longrightarrow \; (\cC^{\operatorname{op}},\parLL^{\operatorname{rev}},K,1)
	\end{align}
	are GV-functors. Using Remark~\ref{rem:closed mon in LD}, we equip $D$ and $D'$ with strong monoidal structures $(\overline{\beta}_{X,Y,K},\higheroverline{\gamma}_K)$ and $({\beta}_{X,Y,K},{\gamma}_K)$, respectively (cf.~\cite[Prop.~2.54 and Cor.~2.55]{DeS}). The invertible Frobenius forms are $D(\gamma_K)^{-1}\circ \widetilde{d}_1$ and $D'(\higheroverline{\gamma}_K)^{-1}\circ d_1$, where $\gamma,\higheroverline{\gamma},\beta, \higheroverline{\beta}$ are as in Remark~\ref{rem:canonical isos}.
\end{example}
Using Example \ref{ex: duality functors are GV}, and for applications in Section \ref{sec: Applications}, we introduce the following notion:
\begin{definition}\label{def:pivotal structure}
	A \emph{pivotal structure} on a GV-category $(\cC,\otimes,1,K)$ is a morphism 
	\begin{equation*}
		\pi\colon \; D \; \longrightarrow \; D',
	\end{equation*}
	of GV-functors. Explicitly, $\pi$ is required to satisfy, for all $X,Y\in \cC$:
	\begin{align}
		\beta_{X,Y,K} \circ \pi_{X\otimes Y} & \;=\; (Y\multimap \pi_X)\circ \iota_{Y,K,X}\circ (\pi_Y\multimapinv X)\circ \higheroverline{\beta}_{X,Y,K}.\label{eq: PivotalStructure1}\\
		{\gamma_K}\circ \pi_1 & \;=\; \higheroverline{\gamma}_K.\label{eq: PivotalStructure2}\\
		D(\gamma_K) & \;=\; \iota_{K,1,K}\circ D'(\higheroverline{\gamma}_K)\circ \pi_K.\label{eq: PivotalStructure3}
	\end{align}
\end{definition}

\begin{remark}[Invertibility]
	By Corollary~\ref{cor: morph of GV-functros compat with duality transfo}, any pivotal structure is invertible.
\end{remark}

\begin{remark}[Comparing definitions]\label{rem: equivalent def of pivotality}
	By Theorem~\ref{thm: GV equiv LDN}, \(\pi\) is equivalently an isomorphism of the associated Frobenius LD-functors. Invoking~\cite[Thm. 4.2]{DeS}, a pivotal structure 
	as defined in Definition~\ref{def:pivotal structure} thus corresponds precisely to a pivotal structure in the sense of~\cite[Def. 5.1]{BoDrinfeld}. This can also be deduced directly from~\cite[Prop. 5.7]{BoDrinfeld}. 
\end{remark}

\subsection{Extension to the braided setting}\label{sec: pivotal and braided GV-categories}

\subsubsection{\normalfont\textbf{Braidings}}
A structure that a GV-category may carry is the following.
\vspace{-3pt}
\begin{definition}\label{def: braided GV}(\cite[\S 6]{BoDrinfeld}).
    A GV-category $\cC$ is \emph{braided} if its underlying monoidal category is equipped with a braiding $c_{X,Y}\colon X\otimes Y\xrightarrow{\simeq}Y\otimes X$.
\end{definition}
\vspace{-6pt}
\begin{remark}\label{rem: braided shorthand}(Shorthand from \cite[\S 6.2]{BoDrinfeld}).
	For a braided monoidal category \(\cC\) with braiding \(c\), and objects $X,Y\in \cC$, we will write $c^+_{X,Y}:=c_{X,Y}$ and $c^-_{X,Y}:=c^{-1}_{Y,X}$.
\end{remark}
\vspace{-6pt}
\begin{definition}\label{def: braided GV functor}
   A GV-functor \(F\colon \cC \to \cD\) between braided GV-categories is called \emph{braided} if its underlying lax monoidal functor is braided.
\end{definition}
\vspace{-5pt}
\begin{proposition}\label{prop: BrGV}
	Braided GV-categories, together with braided GV-functors and their morphisms, form a \((2,1)\)-category  \(\mathsf{BrGV}\).
\end{proposition}
We now recall the notion of a braiding for LD-categories.
\vspace{-5pt}
\begin{definition}\label{def:braided LD-category} (\cite[\S 4.11]{MelCatSem}).
A \emph{braided LD-category} is an LD-category $(\cC,\otimes,1,\parLL,K)$ equipped with a braiding $c$ for $(\otimes,1)$ and a braiding $\higheroverline{c}$ for $(\parLL,K)$, satisfying the following two hexagon relations for all $X,Y,Z\in \cC$:
\begin{align}
    (c_{X,Y}\parLL Z) \circ\distl_{X,Y,Z} \circ (X \otimes {\higheroverline{c}^{-1}_{Y,Z}}) &\;=\; \higheroverline{c}^{-1}_{Y\otimes X,Z} \circ \distr_{Z,Y,X}\circ c_{X,Z\parLL Y},\tag{H1}\label{eq: braiding and Frobenius relation 1}\\
    (X\parLL c_{Y,Z})\circ\distr_{X,Y,Z}\circ (\higheroverline{c}^{-1}_{X,Y}\otimes Z) &\;=\; \higheroverline{c}^{-1}_{X, Z\otimes Y} \circ \distl_{Z,Y,X}\circ c_{Y\parLL X,Z}.\tag{H2}\label{eq: braiding and Frobenius relation 2}
\end{align}
A \emph{symmetric LD-category} is an LD-category in which both braidings are symmetric. A \emph{braided (respectively, symmetric) LD-category with negation} is a braided (respectively, symmetric) LD-category whose underlying LD-category is an LD-category with negation.
\end{definition}
\vspace{-5pt}
\begin{example}
	Since strong monoidal functors are Frobenius monoidal \cite[Prop. 3]{NoteOnFrobMonFunctors}, a braided monoidal category is exactly a braided LD-category whose two braided monoidal structures coincide.
\end{example}
\vspace{-6pt}
\begin{example}
	Let $R$ be a finite-dimensional commutative $k$-algebra. The right exact tensor product $\otimes=\otimes_R$ of $R$-modules and the left exact cotensor product $\parLL=\otimes^R$ over the dual coalgebra $R^{\ast}=\operatorname{Hom}_{k}(R,k)$ endow the category $ {_R {\operatorname{Mod}}}$ of $R$-modules with an LD-structure; cf. \cite[Ex. 2.10]{DeS} and \cite[\S 6]{fuchs2024grothendieckverdierdualitycategoriesbimodules}. The symmetric braiding of $k$-vector spaces induces symmetric braidings on both $\otimes_R$ and $\otimes^R$, making $ {_R {\operatorname{Mod}}}$ a symmetric LD-category.
\end{example}

The hexagon relations can be understood as Frobenius relations:
\vspace{-4pt}
\begin{remark}[Frobenius relations]\label{rem: braided Frobenius relations}
By \cite[Ex. 2.5]{JoSt}, in any braided LD-category $\cC$, the \(\ot\)-braiding \(c\) endows the identity functor on \(\cC\) with a strong monoidal structure 
\begin{equation}
(\cC, \otimes^{\operatorname{rev}}, 1) \;\xlongrightarrow{\simeq} \; (\cC, \otimes, 1),
\end{equation}
while the inverse \(\parLL\)-braiding \(\higheroverline{c}^{-1}\) yields a strong opmonoidal structure \((\cC, \parLL^{\operatorname{rev}}, K) \to (\cC, \parLL, K)\).

The hexagon relations~\eqref{eq: braiding and Frobenius relation 1} and~\eqref{eq: braiding and Frobenius relation 2} then express precisely that these two structures make the identity functor on \(\cC\) a Frobenius LD-functor \begin{equation}\label{eq: Joyal-Street equivalence}
    I\colon\; \cC^{\operatorname{rev}}=(\cC,\otimes^{\operatorname{rev}},1,\parLL^{\operatorname{rev}},K) \;\xlongrightarrow{\simeq}\;(\cC,\otimes,1,\parLL,K)=\cC.
\end{equation}

In fact, \(I\) is a Frobenius LD-equivalence, whose inverse \(I^{-1}\colon \cC \to \cC^{\operatorname{rev}}\) is given by the identity functor on \(\cC\), equipped with the inverse \(\ot\)-braiding \(c^{-1}\) and the \(\parLL\)-braiding \(\higheroverline{c}\).
\end{remark}
\vspace{-5pt}
\begin{definition}\label{def: Joyal-Street equivalence}(Cf. \cite[\S 6.1]{BoDrinfeld}).
	The \emph{Joyal--Street equivalences} of a braided LD-category \(\cC\) are the squared Frobenius LD-equivalence \(J= I^{2} \colon \cC \rightarrow \cC\) and its inverse \(J^{-1}=I^{-2} \colon \cC \rightarrow \cC\).
\end{definition}
\vspace{-5pt}
\begin{definition}\label{def: braided Frobenius LD-functor}
A Frobenius LD-functor \(F\colon \cC \to \cD\) between braided LD-categories is called \emph{braided} if its underlying lax \(\ot\)-monoidal and oplax \(\parLL\)-monoidal functor are both braided.
\end{definition}
\vspace{-5pt}
\begin{proposition}\label{prop: BrLDN}
	Braided LD-categories with negation, together with braided Frobenius LD-functors and their morphisms, form a \((2,1)\)-category \(\mathsf{BrLDN}\).
\end{proposition}

\subsubsection{\normalfont\textbf{A $2$-equivalence between \(\mathsf{BrGV}\) and \(\mathsf{BrLDN}\)}}\label{sec: 2-equivalence BrGV to BrLDN}

\begin{remark}[Notation]
Let \(\mathsf{BrMonCat}_g\) denote the \((2,1)\)-category of braided monoidal categories, braided lax monoidal functors and monoidal natural isomorphisms. The \(2\)-functors from Remark~\ref{rem: forgetful 2-functor to MonCatl} lift to forgetful $2$-functors \(\mathsf{BrLDN} \to \mathsf{BrMonCat}_g\) and \(\mathsf{BrGV} \to \mathsf{BrMonCat}_g\).
\end{remark}

The second main result of the paper is the following theorem.

\begin{theorem}\label{thm: braided GV=braided LDneg}
The \((2,1)\)-categories \(\mathsf{BrGV}\) and \(\mathsf{BrLDN}\) are {\(2\)-equivalent}. The \(2\)-equivalence can be chosen to strictly commute with the forgetful \(2\)-functors to \(\mathsf{BrMonCat}_g\).
\end{theorem}

\begin{remark}[The rigid case]
Analogously to Theorem~\ref{thm: GV equiv LDN} (as explained in
Remark~\ref{rem: rigid case of 2-equivalence}), the $2$-equivalence of Theorem~\ref{thm: braided GV=braided LDneg} specializes to the rigid braided setting.
\end{remark}

\begin{remark}[The symmetric case]
	Similarly, the $2$-equivalence of Theorem~\ref{thm: braided GV=braided LDneg} specializes to the symmetric setting.
\end{remark}

First, we lift the \(2\)-functor \(\mathsf{GV}\to \mathsf{LDN}\) from Proposition~\ref{prop: GV to LDN is functorial} to a \(2\)-functor \(\mathsf{BrGV}\to \mathsf{BrLDN}\):
\begin{construction}[$0$-cells]\label{constr: BrGV to BrLDN}
For a GV-category \(\cC\) with braiding $c$, define the isomorphism 
\begin{equation}\label{def: c overline}
\higheroverline{c}^{\pm}_{X,Y}\colon\; X \parLL Y \,\eqdef\, D\big(D'(Y)\ot D'(X)\big)\;\xlongrightarrow{\simeq}\; D\big(D'(X)\ot D'(Y)\big) \,\eqdef\,Y \parLL X,
\end{equation}
natural in \(X,Y\in \cC\), by \(\higheroverline{c}^{\pm}_{X,Y}:=D(c^{\pm}_{D'(X),D'(Y)})\), using the shorthand from Remark~\ref{rem: braided shorthand}.
\end{construction}

For later use, we recall another natural isomorphism.

\begin{remark}[Braided closed monoidal categories]\label{rem:braided closed mon cats}
Let \(\cC\) be a closed monoidal category with braiding \(c\). By Yoneda's Lemma, the isomorphisms \(c^{\pm}\) induce isomorphisms 
	\begin{equation}\label{def: c tilde}
		\widetilde{c}^{\,\pm}_{X,Y}\colon\; (X \multimap Y) \;\xlongrightarrow{\simeq}\; (Y \multimapinv X),
	\end{equation}
	natural in \(X,Y\in \cC\). Explicitly, we define
	\allowdisplaybreaks
	\begin{align}
		\widetilde{c}^{\,\pm}_{X,Y}&:=(\operatorname{ev}^X_Y \multimapinv X)\circ (c^{\pm}_{X\multimap Y,X}\multimapinv X)\circ \higheroverline{\operatorname{coev}}^X_{X \multimap Y},
	\end{align}
	with inverse given by
	\begin{align}
		(\widetilde{c}^{\,\pm}_{X,Y})^{-1}&=(X \multimap \higheroverline{\operatorname{ev}}^X_Y)\circ (X\multimap c^{\mp}_{X,Y\multimapinv X})\circ \operatorname{coev}^X_{Y \multimapinv X}.
	\end{align}
\end{remark}

The following lemma is proved in Appendix~\ref{sec:Braided-GV-categories}.

\begin{lemma}\label{lemma: tildec ev and coev}
	Let \(\cC\) be a closed monoidal category with braiding \(c\). For \(\widetilde{c}^{\,\pm}\) as in Remark~\ref{rem:braided closed mon cats}
	\begin{align}
		{\higheroverline{\operatorname{coev}}^Y_X} \;&=\; {\widetilde{c}^{\,\pm}_{Y,X\ot Y}} \circ {(Y \multimap c^{\mp}_{Y,X})} \circ \operatorname{coev}^Y_X,\label{eq: ctilde compat coev}\\
		{\operatorname{ev}^{Y}_{X}}\;&=\;\higheroverline{\operatorname{ev}}^Y_X \circ c^{\mp}_{Y,X \multimapinv Y} \circ \big(Y \ot \widetilde{c}^{\,\pm}_{Y,X}\big),\label{eq: ctilde compat ev}
	\end{align}
	for all \(X,Y\in \cC\).
\end{lemma}

Our candidate for the $\parLL$-braiding from Construction~\ref{constr: BrGV to BrLDN} (Equation~\eqref{def: c overline}) can also be constructed from the isomorphism of Remark~\ref{rem:braided closed mon cats} (Equation~\eqref{def: c tilde}):
\begin{lemma}\label{lemma: relationship c overline and c tilde}
	Let \(\cC\) be a GV-category with braiding $c$. The following diagram
	\begin{equation}\label{dgm: relationship c overline and c tilde}
		\begin{tikzcd}
			{X \parLL Y}&&{Y \parLL X}\\
			{D(X)\multimap Y}&{Y\multimapinv D(X)}&{Y\multimapinv D'(X),}
			\arrow["\higheroverline{c}^{\pm}_{X,Y}"{yshift=1pt},"\simeq"'{yshift=-1pt}, from=1-1, to=1-3]
			\arrow["\widetilde{c}^{\,\mp}_{D(X),Y}"'{yshift=-3pt},"\simeq"{yshift=1pt}, from=2-1, to=2-2]
			\arrow["Y \multimapinv \widetilde{c}^{\,\pm}_{X,K}"'{yshift=-3pt},"\simeq"{yshift=1pt}, from=2-2, to=2-3]
			\arrow[dash,"\eqref{eq:rightInHom}"'{xshift=-2pt},"\simeq"{xshift=2pt}, from=1-3, to=2-3]
			\arrow[dash,"\eqref{eq:leftInHom}"'{xshift=-2pt},"\simeq"{xshift=2pt}, from=1-1, to=2-1]
		\end{tikzcd}
	\end{equation}
	commutes for all \(X,Y\in \cC\).
\end{lemma}

See Appendix~\ref{sec:Braided-GV-categories} for a proof. We return to the proof of Theorem~\ref{thm: braided GV=braided LDneg}:

\begin{proposition}\label{prop: braided GV is braided LDN}
Let $\cC$ be a GV-category with braiding $c$.
The natural isomorphism \(\higheroverline{c}=\higheroverline{c}^{\,+}\) from Construction~\ref{constr: BrGV to BrLDN} makes \(\cC\) a braided LD-category with negation.
\end{proposition}

This proposition, and the following one, are proved in Appendix~\ref{sec:Braided-GV-categories}. 

\begin{proposition}\label{prop: braided GVf is braided LDN f}
	Let \(F\colon \cC \to \cD\) be a braided GV-functor between braided GV-categories. The Frobenius LD-functor between the associated braided LD-categories with negation, induced by Lemma~\ref{lemma: F is Frob LD}, is braided.
\end{proposition}

\begin{remark}[Duality functors as braided Frobenius LD-equivalences]\label{rem: D and D' braided FLD functors}
	 Let \(\cC\) be a braided GV-category. The duality functors \(D,D'\colon (\cC,\otimes,1,K) \to (\cC^{\operatorname{op}},\parLL^{\operatorname{rev}},K,1)\) are GV-functors by Example~\ref{ex: duality functors are GV}. The \(\parLL\)-braiding \(\higheroverline{c}^{\, +}\) on \(\cC\) is defined so that \(D\) (and equivalently \(D'\)) becomes a braided GV-functor. By Proposition~\ref{prop: braided GVf is braided LDN f}, it follows that \(D\) and \(D'\) are braided LD-equivalences between the associated braided LD-categories with negation.
\end{remark}

In light of Propositions~\ref{prop: braided GV is braided LDN} and~\ref{prop: braided GVf is braided LDN f}, the following proposition is immediate.

\begin{proposition}\label{prop: BrGV BrLDN 2functor}
The \(2\)-functor \(\mathsf{GV}\to \mathsf{LDN}\) from Proposition~\ref{prop: GV to LDN is functorial} lifts to a \(2\)-functor \(\mathsf{BrGV}\to \mathsf{BrLDN}\).
\end{proposition}

The following results likewise follow immediately.

\begin{proposition}\label{prop: BrLDN to BrGV is functorial}
	The \(2\)-functor \(\mathsf{LDN}\to \mathsf{GV}\) from Proposition~\ref{prop: LDN to GV is functorial} lifts to a \(2\)-functor \(\mathsf{BrLDN}\to \mathsf{BrGV}\).
\end{proposition}

\begin{lemma}\label{lemma: BrGV to BrGV is equivalent to identity}
	The composite \(2\)-functor \(\mathsf{BrGV}\to \mathsf{BrLDN} \to \mathsf{BrGV}\) is \(2\)-naturally isomorphic to the identity \(2\)-functor on \(\mathsf{BrGV}\).
\end{lemma}

The following lemma is proved in Appendix~\ref{sec:Braided-GV-categories}.

\begin{lemma}\label{lemma: BrLDN to BrLDN is equivalent to identity}
	The composite \(2\)-functor \(\mathsf{BrLDN}\to \mathsf{BrGV} \to \mathsf{BrLDN}\) is \(2\)-naturally isomorphic to the identity \(2\)-functor on \(\mathsf{BrLDN}\).
\end{lemma}

We collect our results.

\begin{proof}[Proof of Theorem~\ref{thm: braided GV=braided LDneg}]
The \(2\)-functors \(\mathsf{BrGV}\to \mathsf{BrLDN}\) and \(\mathsf{BrLDN}\to \mathsf{BrGV}\) from Propositions~\ref{prop: BrGV BrLDN 2functor} and~\ref{prop: BrLDN to BrGV is functorial} clearly commute with the forgetful \(2\)-functors to \(\mathsf{BrMonCat}_g\). The claim then follows from Lemmas~\ref{lemma: BrGV to BrGV is equivalent to identity} and~\ref{lemma: BrLDN to BrLDN is equivalent to identity}.
\end{proof}

\begin{remark}[Commutative Frobenius algebras]
	Applying Theorem~\ref{thm: braided GV=braided LDneg} to the terminal category $\ast$, we obtain an equivalence of hom-categories
	\begin{equation}
		\mathsf{BrGV}(\ast, \cC) \,\cong\, \mathsf{BrLDN}(\ast, \cC),
	\end{equation}
	for any braided GV-category \(\cC\). This implies that the data of a commutative GV-algebra $A \in \cC$ equipped with a Frobenius form $A \to K$ in the sense of \cite[\S 3.2]{DeS} is equivalent to the data of an object $A\in \cC$ endowed with both a commutative GV-algebra and a cocommutative GV-coalgebra structure satisfying the LD-Frobenius relations from \cite[\S 3.1]{DeS}. 
\end{remark}

For applications of the lifting theorem (Section~\ref{sec: Applications}), we discuss how a braided GV-category $\cC$ comes with natural isomorphisms that \emph{nearly} endow it with a pivotal structure: Specializing the isomorphisms \(\widetilde{c}^{\,\pm}_{X,Y}\) from Equation~\eqref{def: c tilde} in Remark~\ref{rem:braided closed mon cats} to \(Y=K\) yields
\begin{align}\label{eq: identification of duals in braided GV-cats}
	\varphi^{\pm}_X \, := \, \widetilde{c}^{\,\pm}_{X,K}\colon\; D'(X)\; \xlongrightarrow{\simeq}\; D(X),
\end{align}
natural in \(X \in \cC\).

\begin{remark}[Comparing definitions]\label{rem: comparing duality transfos in braided GV}
The natural isomorphisms $\varphi^{\pm}$ coincide with those denoted by the same symbols in \cite[Lemma 6.8]{BoDrinfeld}.
\end{remark}

By Theorem~\ref{thm: braided GV=braided LDneg}, every braided GV-category is a braided LD-category. Hence we have the Joyal--Street equivalences \(J^{\pm 1}\colon \cC \rightarrow \cC\) as in Definition~\ref{def: Joyal-Street equivalence}. 
\begin{prop}\label{prop: Drinfeld iso is morphism of Frob LD}
    Let $\cC$ be a braided GV-category. The natural isomorphisms $\varphi^{\pm}$ are isomorphisms of Frobenius LD-functors \(D'\circ J^{\pm 1}\xrightarrow{\simeq} D\). 
\end{prop}

See Appendix~\ref{sec:Braided-GV-categories} for a proof.

\section{Algebraic structures for the lifting theorem}\label{sec: lifting theorem prep}
We review algebraic structures to which the lifting theorem \ref{main thm} will be applied.
\subsection{Algebras, bimodules, and local modules}\label{sec: algebras, bimodules, local modules}
In this subsection, we fix a closed monoidal category $\cC=(\cC,\otimes,1)$. Let $A=(A,\mu,\eta)$ be an algebra in \(\cC\). Assume that $\cC$ admits coequalizers. Since \(\cC\) is closed, the monoidal product $\otimes$ preserves coequalizers in each variable. Denote by \(_A\cC\) and \(\cC_A\) the categories of left and right \(A\)-modules in \(\cC\), and by ${_A}\cC_A$ the category of \(A\)-bimodules in \(\cC\).

\begin{remark}[The monoidal category of \(A\)-bimodules]\label{rem: monoidal category of bimodules}
As usual, the monoidal product of a right \(A\)-module $(M,r^M\colon M\otimes A \rightarrow M)\in \cC_A$ and a left \(A\)-module $(N,l^N\colon A\otimes N \rightarrow N)\in {_A\cC}$ is defined by the reflexive coequalizer
\begin{equation}
		\begin{tikzcd}
		{M\otimes_A N \;:=\; \operatorname{coeq}\!\big(M \otimes A \otimes N} && {M\otimes N\big).}
			\arrow[->,shift left=.6ex, "M\,\otimes \,l^N"{yshift=1.5pt}, from=1-1, to=1-3]
			\arrow[->,shift left=-.6ex, "r^M \,\otimes\, N"'{yshift=-1.5pt}, from=1-1, to=1-3]
		\end{tikzcd}
\end{equation}
This construction endows ${_A}\cC_A$ with a monoidal structure whose unit is $A$. Moreover, the canonical projection \(p_{M,N}\colon M\otimes N \twoheadrightarrow M \otimes_A N\), together with the unit \(\eta\colon 1 \rightarrow A\), equips the forgetful functor \(U_A\colon _A\cC_A \to \cC\) with a lax monoidal structure.
\end{remark}

We now discuss internal homs in ${_A}\cC_A$. To do so, we first establish a few technical lemmas. For $(M,r^M)\in {\cC_A}$, consider the morphism
\begin{equation}\label{eq: r underline morphism}
	\underline{r}^M\colon\; A \, \xrightarrow{\operatorname{coev^M_A}}\, M\multimap (M\otimes A) \,\xrightarrow{M\multimap\,r^M}\, M\multimap M\eqdef E_M,
\end{equation}
which, by Lemma~\ref{lemma:extranaturality of eval and coeval}, is extranatural in $M\in \cC.$ The next result follows from a straightforward computation.

\begin{lemma}\label{lem:morphism of algebras into internal hom}
    For every $(M,r^M)\in {\cC_A}$, the morphism $\underline{r}^M$ is a morphism of algebras.
\end{lemma}

Lemmas~\ref{lem:module over interal algebra} and~\ref{lem:morphism of algebras into internal hom} directly imply the following result.

\begin{lemma}\label{lemma: module structure on left internal hom}
	\begin{enumerate}[label=(\roman*)]
	\item For $(M,r^M)\in {\cC_A}$ and $N\in \cC$, the composite 
	\begin{equation}\label{eq: left action on internal hom}
		l^{M\multimap N}\;:=\;{\operatorname{comp}^l_{M,M,N}}\circ {\big(\underline{r}^M\otimes (M\multimap N)\big)}
	\end{equation}
	defines a left $A$-module structure on the left internal hom $M\multimap N$. 
	
	\item Likewise, for $M\in \cC$ and $(N,r^N)\in {\cC_A}$, the composite
	\begin{equation}\label{eq: right action on internal hom}
	r^{M\multimap N}\;:=\; {\operatorname{comp}^l_{M,N,N}} \circ \big((M\multimap N)\otimes \underline{r}^N\big)
	\end{equation}
	defines a right $A$-module structure on $M\multimap N$.
	\item For $(M,r^M),(N,r^N)\in {\cC_A}$, the above \(A\)-actions \(l^{M\multimap N}\) and \(r^{M\multimap N}\) make $M\multimap N$ an \(A\)-bimodule.
\end{enumerate}	
Analogous statements hold for right internal homs.
\end{lemma}

One may rewrite the action in Equation~\eqref{eq: left action on internal hom} as follows; for a proof see Appendix~\ref{sec:Bimodules-GV-categories}.

\begin{lemma}\label{lemma: rewrite inHom action}
 For $(M,r^M)\in {\cC_A}$ and $N\in \cC$, we have
\begin{align}\label{eq: rewrite inHom action}
    l^{M\multimap N}\;=\; {\operatorname{ev}^A_{M\multimap N}} \circ {\big(A\otimes \beta_{M,A,N}\big)} \circ {\big(A\otimes (r^M\multimap N)\big)}.
\end{align}
An analogous formula holds for right internal homs.
\end{lemma}

For later use, we record the following technical lemma; for a proof see again Appendix~\ref{sec:Bimodules-GV-categories}.

\begin{lemma}\label{lemma: rewrite coreflexive pair}
	Let \((L,r^L)\in\cC_A\) and \(M,N\in \cC\). Then, omitting associators, we have\vspace{4pt}
	\begin{align}
		{\big((A\otimes M)\multimap l^{L\multimap N}\big)} \circ {\underline{A\otimes}_{M,L\multimap N}}\circ \beta_{L,M,N} &\,=\, \beta_{L,A\otimes M,N} \circ \big((r^L\otimes M)\multimap N\big).
	\end{align} 

	Here, \(\underline{A\otimes }_{M,L\multimap N}\colon\,M\multimap (L\multimap N) \,\longrightarrow\, (A\otimes M)\multimap \big(A\otimes (L \multimap N)\big)\) is the morphism of Remark~\ref{rem: internal hom tensorality}, and \(l^{L\multimap N}\) is the left $A$-action from Equation~\eqref{eq: left action on internal hom} in Lemma~\ref{lemma: module structure on left internal hom}.
	
	An analogous formula holds for right internal homs.
\end{lemma}

We are finally ready to describe the internal homs in \(_A\cC_A\). We now additionally assume that $\cC$ admits equalizers.
\begin{remark}[Internal homs in ${_A}\cC_A$]\label{rem: inhoms in abimod}
	For \(M,N\in {_A\cC_A}\), consider the equalizer
	\begin{equation}
			\begin{tikzcd}
				{M\multimap_A N \;:=\; \operatorname{eq}\!\big(M\multimap N} &&&& {(A\otimes M)\multimap N\big).}
				\arrow[->,shift left=.7ex, "l^M\,\multimap\, N"{yshift=1.5pt}, from=1-1, to=1-5]
				\arrow[->,shift left=-.7ex, "{\big((A\,\otimes \,M)\,\multimap \,l^N\big)}\,\circ\,{\underline{A\otimes }_{M,N}}"'{yshift=-1.5pt}, from=1-1, to=1-5]
			\end{tikzcd}
	\end{equation}
	Here, \(\underline{A\otimes }_{M,N}\colon\,M\multimap N \,\longrightarrow\, (A\otimes M)\multimap (A\otimes N)\) is defined in Remark~\ref{rem: internal hom tensorality}.
The \(A\)-bimodule structure on \(M\multimap N\) supplied by Lemma~\ref{lemma: module structure on left internal hom}, induced by the right \(A\)-actions \(r^M\) and \(r^N\), restricts to the equalizer \(M\multimap_A N\); see \cite[Def. 3.2]{ShiYa}. Candidates for right internal homs in \(_A\cC_A\) are defined analogously.
\end{remark}

\begin{prop}\emph{(\cite[Cor. 3.4]{ShiYa}).}\label{lemma: bimodules closed}
    Let \(A\) be an algebra in a closed monoidal category \(\cC\). If the category $\cC$ admits equalizers and coequalizers, the monoidal category $_A\cC_A$ of $A$-bimodules is closed. Internal homs are given as in Remark~\ref{rem: inhoms in abimod}.
\end{prop}

\begin{remark}[Evaluations and coevaluations in ${_A}\cC_A$]\label{rem: coeval in A-bimod}
In the setting of Proposition~\ref{lemma: bimodules closed}, the (co)evaluations in $_A\cC_A$ are induced by those in \(\cC\). For \(M,N\in {_A\cC_A}\), the morphisms 
\begin{align}
	 {_A{\operatorname{ev}}^M_N} \colon \; &M \otimes_A (M \multimap_A N) \;\longrightarrow\; N,\\
	 {_A{\operatorname{coev}}^M_N} \colon \; &N \;\longrightarrow\; M \multimap_A (M\otimes_A N),
	\end{align}
are characterized uniquely by the equations
\begin{align}
{_A{\operatorname{ev}}^M_N} \circ p_{M,M\multimap_A N} &\;=\; {\operatorname{ev}^M_N} \circ (M \otimes i_{M,N}),\label{eq: AevM,N}\\
{i_{M,M\otimes_A N}} \circ {_A{\operatorname{coev}}^M_N} &\;=\; (M\multimap p_{M,N}) \circ {\operatorname{coev}^M_N},\label{eq: AcoevM,N} 
\end{align}
where \(i_{M,N}\colon\, M\multimap_A N \,\hookrightarrow\, M \multimap N\) and \(p_{M,N}\colon M\otimes N \,\twoheadrightarrow\, M \otimes_A N\) denote the canonical monomorphism and epimorphism associated to the equalizer and coequalizer, respectively.
\end{remark}

\begin{remark}[Comparator for \(U_A\)]\label{rem: comparator for U}
By Equation~\eqref{eq: AevM,N}, the comparator 
\(\tau^{l,U}_{M,N}\) (see Definition~\ref{def: closed monoidal functor}) for the lax monoidal forgetful functor \(U_A\colon {_A\cC_A} \to \cC\) (see Remark~\ref{rem: monoidal category of bimodules}) coincides with the canonical monomorphism \(i_{M,N}\colon (M \multimap_A N) \,\hookrightarrow\, (M\multimap N)\).
\end{remark}

\begin{prop}\emph{(\cite[Cor. 3.7]{ShiYa}).}\label{lemma: modules closed}
	Let \(A\) be a commutative algebra in a braided closed monoidal category \(\cC\). If the category $\cC$ admits equalizers and coequalizers, the monoidal category $_A\cC$ of left $A$-modules is closed. Internal homs are again those in Remark~\ref{rem: inhoms in abimod}.
\end{prop}

\begin{remark}[Local modules]\label{rem: local modules}
Let $A$ be a commutative algebra in a braided monoidal category $\cC$. A left $A$-module $(M,l^M)$ is called \emph{local} if $l^M=l^M \circ c_{M,A} \circ c_{A,M}$. By \cite[Cor. 3.7]{ShiYa}, if $\cC$ is closed and admits equalizers and coequalizers, the full monoidal subcategory of local $A$-modules ${_A\cC}^{\operatorname{loc}}\subseteq {_A\cC}$  is closed with the same internal homs as those in Remark~\ref{rem: inhoms in abimod}.
\end{remark}

From now on, let \(K\in \cC\) be a dualizing object in the monoidal category $(\cC,\otimes,1)$. 
\begin{definition}(\cite[Def. 4.2]{fuchs2024grothendieckverdier}).
	An algebra in $(\cC,\otimes,1)$ is called a \emph{GV-algebra}, while a coalgebra in the monoidal category $(\cC,\parLL,K)$ is called a \emph{GV-coalgebra}.
\end{definition}

Recall from Example~\ref{ex: duality functors are GV} that the categorical equivalences 
\begin{equation*}
D,D'\colon\,(\cC, \otimes, 1) \,\xlongrightarrow{\simeq}\, (\cC^{\operatorname{op}}, \parLL^{\operatorname{rev}}, K)
\end{equation*}
carry strong monoidal structures. Hence the next statement follows directly.

\begin{lemma}\emph{(\cite[Lem. 4.3]{fuchs2024grothendieckverdier}).}\label{lemma: GV-algebras-GV-coalgebras}
  In a GV-category \(\cC\), GV-algebras and GV-coalgebras are in bijection under either one of the duality functors \(D\) or \(D'\).
\end{lemma}

\begin{remark}[Cocommutativity]
    Let \(\cC\) be a braided GV-category. By Remark~\ref{rem: D and D' braided FLD functors}, the bijections from Lemma~\ref{lemma: GV-algebras-GV-coalgebras} restrict to the classes of commutative GV-algebras and cocommutative GV-coalgebras in \(\cC\).
\end{remark}

\begin{remark}[Coalgebra structure and internal homs]\label{rem: coalgebra and in homs}
    By Remark~\ref{rem:closed mon in LD} and Example~\ref{ex: duality functors are GV}, the comultiplication $\Delta$ and counit $\epsilon$ of the GV-coalgebra $D(A)$ are given explicitly by
\begin{align}
    \Delta&\colon\; D(A)\,\xrightarrow{D(\mu)}\,D(A\otimes A)\,\xrightarrow{\higheroverline{\beta}_{A,A,K}}\,(K\multimapinv A)\multimapinv A\;\;\eqabove{\tiny~\eqref{eq:rightInHom}}\;\;D(A)\parLL D(A),\label{def:Comultiplication of DA}\\
    \epsilon&\colon\; D(A)\,\xrightarrow{D(\eta)}\,D(1) \xrightarrow{\higheroverline{\gamma}_K}\,K.
\end{align}
The comultiplication and counit of the GV–coalgebra \(D'(A)\) are described by analogous formulas. Consequently, a morphism of GV-coalgebras $f\colon D(A)\to D'(A)$ is a morphism of the underlying objects of $\cC$ satisfying
\begin{align}
    \beta_{A,A,K}\circ D'(\mu) \circ f &\,=\, (A\multimap f)\circ \iota_{A,K,A}\circ (f\multimapinv A) \circ \overline{\beta}_{A,A,K}\circ D(\mu),\label{def:f is comultiplicative}\\
    \gamma_K \circ D'(\eta)\circ f &\,=\, \higheroverline{\gamma}_K \circ D(\eta).\label{def:f is counital}
\end{align}
\end{remark}

\smallskip

Next, in preparation for applications of the lifting theorem (Section \ref{sec: Applications}), we investigate when, for a GV-algebra $A$ in \(\cC\), the GV-coalgebras \(D(A)\) and \(D'(A)\) admit \(A\)-bimodule structures. By Lemma~\ref{lemma: module structure on left internal hom}, the regular $A$-action on $A$ induces a left $A$-action on $D'(A)$ and, analogously, a right $A$-action on $D(A)$. The following result is proved in Appendix~\ref{sec:Bimodules-GV-categories}.
\begin{prop}\label{prop: bimodule structure on dual of algebra}
Any \(\parLL\)-comultiplicative isomorphism $f\colon D'(A)\xrightarrow{\simeq} D(A)$ equips $D(A)$ with an $A$-bimodule structure.
\end{prop}

Let us give examples of such \(\parLL\)-comultiplicative isomorphisms. The following lemma follows directly from Remark~\ref{rem: equivalent def of pivotality}.

\begin{lemma}\label{lemma: coalgebra iso via pivotality}
Let $A$ be a GV-algebra in a GV-category $\cC$, and let $\pi\colon D\xrightarrow{\simeq} D'$ be a pivotal structure on \(\cC\). The component $\pi_A\colon D(A)\xrightarrow{\simeq} D'(A)$ is an isomorphism of GV-coalgebras.  
\end{lemma}
The following lemma is proved in Appendix~\ref{sec:Bimodules-GV-categories}.

\begin{lemma}\label{lemma: coalgebra iso via braiding}
Let \(A\) be a \emph{commutative} GV-algebra in a \emph{braided} GV-category \(\cC\). Recall the isomorphisms $\varphi^{\pm}_A\colon D'(A) \xrightarrow{\simeq} D(A)$ from Equation~\eqref{eq: identification of duals in braided GV-cats}. 
\begin{enumerate}[label=(\roman*)]
	\item $\varphi^{\pm}_A$ are isomorphisms of GV-coalgebras.
	\item $\varphi^{\pm}_A$ are isomorphisms of right $A$-modules. Here, \(D'(A)\) carries the right $A$-action $\newline l^{D'(A)}\circ c^{\pm}_{D'(A),A}$, where $l^{D'(A)}$ is the left $A$-action from Equation~\eqref{eq: left action on internal hom} in Lemma~\ref{lemma: module structure on left internal hom}, induced by the multiplication of $A$. The object \(D(A)\) carries the right $A$-action from the right internal hom version of Lemma~\ref{lemma: module structure on left internal hom}.
\end{enumerate}
\end{lemma}

\subsection{Hopf monads and Hopf algebroids}\label{sec: Hopf monads and Hopf algebroids}

We now turn to another class of algebraic structures to which we will apply the lifting theorem: Hopf monads and Hopf algebroids. The following facts and examples are all well-known.
\subsubsection{\normalfont\textbf{Hopf monads}}
As in \cite{MoerMonads, BVHopf, BLV}, the notion of a bialgebra generalizes to monoidal categories that are not necessarily braided:

\begin{definition}
    A \emph{bimonad} on a monoidal category $\cC$ is a monad $(T,\mu,\eta)$ on $\cC$, together with an opmonoidal structure $(\upsilon^2,\upsilon^0)$ on the functor $T$, with respect to which the monad multiplication $\mu\colon T^2\rightarrow T$ and unit $\eta\colon \operatorname{id}_{\cC}\rightarrow T$ are opmonoidal natural transformations. 
\end{definition}

\begin{definition}(\cite[\S3]{MoerMonads}).
	A bimonad on a braided monoidal category $\cC$ is called \emph{cocommutative} if its underlying oplax monoidal functor is braided.
\end{definition}

Recall the category of modules over a monad (also known as the Eilenberg--Moore category):

\begin{definition}\label{def:eilenberg-moore category}
    Let $(T,\mu,\eta)$ be a monad on a category $\cC.$ A \emph{$T$-module} consists of an object $M\in \cC$ and a morphism $\omega\colon\, T(M) \to M$ such that $\omega\circ T(\omega)\,=\,\omega\circ\mu_M$ and $\omega\circ\eta_M\,=\,\operatorname{id}_M.$\linebreak A \emph{morphism of $T$-modules} $(M,\omega) \to (N,\gamma)$ is a morphism $f\in \operatorname{Hom}_{\cC}(M, N)$ such that\linebreak $f\circ \omega\,=\,\gamma\circ T(f).$ The resulting category of $T$-modules is denoted $\cC^T$.
\end{definition}

The following remark justifies the term ‘bimonad’:

\begin{remark}[Bimonads lift monoidal structures]\label{monoidal structure on CT}
	Let $T$ be a bimonad on a monoidal category $\cC=(\cC,\otimes,1)$. Bimonad structures on $T$ correspond bijectively to monoidal structures on $\cC^T$ making the forgetful functor $U_T\colon \cC^T\to \cC$ strict monoidal; see \cite[Thm. 7.1]{MoerMonads} or \cite[Thm. 2.3]{BVHopf}. For a bimonad $T$, the category $\cC^T$ is monoidal with
	\[
	(M,\omega)\otimes (N,\gamma) := (M\otimes N, (\omega\otimes\gamma)\circ \upsilon^{2,T}_{M,N})
	\qquad
	\text{and}
	\qquad 
	1_{\cC^T} := (1, \upsilon^{0,T}).
	\]
\end{remark}

\begin{definition}(\cite[\S2.6]{BLV}).
    Let $T$ be a bimonad on a monoidal category $(\cC,\otimes,1).$ The \emph{left fusion operator} of $T$ is the natural transformation 
    \begin{equation*}
    H^l\colon \;T\circ \otimes \circ ({T}\times {\operatorname{id}_{\cC}})\,\Longrightarrow \,\otimes \circ (T\times T),
	\end{equation*} 
    defined, for $X,Y\in \cC$, by \begin{equation*}
        H^l_{X,Y}\,:=\,(\mu_X \otimes T(Y))\circ \upsilon^2_{T(X),Y}.
    \end{equation*} The \emph{right fusion operator} is the left fusion operator in the reversed monoidal category $\cC^{\text{rev}}.$
\end{definition}

A bialgebra in a braided monoidal category admits an antipode (is a Hopf algebra) if and only if its left (or equivalently right) fusion morphism is invertible; see \cite[Prop. 10]{ToHeFlo} for a discussion using string diagrams. This characterization motivates the following terminology:
\begin{definition}\label{def: Hopf monad}(\cite[\S2.6]{BLV}).
    A bimonad is \emph{left} (respectively, \emph{right}) \emph{Hopf} if its left (respectively, right) fusion operator is invertible, and \emph{Hopf} if it is both left and right Hopf.
\end{definition}

\begin{remark}[Left/right distinction]\label{rem: Left vs. right Hopf}
	Since $\cC=(\cC,\otimes,1)$ and its reverse $\cC^{\text{rev}}=(\cC,\otimes^{\text{rev}},1)$ are a priori not identified in our setting, left and right fusion operators must be distinguished. Yet, for a cocommutative bimonad on a braided monoidal category, being left Hopf is clearly equivalent to being right Hopf.
\end{remark}

For later use, we recall a result on bimonads in closed monoidal categories.

\begin{theorem}\label{Thm: Hopf monad closed}\emph{(\cite[Thm. 3.6]{BLV}).}
Let $T$ be a bimonad on a left (resp. right) closed monoidal category $\cC$. The following assertions are equivalent:
\begin{enumerate}[label=(\roman*)]
    \item The bimonad $T$ is a left (resp. right) Hopf monad.
    \item The monoidal category $\cC^T$ of $T$-modules is left (resp. right) closed, and the forgetful functor $U_T\colon \cC^T \rightarrow \cC$ is left (resp. right) closed.
\end{enumerate}
\end{theorem}

\subsubsection{\normalfont\textbf{Hopf algebroids}}
In this subsection we discuss a particular class of Hopf monads. Throughout, let $R$ be an algebra over a commutative ring $k$. We follow \cite[\S7.1]{BLV} and adopt the following notation: 

\begin{remark}[Notation]\label{notation: enveloping algebra}
We denote by $_R {\operatorname{Mod}}_R$ the category of $R$-bimodules. This category can be identified with the category of left (respectively, right) $R^e$-modules $_{R^e} {\operatorname{Mod}}$ (respectively, ${\operatorname{Mod}}_{R^e}$), where ${R^e=R\otimes_k R^{\operatorname{op}}}$ is the enveloping algebra of $R$. The (right exact) tensor product of $R$-bimodules induces a monoidal product $\boxtimes$ on $_{R^e}{\operatorname{Mod}}$ (respectively, ${\operatorname{Mod}}_{R^e}$).  
\end{remark}

\begin{definition}\label{def:bialgebroid}
   A left (respectively, right) \emph{$R$-bialgebroid} is a $k$-linear bimonad on the monoidal category of left (respectively, right) $R^e$-modules that admits a right adjoint.
\end{definition}

\begin{remark}[Left/right distinction]
The monoidal categories of left and right $R^e$-modules are equivalent. Therefore, we restrict attention to left $R$-bialgebroids and refer to them simply as \emph{$R$-bialgebroids}.
\end{remark}

\begin{remark}[Commutative base algebra]\label{rem: Bialgebroid over commutative base}
If $R$ is commutative, a $k$-linear bimonad on the monoidal category of left $R$-modules that admits a right adjoint is also called an \emph{$R$-bialgebroid}. We will indicate when the term is used in this sense.
\end{remark}

\begin{definition}\label{def: cocommutative bialgebroid}
For commutative $R$, an $R$-bialgebroid as in Remark~\ref{rem: Bialgebroid over commutative base} is called \emph{cocommutative} if its underlying bimonad is cocommutative.
\end{definition}

\begin{definition}\label{def:hopf algebroid}
    An $R$-bialgebroid is called \emph{left} (respectively, \emph{right}) \emph{$R$-Hopf} if it is left (respectively, right) Hopf, and an \emph{$R$-Hopf algebroid} if it is both left and right Hopf.
\end{definition}

To recall a more explicit algebraic characterization of $R$-bialgebroids and $R$-Hopf algebroids, we need the following definitions:

\begin{definition}
An \emph{$R$-ring} is an algebra in the monoidal category $_{R}{\operatorname{Mod}}_{R}$ of $R$-bimodules. A \emph{morphism of $R$-rings} from $(A,\mu_A,\eta_A)$ to $(B,\mu_B,\eta_B)$ is a morphism $f\colon A\rightarrow B$ of $R$-bimodules such that ${\mu_B\circ(f\otimes_R f)=f\circ\mu_A}$ and $f\circ\eta_A=\eta_B$. This defines the category $\operatorname{Alg}(_R {\operatorname{Mod}}_R)$ of $R$-rings.
\end{definition}

\begin{definition}\label{module over r-ring}(\cite[Def. 2.3]{HandbookBoehm}).
The category of \emph{(left) modules} over an $R$-ring $A$ is defined as the category of modules over the monad $T_A:=\,A\otimes_{R}{?}$ on $_{R}{\operatorname{Mod}}$.
\end{definition}

\begin{remark}[Caveat]
Modules over an $R$-ring $A$ in the sense of Definition~\ref{module over r-ring} are left $R$-modules only. In particular, they are not defined as modules over the algebra $A$ in $_R {\operatorname{Mod}}_R$.
\end{remark}

The following lemma reshuffles algebraic data; the proof is left as an exercise.
\begin{lemma}\label{isoOverRing}\emph{(\cite[Lem. 2.2]{HandbookBoehm}).}
The category $\operatorname{Alg}(_R {\operatorname{Mod}}_R)$ of $R$-rings is isomorphic to the coslice category $R\downarrow \operatorname{Alg}(_{k}{\operatorname{Mod}})$ of the category of $k$-algebras $\operatorname{Alg}(_{k}{\operatorname{Mod}})$ under $R$.
\end{lemma}

The next notion has a long history, e.g. \cite{SweedlerGroups,TakeuchiGroups,HandbookBoehm}. We follow \cite[\S7.1]{BLV} in our presentation of this notion:
\begin{definition}\label{takeuchi bialgebra}
Let $R$ be an algebra over a commutative ring $k$. A \emph{(left) $\times_R$-bialgebra} consists of the data $(B,s,t,\Delta,\epsilon)$, where:
\begin{itemize}
    \item $B=(B,\mu,\eta)$ is a $k$-algebra with multiplication $\mu\colon \,B\otimes_k B\rightarrow B$.
    \item The \emph{source} $s\colon R\rightarrow B$ and the \emph{target} $t\colon R^{\operatorname{op}}\rightarrow B$ are $k$-algebra morphisms whose images in $B$ commute. This yields a $k$-algebra morphism
\begin{equation*}
e:=\,\mu\circ (s \otimes_k t)\colon \,R^{e}\rightarrow B,
\end{equation*}
which with Lemma~\ref{isoOverRing} gives rise to an $R^e$-ring structure on $B$. We denote the underlying $R^{e}$-bimodule (resp. left $R^{e}$-module) of this $R^e$-ring by $_e{B}_e$ (resp. by $_e{B}$).
    \item $(_eB,\Delta,\epsilon)$ is a coalgebra in the monoidal category $(_{R^e}\operatorname{Mod},\boxtimes,R).$
\end{itemize}
In this situation, the \emph{(left) Takeuchi product} $B {\operatorname{\times}}_R B\subset {_e B \boxtimes {_e B}}$, defined by
\begin{equation*}
    B {\operatorname{\times}}_R B := \,\left\{ \sum a_i\otimes b_i \ \vert \ \sum a_it(r)\otimes b_i=\sum a_i \otimes b_i s(r) \text{ for all } r\in R\right\},
\end{equation*}
is a $k$-algebra with product given by factorwise multiplication on representatives.

We require:
\begin{enumerate}[label=(\roman*)]
    \item $\Delta(B)\subset B {\operatorname{\times}_R} B$.
    \item $\Delta\colon \,B \rightarrow B {\operatorname{\times}_R} B$ is a $k$-algebra morphism.
    \item $\epsilon(b s(\epsilon(b^{\prime})))\,=\,\epsilon(bb^{\prime})\,=\,\epsilon(b t(\epsilon(b^{\prime})))$ for all $b,b^{\prime}\in B$.
    \item $\epsilon (1_B)\,=\,1_R.$
\end{enumerate}
We call the underlying $k$-algebras $B$ and $R$ the \emph{total algebra} and the \emph{base algebra}.

Right $\times_R$-bialgebras are defined analogously.
\end{definition}

\begin{remark}[Equivalent definitions]
    The notions of $\times_R$-bialgebra, Lu's $R$-bialgebroid \cite{LuBialgebroids} and Xu's $R$-bialgebroid with an anchor \cite{XuQuantum} are all equivalent; see \cite{BrMiBialg}. 
\end{remark}

\begin{remark}[Commutative base algebra]\label{def: timesR-bialgebra for R commutative}
For commutative $R$, one often additionally requires that the source and target maps of a $\times_R$-bialgebra $B$ agree. In this case, $(_eB,\Delta,\epsilon)$ becomes a coalgebra in the symmetric monoidal category $_R{\operatorname{Mod}}$ of left $R$-modules.
\end{remark}

\begin{definition}\label{def: cocomm timesR-bialgebra}
 Let  $B$ be a $\times_R$-bialgebra with commutative base algebra $R$ in the sense of Remark~\ref{def: timesR-bialgebra for R commutative}. If the coalgebra $(_eB,\Delta,\epsilon)$ in $_R{\operatorname{Mod}}$ is cocommutative, we call $B$ \emph{cocommutative}.
\end{definition}

\begin{definition}\label{def:moduel over timesR-bialgebra} Let $B$ be a $\times_R$-bialgebra. The category $B\operatorname{-Mod}$ of \emph{(left) $B$-modules} is the category of left modules over the $R^{e}$-ring $_eB_e$ in the sense of Definition~\ref{module over r-ring}.
\end{definition}

The notions from Definitions~\ref{def:bialgebroid} and~\ref{takeuchi bialgebra} are indeed equivalent:

\begin{theorem}\label{Sz corr bialgebroied}\emph{(\cite{SzCorrespondence}).}
As in Definition~\ref{module over r-ring}, given an $R^e$-ring $B$, denote by $_B T$ the monad $B\otimes_{R^e}{?}$ on the monoidal category $(_{R^e} {\operatorname{Mod}},\boxtimes)$.

The assignment $B\mapsto {_BT}$ yields a correspondence between $\times_R$-bialgebras and $R$-bialgebroids.
\end{theorem}

\begin{remark}[$\times_R$-Hopf algebras are left $R$-Hopf algebroids]\label{rem: timesR Hopf as left Hopf algebrd}
    Under the correspondence in Theorem~\ref{Sz corr bialgebroied}, Schauenburg's $\times_R$-Hopf algebras \cite[Def. 3.5]{HopfSchauen} (which are defined as $R$-bialgebroids where a so-called Hopf-Galois map is invertible) correspond to right $R$-Hopf algebroids; see \cite[\S7.2]{BLV}.
\end{remark}

\begin{remark}[Further correspondences]\label{rem: further correspondences}
	\begin{enumerate}[label=(\roman*)]
		\item By definition, modules over a $\times_R$-bialgebra coincide with modules over the associated $R$-bialgebroid.
		\item If \(R\) is commutative, Theorem~\ref{Sz corr bialgebroied} specializes to a correspondence between the class of $\times_R$-bialgebras (as in Remark~\ref{def: timesR-bialgebra for R commutative}) and that of $R$-bialgebroids (as in Remark~\ref{rem: Bialgebroid over commutative base}).
		\item This correspondence further restricts to one between cocommutative $\times_R$-bialgebras (Definition~\ref{def: cocomm timesR-bialgebra}) and cocommutative $R$-bialgebroids (Definition~\ref{def: cocommutative bialgebroid}).
	\end{enumerate}
\end{remark}

\begin{example}[Enveloping algebras of Lie--Rinehart algebras]\label{ex: envelopin algebras of lie-rinehart}
Let $R$ be a commutative $k$-algebra.
A \emph{Lie--Rinehart algebra} \cite{Hue, KrMah} $L$ over $R$ is a $k$-Lie algebra $L$ equipped with an $R$-module structure $R \otimes_k L \rightarrow L, \; r \otimes x \mapsto r \cdot x$, and an $R$-linear $k$-Lie algebra morphism $\omega \colon L \to \operatorname{Der}_k(R)$ satisfying the \emph{Leibniz rule}
\begin{align}
[x, r \cdot y] \;&=\; r\cdot[x, y]+\omega(x)(r)\cdot y,
\end{align}
for all $x, y \in L$ and $r \in R$. Lie--Rinehart algebras are also known as $(R,L)$-Lie algebras \cite{Rine} or Lie algebroids \cite{abedin2024quantumgroupoidsmodulispaces}. A standard geometric example is the pair $(C^{\infty}(M),\Gamma(TM))$ for a smooth manifold $M$, where $R=C^{\infty}(M)$ is the (generally infinite-dimensional) algebra of smooth functions and $L=\Gamma(TM)$ is the Lie algebra of smooth vector fields. In contrast, we restrict our attention in this paper to finite-dimensional 
$R$.

The \emph{universal enveloping algebra} $\mathscr{U}_R(L)$ of a Lie--Rinehart algebra $L$ over $R$ is the universal $k$-algebra equipped with a morphism $\iota_R\colon R \to \mathscr{U}_R(L)$ of $k$-algebras and another morphism $\iota_L\colon L \to \mathscr{U}_R(L)$ of $k$-Lie algebras satisfying
\begin{align}
	\iota_L(r\cdot x) \;=\; \iota_R(r)\iota_L(x) \qquad \text{and} \qquad \iota_L(x)\iota_R(r)-\iota_R(r)\iota_L(x) \;=\; \iota_R\big(\omega(x)(r)\big),
\end{align}
for all $r\in R$ and $x\in L$; see, e.g., \cite{Hue90} for details. It is well known \cite[Example~8]{KoKr} that $\mathscr{U}_R(L)$ is a cocommutative $\times_R$-Hopf algebra with source and target maps $s = t= \iota_R$ and comultiplication and counit determined by
	\begin{align}
	\Delta(\iota_L(x)) \;=\; 1_R \boxtimes \iota_L(x)+ \iota_L(x) \boxtimes 1_R 
	\qquad \text{and} \qquad
	\epsilon(\iota_L(x)) \;=\; 0,
\end{align}
for all $x\in L$. By Remark~\ref{rem: Left vs. right Hopf}, $\mathscr{U}_R(L)$ is thus an $R$-Hopf algebroid over the commutative base $R$, in the sense of Remarks~\ref{rem: further correspondences}.(ii) and~\ref{rem: GV for comm base}.
\end{example}

We also need a more restrictive notion of Hopf algebroid than that in Definition~\ref{def:hopf algebroid}:

\begin{definition}(\cite[Def. 4.1]{HopfAlBoehmSzl}).\label{def: full Hopf algebroids}
	An \emph{antipode} on a $\times_R$-bialgebra $(B,s,t,\Delta,\epsilon)$ is an invertible anti-algebra map $S\colon B \to B$, with inverse $S^{-1}\colon B \to B$, such that
	\begin{align}
		S \circ t & \;=\; s,
	\end{align}
	and, for all $b\in B$,
	\begin{align}
		\big(S(b_{(1)})_{(1)} \, b_{(2)}\big) \boxtimes S(b_{(1)})_{(2)} &\;=\; 1_B \boxtimes S(b),\\
		S^{-1}(b_{(2)})_{(1)} \boxtimes  \big(S^{-1}(b_{(2)})_{(2)} \, b_{(1)}\big) &\;=\; S^{-1}(b) \boxtimes 1_B.
	\end{align}
	Here, the multiplication of the $R^e$-ring $_e{B}_e$ is denoted by concatenation.
	
	A $\times_R$-bialgebra, together with the datum of an antipode, is called a \emph{full $R$-Hopf algebroid}.
\end{definition}

\begin{remark}[Comparing definitions]\label{rem: full vs Hopf algebroids}
\begin{enumerate}[label=(\roman*)]
	\item The underlying $R$-bialgebroid of a full $R$-Hopf algebroid is indeed an $R$-Hopf algebroid in the sense of Definition~\ref{def:hopf algebroid}; see \cite[Prop. 4.2]{HopfAlBoehmSzl} or \cite[Prop. 2.7]{DaLaZan}. 
	\item Conversely, an $R$-Hopf algebroid need not admit an antipode \cite[Rem. 3.12]{KoPost}, \cite{KrRo}, \cite{Woj}. Even when one exists, it may fail to be unique \cite{Boehm}.
\end{enumerate}
\end{remark}

\begin{example}[Smash product algebras]\label{ex: smash product algebras}
	Let $H$ be a Hopf algebra over $k$ with invertible antipode $S_H$, and let $(R,\tikztriangleright,\rho)$ be a commutative algebra in the braided monoidal category ${}_H\mathcal{YD}^{H}$ of left-right Yetter--Drinfeld modules over $H$. The \emph{smash product algebra} $R \# H$ is the vector space $R \otimes_k H$ with unit $1_R \# 1_H$ and multiplication
	\begin{equation}
		(x \# g) \, (y \# h) \;:=\; x \cdot (g_{(1)}\tikztriangleright y)\, \# \, g_{(2)}\cdot_H h, \qquad \qquad x,y\in R, \quad g,h\in H.
	\end{equation}
	
	The $k$-algebra $R \# H$ becomes an $R$-bialgebroid \cite[Thm. 4.1]{BrMiBialg} with structure maps:
	\begin{align}
		s(x) &\;:=\; x \# 1_H,\\
		t(x) &\;:=\; \rho(x)={x_{\langle 0\rangle}} \# {x_{\langle 1\rangle}},\\
		\Delta(x \# g) &\;:=\; (x \# g_{(1)}) \boxtimes (1_R \# g_{(2)}),\\
		\epsilon(x \# g) &\;:=\; \epsilon_{H}(g)x,\label{eq: counit smash product algebra}
	\end{align}
	for $x\in R$, $g\in H$. As claimed in \cite[Ex. 4.14]{HopfAlBoehmSzl}, the following endomorphism $S$ on $R \# H$
	\begin{equation}\label{eq: antipode smash product algebra}
		S(x \# g) \;:=\; \big(S_H(g_{(2)}) \cdot_H S_H^2(x_{\langle 1 \rangle})\big) \tikztriangleright {x_{\langle 0\rangle}} \, \# \, S_H(g_{(1)}) \cdot_H S_H^2(x_{\langle 2\rangle}),
	\end{equation}
	where $x\in R$ and $g\in H$, defines an antipode, making $R \# H$ a full $R$-Hopf algebroid.
\end{example}

\begin{example}[Skew group algebras]\label{ex: skew group algebras}
	Let $(G, \tikztriangleright)$ be a group acting by algebra automorphisms on a commutative $k$-algebra $R$. Equipped with the trivial right $G$-coaction $r\mapsto r\otimes e_G$, the algebra $R$ becomes a commutative algebra in ${}_{k[G]}\mathcal{YD}^{k[G]}$. The corresponding smash product algebra $R \# k[G]$, classically called the \emph{skew group algebra}, is a full $R$-Hopf algebroid by Example~\ref{ex: smash product algebras}. Specializing Equations~\eqref{eq: counit smash product algebra} and~\eqref{eq: antipode smash product algebra}, its counit and antipode are
	\begin{align}
		\epsilon(x \# g)\;=\;x, \qquad \qquad S(x \# g) \;=\; (g^{-1} \tikztriangleright x) \# g^{-1}, \qquad \qquad x\in R, \quad g \in G.
	\end{align}
\end{example}

\smallskip

This completes the list of Hopf algebroids to which we will return in Section~\ref{sec: Applications}.

\medskip

\section{Lifting theorem}\label{sec:liftingGVstructures}

We are now ready to state the last main result of this paper.
\begin{theorem}\label{main thm}
Let \(\cC\) be a closed monoidal category equipped with a distinguished object \(K\in \cC\), and let \(\cD\) be a GV-category with dualizing object \(k\in \cD\). Let $U\colon \cC \rightarrow \cD$ be a lax monoidal functor equipped with a Frobenius form $\upsilon^{0,U}\colon U(K) \to k$. If $U$ is conservative, then $\cC$ is a GV-category with dualizing object $K$. Moreover, the pair \((U,\upsilon^{0,U})\) defines a GV-functor.
\end{theorem}

\begin{remark}[\(U\) as a Frobenius LD-functor]\label{rem: Liftingt thm functor ist Frob LD}
 The GV-functor \(U\) is equivalently a Frobenius LD-functor between the LD-categories with negation \(\cC\) and \(\cD\), by Theorem~\ref{thm: GV equiv LDN}. In particular, \(U\) preserves LD-Frobenius algebras; see \cite[Rem. 3.15]{DeS}.
\end{remark}

Specializing Theorem~\ref{main thm} to $k=U(K)$ yields:

\begin{corollary}\label{second main thm}
Let \(\cC\) and \(\cD\) be closed monoidal categories, and let $U\colon \cC\rightarrow \cD$ be a lax monoidal functor. Let $K\in \cC$ be an object such that $U(K)\in \cD$ is dualizing. If the functor $U$ is conservative and closed, then $K$ is dualizing.
\end{corollary}
\begin{proof}
    Since \(U\) is closed, the form $\operatorname{id}_{U(K)}$ is Frobenius. We can thus apply Theorem~\ref{main thm}.
\end{proof}

The following result is an immediate consequence of Corollary~\ref{second main thm}:
\begin{corollary}\label{cor lifting thm}
Let $U\colon \cC\rightarrow \cD$ be a strong monoidal functor between closed monoidal categories. Assume that $\cD$ is an r-category in the sense of Definition~\ref{def:GV-category}. If the functor $U$ is conservative and closed, then the monoidal category $\cC$ is an r-category.  
\end{corollary}

\begin{remark}[The rigid case]
As mentioned in the introduction, in the setting of Corollary~\ref{cor lifting thm}, if the monoidal category $\cD$ is actually rigid, then so is $\cC$; see \cite[Lem. 3.4]{BLV}.
\end{remark}

\begin{remark}[\(U\) as a \emph{strong} Frobenius LD-functor]
	The \emph{closed strong monoidal} functors \(U\) of Corollaries~\ref{second main thm} and~\ref{cor lifting thm} are equivalently \emph{strong} Frobenius LD-functors between the LD-categories with negation \(\cC\) and \(\cD\). This follows from Theorem~\ref{thm: GV equiv LDN} and Proposition~\ref{cor:strong monoidal implies closed}.
\end{remark}

\begin{proof}[Proof of Theorem~\ref{main thm}]
  Denote by ${\varphi^{2,U}}\colon {\otimes} \circ {(U\times U)} \rightarrow {U}\circ {\otimes}$ the multiplication morphism of the lax monoidal functor $U$. By Proposition~\ref{CompDualizing}, it suffices to show that the unit $d^K=d$ and the counit $\widetilde{d}^K=\widetilde{d}$ from Equations~\eqref{DoubDualMor1} and~\eqref{DoubDualMor2} are both invertible. Fix an object $X\in \cC$. We only prove the invertibility of $\widetilde{d}_X$; the proof for $d_X$ is dual. Since the functor $U$ is conservative, it suffices to show that the morphism $U(\widetilde{d}_X)$ is invertible. To do so, consider the following outer diagram, whose top line is just $U(\widetilde{d}_X)$:
  \begin{equation*}
  \adjustbox{max width=\textwidth}{
  \begin{tikzcd}
	{U(X)} & {U\big(X (X\multimap K)\multimapinv (X\multimap K)\big)} & {U\Big(K\multimapinv (X\multimap K)\Big)} \\
	& {U\big( X (X\multimap K)\big) \multimapinv U\big(X\multimap K\big)} \\
	& {U(X) U(X\multimap K) \multimapinv U(X\multimap K)} & {k\multimapinv U(X\multimap K)}\\
    & {U(X) (U(X)\multimap k) \multimapinv U(X\multimap K)} \\
	{U(X)} & {U(X) \big(U(X)\multimap k\big)\multimapinv \big(U(X)\multimap k\big)} & {k\multimapinv \big(U(X)\multimap k\big).}
 \arrow["U(\higheroverline{\operatorname{coev}}^{X\multimap K}_X)"{yshift=6pt}, from=1-1, to=1-2]
 \arrow["U\big(\operatorname{ev}^X_K\,\multimapinv\,(X\,\multimap\, K)\big)"{yshift=6pt}, from=1-2, to=1-3]
 \arrow[equal, bend right=40, from=1-1, to=5-1]
 \arrow["\xi^{r,U}_{X\multimap K}", "\simeq"', bend left=30, from=1-3, to=3-3]
 \arrow["k\,\multimapinv\,\xi^{l,U}_X"', "\simeq", bend right=30, from=5-3, to=3-3]
 \arrow["\higheroverline{\operatorname{coev}}^{U(X)\,\multimap\, k}_{U(X)}"'{yshift=-5pt, xshift=10pt}, from=5-1, to=5-2]
 \arrow["\operatorname{ev}^{U(X)}_{k}\,\multimapinv\, \big(U(X)\,\multimap\, k\big)"'{yshift=-5pt}, pos=1, from=5-2, to=5-3]
 \arrow["{\tau^{r,U}}"', from=1-2, to=2-2]
 \arrow["\varphi^{2,U}\,\multimapinv\,U(X\,\multimap\, K)", from=3-2, to=2-2]
 \arrow["U(X) {\xi^{l,U}_X}\,\multimapinv\,U(X\multimap K)"', "\simeq", from=3-2, to=4-2]
 \arrow["U(X) \big(U(X)\multimap k\big)\,\multimapinv\,\xi^{l,U}_X", "\simeq"', from=5-2, to=4-2]
 \arrow["\big(\upsilon^{0,U} \,\circ\,{U(\operatorname{ev}^X_K)}\big)\,\multimapinv\,U(X\,\multimap\, K)"{description, yshift=-2pt}, bend left=10, from=2-2, to=3-3]
 \arrow["\higheroverline{\operatorname{coev}}^{U(X\multimap K)}_{U(X)}"{description}, bend right=30, from=1-1, to=3-2]
 \arrow["\operatorname{ev}^{U(X)}_{k}\multimapinv U(X\multimap K)"', bend right=2, from=4-2, to=3-3]
 \arrow[phantom,"\textup{(I)}"{description}, bend right=5, from=1-1, to=2-2]
 \arrow[phantom,"\textup{(II)}"{description, xshift=-25pt,yshift=7pt}, bend left=20, from=1-2, to=3-3]
\arrow[phantom,"\textup{(III)}"{xshift=30pt,yshift=-10pt}, bend right=10, from=2-2, to=3-3]
\arrow[phantom,"\textup{(IV)}"{xshift=0}, bend left=5, from=4-2, to=5-3]
\arrow[phantom,"\textup{(V)}"{xshift=-10pt}, bend left=40, from=5-1, to=3-2]
 \end{tikzcd}
}
\end{equation*}

For better readability, we have omitted some indices. Occasionally, we have also left out parentheses, avoiding ambiguous expressions by reading the monoidal product $\otimes$ before the internal homs. Finally, to fit the diagram onto the page, we have omitted the $\otimes$-symbol, e.g. we have written $XY$ instead of $X\otimes Y$.

Let us take a closer look at the above diagram: The bottom horizontal morphism is the unit morphism $\widetilde{d}_{U(X)}$ with respect to the object $k\in \cD$. It is invertible since $k$ is dualizing in $\cD$ by assumption. Additionally, since $\upsilon^{0,U}$ is a Frobenius form (Definition \ref{def: Frobenius form}), both rightmost vertical morphisms are invertible. To prove the invertibility of the topmost horizontal morphism $U(\widetilde{d}_X)$, it therefore suffices to show that the outer diagram commutes. 

The commutativity of the outer diagram follows from the commutativity of the inner diagrams labelled by Roman numerals. We show their commutativity next:

Diagram (I) commutes by Lemma~\ref{lemma: compatibility of multi and left internal hom transfo}; (II) by the definition of $\xi^{r,U}_X$ and the naturality of the comparator $\tau^{r,U}$; (IV) by the functoriality of right internal hom $\multimapinv$; and (V) by Lemma~\ref{lemma:extranaturality of eval and coeval}. By definition of the duality transformation $\xi^{l,U}_X$, the commutativity of diagram (III) amounts to the commutativity of the following outer diagram:

\begin{equation*}
\adjustbox{max width=1\textwidth, center}{
  \begin{tikzcd}
	{U\big( X (X\multimap K)\big) \multimapinv U\big(X\multimap K\big)} && {U(K)\multimapinv U(X\multimap K)}\\
    {U(X) U(X\multimap K) \multimapinv U\big(X\multimap K\big)} && {}\\
    {U(X) \big(U(X)\multimap U(K)\big) \multimapinv U\big(X\multimap K\big)}&& {}\\
    {U(X) \big(U(X)\multimap k \big) \multimapinv U\big(X\multimap K\big)} && {k\multimapinv U(X\multimap K).}
 \arrow["U(\operatorname{ev}^X_K)\,\multimapinv\, U(X\multimap K)"{yshift=4pt}, from=1-1, to=1-3]
 \arrow["\upsilon^{0,U}\,\multimapinv\, U(X\multimap K)"{xshift=4pt}, from=1-3, to=4-3]
 \arrow["\varphi^{2,U}\,\multimapinv\,U(X\multimap K)"{xshift=-2pt}, from=2-1, to=1-1]
 \arrow["U(X) \tau^{l,U} \,\multimapinv\, U(X\multimap K)"'{xshift=-2pt}, from=2-1, to=3-1]
 \arrow["U(X)\big(U(X)\,\multimap \,\upsilon^{0,U}\big)\,\multimapinv\, U(X\multimap K)"'{xshift=-2pt}, from=3-1, to=4-1]
 \arrow["\operatorname{ev}^{U(X)}_k\,\multimapinv\, U(X\multimap K)"'{yshift=-2pt}, from=4-1, to=4-3]
 \arrow["\operatorname{ev}^{U(X)}_{U(K)}\multimapinv\, U(X\multimap K)"'{description}, bend right=20,  from=3-1, to=1-3]
 \arrow[phantom,"\textup{(1)}"{xshift=0pt}, bend right=20, from=1-1, to=1-3]
 \arrow[phantom,"\textup{(2)}", bend left=15, from=3-1, to=4-3]
 \end{tikzcd}
}
\end{equation*}

Diagram (1) commutes by Lemma~\ref{lemma: compatibility of multi and left internal hom transfo}, while the commutativity of diagram (2) follows from the naturality of the evaluation $\operatorname{ev}^{U(X)}.$ This shows that diagram (III) commutes.
\end{proof}

\section{Applications}\label{sec: Applications}

We derive corollaries from the lifting theorem~\ref{main thm} applied to the algebraic structures in Subsections~\ref{sec: Hopf monads and Hopf algebroids} and~\ref{sec: algebras, bimodules, local modules}. Some are known; others, to the best of our knowledge, are new.

\begin{prop}\label{prop: category of bimodules is gv}
  Let $A=(A,\mu,\eta)$ be a GV-algebra in a GV-category $\cC$ admitting equalizers and coequalizers. Let $f\colon D(A)\xrightarrow{\simeq} D'(A)$ be an isomorphism of GV-coalgebras. Then the monoidal category $_A\cC_A$ of $A$-bimodules is a GV-category with dualizing object $D(A)\cong D'(A)$, whose \(A\)-bimodule structure is induced by $f$ via Proposition~\ref{prop: bimodule structure on dual of algebra}. Moreover, the forgetful functor \(U\colon {_A\cC_A} \to \cC\) is a GV-functor.
\end{prop}

See Appendix~\ref{app: Applications} for a proof.

\begin{remark}
An analogous result holds for the category of bicomodules over a GV-coalgebra.
\end{remark}

The next two corollaries follow immediately from Proposition~\ref{prop: category of bimodules is gv}.

\begin{corollary}\label{cor: pivotal GV bimodule}
    Let \(A\) be a GV-algebra in a \emph{pivotal} GV-category \(\cC\) admitting equalizers and coequalizers. Then the category $_A\cC_A$ is a GV-category with dualizing object \(D(A)\cong D'(A)\), whose \(A\)-bimodule structure is induced by the pivotal structure via Proposition~\ref{prop: bimodule structure on dual of algebra} and Lemma~\ref{lemma: coalgebra iso via pivotality}. Moreover, the forgetful functor \({_A\cC_A} \to \cC\) is a GV-functor.
\end{corollary}

The following example illustrates, that although $_A\cC_A$ forms a GV-category by Corollary~\ref{cor: pivotal GV bimodule} when $\cC$ is pivotal rigid, $_A\cC_A$ need not be rigid:

\begin{example}[Finitely-generated projective bimodules]\label{ex: Finitely-generated projective bimodules}
Let \(k\) be a commutative ring, and let \(R\) be an algebra in the pivotal rigid monoidal category of finitely-generated projective \(k\)-modules. By Corollary~\ref{cor: pivotal GV bimodule}, the corresponding category of \(R\)-bimodules is a GV-category, and the associated forgetful functor is a GV-functor. More explicitly, the $k$-linear dual $k$-module $R^{\ast}=\operatorname{Hom}_{k}(R,k)$, equipped with the $R$-bimodule structure
\begin{equation}
	(z.f.x)(y)\,:=\,f(x\cdot y \cdot z), \qquad \qquad x,y,z\in R, \quad f\in R^{\ast},
\end{equation} 
is a dualizing object. For $k$ a field, this is the main example discussed in \cite{fuchs2024grothendieckverdierdualitycategoriesbimodules}.
\end{example}

\begin{example}[Suplattices]
The category $\mathsf{SupLat}$ of complete lattices and supremum-preserving maps is a complete and cocomplete r-category \cite[\S2]{JoTi}. Its duality functor, given by order reversal, is an involution. By Corollary~\ref{cor: pivotal GV bimodule}, for any algebra $A$ in the pivotal r-category $\mathsf{SupLat}$, the category of \(A\)-bimodules in $\mathsf{SupLat}$ is a GV-category with dualizing object the opposite poset \(A^{\operatorname{op}}\). Algebras in $\mathsf{SupLat}$ are known as \emph{(unital) quantales}. Examples include the lattice of ideals of a ring, the power set of a monoid, and any locale. Their categories of bimodules (and modules) have been extensively studied; see, e.g., \cite{Nie} and \cite{JoTi}.
\end{example}

\begin{corollary}\label{cor: A-bimod for A comm}
	   Let \(A\) be a \emph{commutative} GV-algebra in a \emph{braided} GV-category \(\cC\) admitting equalizers and coequalizers. The category $_A\cC_A$ of $A$-bimodules admits two GV-structures. In each, the dualizing object is \(D(A)\cong D'(A)\), which carries two 
	   $A$-bimodule structures induced by the braiding. Explicitly, these are realized via the isomorphisms $\varphi^{+}_A$ and $\varphi^{-}_A$, respectively, by Proposition~\ref{prop: bimodule structure on dual of algebra} and Lemma~\ref{lemma: coalgebra iso via braiding}.(i). Moreover, for both GV-structures, the forgetful functor \({_A\cC_A} \to \cC\) is a GV-functor.
\end{corollary}

\begin{prop}\label{prop: A-mod is GV}\emph{(Cf. \cite[Thm. 3.9]{Creutzig_2025}).}
 Let \(A\) be a \emph{commutative} GV-algebra in a \emph{braided} GV-category \(\cC\) admitting equalizers and coequalizers. The category $_A\cC$ of left \(A\)-modules is a GV-category with dualizing object \(D'(A)\). The left \(A\)-action on $D'(A)$ is induced by the multiplication of $A$, via Equation~\eqref{eq: left action on internal hom} in Lemma~\ref{lemma: module structure on left internal hom}. Moreover, the forgetful functor \({_A\cC} \to \cC\) is a GV-functor.
\end{prop}
\begin{proof}
	Consider the inclusions \(B^{\pm}\colon {_A\cC} \hookrightarrow {_A\cC_A}\) that equip a left $A$-module $M$ with a right $A$-action by precomposing its left action with the (inverse) braiding $c^{\pm}_{M,A}$. By Proposition~\ref{lemma: modules closed}, these functors are conservative, closed, and strict monoidal. By Lemma~\ref{lemma: coalgebra iso via braiding}.(ii), the right $A$-actions on $B^{\pm}D'(A)$ coincide with those from Corollary~\ref{cor: A-bimod for A comm}. Hence, by that corollary, $B^{\pm}D'(A)$ are dualizing objects in ${_A\cC_A}$. By Corollary~\ref{second main thm}, ${_A\cC}$ is thus a GV-category, and $B^{\pm}$ become GV-functors. The forgetful functor \({_A\cC} \to \cC\) factors through ${_A\cC_A}$ via either $B^+$ or $B^-$. As a composite of GV-functors, it is thus itself a GV-functor by Corollary~\ref{cor: A-bimod for A comm}.
\end{proof}

\begin{remark}
 An analogous result holds for the category of right $A$-modules.
\end{remark}

\begin{proposition}\label{prop: local modules}\emph{(Cf. \cite[Thm. 3.11]{Creutzig_2025}).}
Let \(A\) be a \emph{commutative} GV-algebra in a \emph{braided} GV-category \(\cC\) admitting equalizers and coequalizers. If the $A$-module $D'(A)$ from Equation~\eqref{eq: left action on internal hom} in Lemma~\ref{lemma: module structure on left internal hom} is local, then the category of local $A$-modules ${_A\cC}^{\operatorname{loc}}$ is a braided GV-category, and the forgetful functor \({_A\cC}^{\operatorname{loc}} \to \cC\) is a braided GV-functor.
\end{proposition}

\begin{proof}
Since the braiding on ${_A\cC}^{\operatorname{loc}}$ is induced from the braiding on \(\cC\) \cite[Thm. 2.5]{Par95}, the lax monoidal forgetful functor is braided. The assumptions of Corollary~\ref{second main thm} are satisfied by Remark~\ref{rem: local modules} together with Proposition~\ref{prop: A-mod is GV}.
\end{proof}

\begin{remark}[The ribbon case]
	In the setting of Proposition~\ref{prop: local modules}, the $A$-module $D'(A)$ is local whenever $\cC$ is a ribbon GV-category; see \cite[Thm. 3.11]{Creutzig_2025}.
\end{remark}

Another corollary of the lifting theorem is due to Hasegawa and Lemay \cite[Thm. 5.9]{HaLe}, where it is discussed with different techniques.

\begin{prop}\label{hopf monad on gv}
    Let $T$ be a Hopf monad on a GV-category $(\cC,\otimes,1,K)$. Any $T$-module structure $\omega\colon \, T(K)\rightarrow K$ on the dualizing object $K$ yields a dualizing object $(K,\omega)$ in the monoidal category $\cC^T$ of $T$-modules.
    Moreover, this gives a bijective correspondence between:
    \begin{itemize}
        \item $T$-module structures on the dualizing object $K\in \cC$.
        \item Dualizing objects for the monoidal category of $T$-modules such that the forgetful functor $\cC^T\to \cC$ is a closed strict monoidal functor that strictly preserves the dualizing object.
    \end{itemize}
\end{prop}
\begin{proof}
    The forgetful functor $U_T\colon \cC^T\rightarrow \cC$ is conservative. It is a strict monoidal functor by definition of the monoidal structure on $\cC^T$ from Remark~\ref{monoidal structure on CT}. Since $T$ is a Hopf monad, we know by Theorem~\ref{Thm: Hopf monad closed} that the monoidal category $\cC^T$ is closed and that the forgetful functor $U_T$ is closed. The claim now follows from Corollary~\ref{second main thm}.
\end{proof}

Next, we apply Theorem~\ref{main thm} to Hopf algebroids. Connections between GV-categories and Hopf algebroids have been noted previously \cite{QuantumCatStreetDay,allen2024hopfalgebroidsgrothendieckverdierduality}. Throughout, let \(k\) be a field. Our results extend to commutative rings, provided finite-dimensionality is replaced by finite generation and projectivity (see Example~\ref{ex: Finitely-generated projective bimodules}); we restrict to the field case for readability.

\begin{prop}\label{Hopf algebroids lifting}
    Let $B$ be an $R$-Hopf algebroid with finite-dimensional base $k$-algebra $R$. Any $B$-module structure on the dual $R$-bimodule $R^{\ast}$ (Example~\ref{ex: Finitely-generated projective bimodules}) yields a dualizing object of the monoidal category $B{\operatorname{-mod}}^{\operatorname{fd}}$ of finite-dimensional $B$-modules. Also, the strict monoidal forgetful functor $B{\operatorname{-mod}}^{\operatorname{fd}}\rightarrow {_R {\operatorname{Mod}}}^{\operatorname{fd}}_R$ is closed and strictly preserves this dualizing object.
\end{prop}

\begin{proof}
By Theorem~\ref{Thm: Hopf monad closed}, the monoidal category $B{\operatorname{-mod}}$ is closed, and the strict monoidal forgetful functor $B{\operatorname{-mod}}\rightarrow {_{R} {\operatorname{Mod}}_{R}}$ is closed. Applying Proposition~\ref{lemma: bimodules closed} to finite-dimensional vector spaces, we deduce that $B{\operatorname{-mod}}^{\operatorname{fd}}$ is closed. Clearly, the restricted forgetful functor $B{\operatorname{-mod}}^{\operatorname{fd}}\rightarrow {_R {\operatorname{Mod}}}^{\operatorname{fd}}_R$ is also strict monoidal and closed. Using Example~\ref{ex: Finitely-generated projective bimodules}, the claim now follows from Theorem~\ref{second main thm}.
\end{proof}

\begin{remark}[Finite-dimensionality]
	We do not assume that $B$ is finite-dimensional; see Examples~\ref{ex: smash product algebras} and~\ref{ex: enveloping algebras}. Consequently, the endofunctor $B\otimes_{R^e}-$ on $R$-bimodules need not restrict to \emph{finite-dimensional} bimodules. Hence, Proposition~\ref{Hopf algebroids lifting} does not follow directly from Proposition~\ref{hopf monad on gv}.  
\end{remark}

\begin{remark}[Commutative base algebra]\label{rem: GV for comm base}
	Replacing Proposition~\ref{lemma: bimodules closed} and Example~\ref{ex: Finitely-generated projective bimodules} with Propositions~\ref{lemma: modules closed} and~\ref{prop: A-mod is GV}, gives an analogue of Proposition~\ref{Hopf algebroids lifting} for Hopf algebroids over a commutative base algebra in the sense of Remark~\ref{rem: further correspondences}.(ii).
\end{remark}

For a $k$-bialgebra $H$, restriction along the counit endows any vector space with an $H$-action. This generally fails for bialgebroids, as the counit need not be an algebra map; see Equation~(iii) of Definition~\ref{takeuchi bialgebra} and Example~\ref{ex: skew group algebras}. We now ask when the dual $R$-bimodule \(R^*\) still admits a $B$-module structure. The first such case occurs when an antipode is present:

\begin{remark}[Induced action on $R^{\ast}$]\label{rem: induced module structure on R*}
    Let $R$ be a finite-dimensional \(k\)-algebra, $B$ a\linebreak $\times_R$-bialgebra, and $S$ an anti-algebra morphism on $B$. We do not yet require that $S$ is an antipode in the sense of Definition~\ref{def: full Hopf algebroids}. The map $S$ induces a $B$-action via
    \begin{align}\label{eq: black action}
    \blacktriangleright_S\colon \,{B}\otimes_{k} R^{\ast} \, \longrightarrow\ , R^{\ast}, \qquad &(b\blacktriangleright_S f)(r)\,:=\,f(S(b)\triangleright r),
    \end{align} 
    where $\triangleright\colon B \otimes_k R\rightarrow R$ is the canonical $B$-action on the $k$-module $R$ defined by 
  	\begin{align}\label{eq: white action}
       b\triangleright r\,:=\,\epsilon(b\cdot s(r)), \qquad \qquad &b\in B, \quad r\in R.
    \end{align} Conditions (iii) and (iv) in Definition~\ref{takeuchi bialgebra} ensure that $\triangleright$ is indeed a left $B$-action, making the regular $R$-bimodule $R$ the monoidal unit of $B\operatorname{-mod}$.
\end{remark}

\begin{corollary}\emph{(\cite[Thm. 4.5]{allen2024hopfalgebroidsgrothendieckverdierduality}).}\label{cor: full Hopf gives GV}
Let $B$ be a \emph{full} $R$-Hopf algebroid with antipode $S$. The $B$-action $\blacktriangleright_S$ from Remark~\ref{rem: induced module structure on R*} makes the $R$-bimodule $R^{\ast}$ a dualizing object of $B{\operatorname{-mod}}^{\operatorname{fd}}$. Also, the strict monoidal forgetful functor $B{\operatorname{-mod}}^{\operatorname{fd}}\rightarrow {_R {\operatorname{Mod}}}^{\operatorname{fd}}_R$ is closed and strictly preserves this dualizing object.
\end{corollary}
\begin{proof}
The axioms of a full $R$-Hopf algebroid ensure that $\blacktriangleright_S$ descends to the quotient\linebreak $_e{B}_e\otimes_{R^e} R^{\ast}$. By Remark~\ref{rem: full vs Hopf algebroids}, the claim now follows from Proposition~\ref{Hopf algebroids lifting}.
\end{proof}

\begin{example}[Smash product algebras]\label{ex: GV smash product algebras}
	Let $H$ be a $k$-Hopf algebra with invertible antipode, and let $(R,\tikztriangleright,\rho)$ be a \emph{finite-dimensional} commutative algebra in the braided monoidal category ${}_H\mathcal{YD}^{H}$ of left-right Yetter--Drinfeld modules over $H$. Recall the full $R$-Hopf algebroid structure on the smash product algebra $R \# H$ from Example~\ref{ex: smash product algebras}. By Corollary~\ref{cor: full Hopf gives GV}, finite-dimensional $R \# H$-modules form a GV-category. The dualizing object is the $R$-bimodule $R^{\ast}$ from Example~\ref{ex: Finitely-generated projective bimodules}, with the $R\# H$-action induced by $\blacktriangleright_S$ from Equation~\eqref{eq: black action}. Here, $S$ is defined in Equation~\eqref{eq: antipode smash product algebra}.
	\end{example}
	
	\begin{example}[Skew group algebras]\label{ex: GV skew group}
Let $(G, \tikztriangleright)$ be a group acting by algebra automorphisms on a \emph{finite-dimensional} commutative $k$-algebra $R$. Recall the associated skew group algebra $R \# k[G]$ from Example~\ref{ex: skew group algebras}. By Corollary~\ref{cor: full Hopf gives GV}, the category of finite-dimensional \mbox{$R \# k[G]$-modules} (which is equivalent to the category of finite-dimensional $G$-equivariant $R$-modules) has a dualizing object $R^{\ast}$ with the $R\# k[G]$-action given explicitly by
\begin{align}
	\big((x \# g)\blacktriangleright_S f\big)(r)\,:=\,f(g^{-1}\tikztriangleright (x \cdot r)), \quad\quad\quad x,r\in R, \quad g\in G, \quad f\in R^{\ast}.
\end{align}
\end{example}

\smallskip

\begin{corollary}\label{ref: Hopf algebroids over sym Frob}
	Let $B$ be an $R$-Hopf algebroid over a symmetric Frobenius algebra $R$. The canonical $B$-module structure on the regular $R$-bimodule $R$ makes $B\operatorname{-mod}^{\operatorname{fd}}$ an $r$-category for which the strict monoidal forgetful functor $B{\operatorname{-mod}}^{\operatorname{fd}}\rightarrow {_R {\operatorname{Mod}}}^{\operatorname{fd}}_R$ is closed. 
	
	An analogous result holds for $R$-Hopf algebroids over a \emph{commutative} Frobenius algebra $R$ (in the sense of Remark ~\ref{rem: Bialgebroid over commutative base}).
\end{corollary}

\begin{proof}
	By assumption, $R$ is a finite-dimensional \(k\)-algebra such that the regular $R$-bimodule $R$ is isomorphic to its $k$-linear dual $R^{\ast}$. The claim follows directly from Proposition~\ref{Hopf algebroids lifting}.
\end{proof}

We apply Corollary~\ref{ref: Hopf algebroids over sym Frob} to Example~\ref{ex: envelopin algebras of lie-rinehart}:
\begin{example}[Enveloping algebras]\label{ex: enveloping algebras}
Let $L$ be a Lie--Rinehart algebra over a commutative Frobenius algebra $R$. Finite-dimensional modules over its universal enveloping algebra $\mathscr{U}_R(L)$ form an $r$-category. By Equation~\eqref{eq: white action}, the $\mathscr{U}_R(L)$-module structure on $R$ is induced by
\begin{equation}\label{eq: action universal enveloping on R}
	\mathscr{U}_R(L) \otimes_k R \;\xrightarrow{{\mathscr{U}_R(L)\,} \otimes_k {\,\iota_R}}\; \mathscr{U}_R(L) \otimes_k \mathscr{U}_R(L) \;\longrightarrow\; \mathscr{U}_R(L) \;\xlongrightarrow{\epsilon}\; R,
\end{equation}
where the second map is multiplication in $\mathscr{U}_R(L)$.
\end{example}

We specialize Example~\ref{ex: enveloping algebras} to a concrete Lie--Rinehart algebra.

\begin{example}[Truncated modular Weyl algebras]\label{ex: trunctated mod Weyl algebra}
	Let $k$ be a field of characteristic $p>0$, and consider the truncated polynomial algebra
	\begin{equation}
	R := k[x_1,\ldots,x_n]\big/(x_i^p).
	\end{equation}
	It is a commutative Frobenius algebra with Frobenius form $R\to k$ given by extracting the coefficient of the top-degree monomial $x_1^{p-1}x_2^{p-1}\cdots x_n^{p-1}$. Its $k$-Lie algebra of derivations
	\begin{equation}
	\operatorname{Der}_k(R) = \bigoplus_{i=1}^n R \cdot \partial_i, \qquad \partial_i(x_j) = \delta_{ij},
	\end{equation}
	forms a Lie--Rinehart algebra over $R$. The enveloping algebra $\mathscr{U}_R(\operatorname{Der}_k(R))$ is the \emph{truncated modular Weyl algebra} $A_n^{(p)}$; compare this terminology to \cite{PetersShi}. $A_n^{(p)}$ is the $k$-algebra generated by $x_i, \partial_i$, where $1 \leq i \leq n$,
	subject to the relations 
	\begin{align}
		x_i^p=0, \quad  x_i x_j=x_j x_i, \quad \partial_i\partial_j=\partial_j\partial_i, \quad \partial_ix_j-x_j\partial_i=\delta_{ij}.
	\end{align}
	By Example~\ref{ex: enveloping algebras}, the category of finite-dimensional $A_n^{(p)}$-modules is an $r$-category. 
	By Equation~\eqref{eq: action universal enveloping on R}, the $A_n^{(p)}$-module structure on $R$ is the standard action of differential operators, with $x_i$ acting by multiplication and $\partial_i$ by differentiation.
\end{example}

\begin{remark}[Antipodes on $A_n^{(p)}$]
Although universal enveloping algebras of Lie--Rinehart algebras need not admit an antipode \cite{KrRo}, the truncated modular Weyl algebra $A_n^{(p)}$ does: it carries a full $R$-Hopf algebroid structure with antipode given by $S(\partial_i) = -\partial_i$ and $S(x_i) = x_i$.
\end{remark}

\vspace{-0.4cm}
\appendix

\addtocontents{toc}{\protect\setcounter{tocdepth}{1}}
\section{Proofs}\label{sec:appendixproof}

\subsection{Grothendieck--Verdier functors}\label{sec:GV-functors-appendix}

\begin{proof}[Proof of Lemma~\ref{lemma:left-right duality transfos relations}]
For all \(X\in \cC\), we have:
\allowdisplaybreaks
\begin{align*}
	\xi^{l,F}_{D(X)} \circ F(d_X)
	&\;\eqabove{(1)}\;
	\big(FD(X)\multimap \upsilon^{0,F}\big) \circ \tau^{l,F}_{D(X),K}\circ {F\big(D(X)\multimap \higheroverline{\operatorname{ev}}^X_K\big)} \circ F\big(\operatorname{coev}^{D(X)}_{X}\big)\\[0.02em]
	&\;\eqabove{(2)}\;
	\big(FD(X)\multimap \upsilon^{0,F}\big) \circ {\big(FD(X)\multimap F(\higheroverline{\operatorname{ev}}^X_K)\big)}
	\\& \; \qquad \circ \tau^{l,F}_{D(X),D(X)\otimes X} \circ F\big(\operatorname{coev}^{D(X)}_{X}\big)\\[0.02em]
	&\;\eqabove{(3)}\;
	\big(FD(X)\multimap \upsilon^{0,F}\big) \circ {\big(FD(X)\multimap F(\higheroverline{\operatorname{ev}}^X_K)\big)} 
	\\& \; \qquad \circ \big(FD(X)\multimap \varphi^{2,F}_{D(X),X}\big) \circ \operatorname{coev}^{FD(X)}_{F(X)}\\[0.02em]
	&\;\eqabove{(4)}\;
	\big(FD(X)\multimap \upsilon^{0,F}\big) \circ {\big(FD(X)\multimap \higheroverline{\operatorname{ev}}^{F(X)}_{F(K)})\big)} 
	\\& \quad \qquad \circ \big(FD(X)\multimap (\tau^{r,F}_{K,X}\otimes F(X))\big) \circ \operatorname{coev}^{FD(X)}_{F(X)}\\[0.02em]
	&\;\eqabove{(5)}\;
	{\big(FD(X)\multimap \higheroverline{\operatorname{ev}}^{F(X)}_{K})\big)} 
	\circ \big(FD(X)\multimap (\xi^{r,F}_{X}\otimes F(X))\big) 
	\circ \operatorname{coev}^{FD(X)}_{F(X)}\\[0.02em]
	&\;\eqabove{(6)}\;
	 \big(\xi^{r,F}_{X}\multimap \higheroverline{\operatorname{ev}}^{F(X)}_{K}\big) 
	\circ \operatorname{coev}^{DF(X)}_{F(X)}\\[0.02em]
	&\;\eqabove{(7)}\;
	D'(\xi^{r,F}_X) \circ d_{F(X)}.
\end{align*}
Equation (1) holds by the definitions of \(\xi^{l,F}\) and $d$; (2) by the naturality of \(\tau^{l,F}\); (3) by Equation~\eqref{coeval lax functor compatibility} in Lemma~\ref{lemma: compatibility of multi and left internal hom transfo}; (4) by the right closed analogue of Equation~\eqref{eval lax functor compatibility} in the same lemma; (5) by the naturality of \(\higheroverline{\operatorname{ev}}^{F(X)}\); (6) by the extranaturality of \(\operatorname{coev}\) (Lemma~\ref{lemma:extranaturality of eval and coeval}) and the functoriality of the left internal hom; and (7) by the definition of \(d_{F(X)}\). 

This proves Equation~\eqref{eq:left-right duality transfos relations 1}, with Equation~\eqref{eq:left-right duality transfos relations 2} following by a similar argument.
\end{proof}
\vspace{-0.3cm}
\begin{proof}[Proof of Lemma~\ref{lemma:duality transfos relations}]	
	We prove only Equation~\eqref{eq: GV functor 1}. Functoriality of the left internal hom and naturality of \(\beta\) reduce the claim to verifying, for all \(X,Y \in \cC\),
	\begin{equation*}
		\beta_{F(X),F(Y),F(K)} \circ (\varphi^{2,F}_{X,Y} \multimap F(K)) \circ \tau^{l,F}_{X \otimes Y,K} 
		\;=\; (F(Y) \multimap \tau^{l,F}_{X,K}) \circ \tau^{l,F}_{Y,D'(X)} \circ F(\beta_{X,Y,K}).
	\end{equation*}  
This follows directly from Lemma~\ref{lemma: compatibility of multi and beta} (with $Z=K$). Equation~\eqref{eq: GV functor 2} is shown analogously.
\end{proof}
\vspace{-0.8cm}
\begin{proof}[Proof of Lemma~\ref{lemma:pre-GV-morph-duality-transfo}]
	For all \(X\in \cC\), we have:
	\allowdisplaybreaks
	\begin{align*}
		D(f_X) \circ \xi^{r,G}_X \circ f_{D(X)}
		&\;\eqabove{(1)}\;
		\big(\upsilon^{0,G}\multimapinv f_X\big)\circ \tau^{r,G}_{K,X} \circ f_{D(X)}\\
		&\;\eqabove{(2)}\;
		\big(\upsilon^{0,G}\multimapinv F(X)\big) \circ \big(f_K\multimapinv F(X)\big)\circ \tau^{r,F}_{K,X}\\
		&\;\eqabove{(3)}\;
		\big(\upsilon^{0,F}\multimapinv F(X)\big)\circ \tau^{r,F}_{K,X}\\
		&\;\eqabove{(4)}\;
		\xi^{r,F}_X.
	\end{align*}
	Equations (1) and (4) hold by the definition of the duality transformations \(\xi^{r,G}\) and \(\xi^{r,F}\); (2) follows from the right closed analogue of Lemma~\ref{lemma:taul and monoidal nat transfos}; and (3) holds because $f$ is a morphism of GV-functors (Equation~\eqref{eq: morphism of GV-functors}).
\end{proof}

\subsection{From \(\mathsf{GV}\) to \(\mathsf{LDN}\)}\label{sec:GV-cat-and-LD}
\begin{proof}[Proof of Lemma~\ref{lemma: F is Frob LD}]
We verify that \((\upsilon^{2,F},\upsilon^{0,F})\) defines an oplax \(\parLL\)-monoidal structure. For \(X,Y,Z \in \cC\), the naturality of \(D'D \simeq \id_{\cC}\) and \(D'D \simeq \id_{\cD}\), together with \cite[Lem. 2.56]{DeS}, reduces the coassociativity of \(\upsilon^{2,F}\) to the commutativity of the outer diagram below, with all indices omitted for readability:
\allowdisplaybreaks
\begin{equation}
  \adjustbox{max width=\textwidth}{
  \begin{tikzcd}[row sep=4ex,column sep=4ex]
	{FD\big((D'(Z)\ot D'(Y)) \ot D'(X)\big)} & {FD\big(D'(Z)\ot (D'(Y)\ot D'(X))\big)}\\
    {DF\big((D'(Z)\ot D'(Y)) \ot D'(X)\big)} & {DF\big(D'(Z)\ot (D'(Y)\ot D'(X))\big)}\\
    {D\big(F(D'(Z)\ot D'(Y)) \ot FD'(X)\big)} & {D\big(FD'(Z)\ot F(D'(Y)\ot D'(X))\big)}\\
    {D\big((FD'(Z)\ot FD'(Y)) \ot FD'(X)\big)} & {D\big(FD'(Z)\ot (FD'(Y)\ot FD'(X))\big)}\\
    {D\big((D'F(Z)\ot D'F(Y)) \ot D'F(X)\big)} & {D\big(D'F(Z)\ot (D'F(Y)\ot D'F(X))\big).}
    \arrow["FD(\ao)"{yshift=3pt}, from=1-1, to=1-2]
    \arrow["\xi^{r,F}"'{yshift=1pt, xshift=-3pt}, from=1-1, to=2-1]
    \arrow["\xi^{r,F}"{yshift=1pt, xshift=3pt}, from=1-2, to=2-2]
    \arrow["D(\varphi^{2,F})"'{yshift=1pt, xshift=-3pt}, from=2-1, to=3-1]
    \arrow["D(\varphi^{2,F})"{yshift=1pt, xshift=3pt}, from=2-2, to=3-2]
    \arrow["D\big(\varphi^{2,F}\,\ot\, FD'(X)\big)"'{yshift=1pt, xshift=-3pt}, from=3-1, to=4-1]
    \arrow["D\big(FD'(Z) \,\ot\, \varphi^{2,F}\big)"{yshift=1pt, xshift=3pt}, from=3-2, to=4-2]
    \arrow["D\big((\xi^{l,F} \,\ot\, \xi^{l,F})\,\ot\, \xi^{l,F}\big)^{-1}"'{yshift=1pt, xshift=-3pt}, from=4-1, to=5-1]
    \arrow["D\big({\xi^{l,F}} \,\ot\, ({\xi^{l,F}}\,\ot\, {\xi^{l,F}})\big)^{-1}"{yshift=1pt, xshift=3pt}, from=4-2, to=5-2]
    \arrow["DF(\ao)"{yshift=1pt}, from=2-1, to=2-2]
    \arrow["D(\ao)"{yshift=1pt}, from=4-1, to=4-2]
    \arrow["D(\ao)"'{yshift=-4pt}, from=5-1, to=5-2]
    \arrow[phantom,"\textup{(I)}"{yshift=2pt}, from=1-1, to=2-2]
    \arrow[phantom,"\textup{(II)}"{yshift=2pt}, from=2-1, to=4-2]
    \arrow[phantom,"\textup{(III)}"{yshift=2pt}, from=4-1, to=5-2]
 \end{tikzcd}
}
\end{equation}
Diagram (I) commutes by the naturality of \(\xi^{r,F}\); (II) by the associativity of \(\varphi^{2,F}\); and (III) by the naturality of \(\alpha\). The counitality of \((\upsilon^{2,F},\upsilon^{0,F})\) follows similarly from the unitality of \((\varphi^{2,F},\varphi^{0,F})\), together with Lemmas~\ref{lem:internal composition is extranatural} and~\ref{lemma:left-right duality transfos relations}.

Finally, we verify the two Frobenius relations. By the definition of \(\upsilon^{2,F}\), the Frobenius relation~\eqref{eq:F1} amounts to the commutativity of the following diagram for all \(X,Y,Z \in \cC\):
\begin{equation}\label{dgm: GV functor is LD1}
  \adjustbox{max width=\textwidth}{
  \begin{tikzcd}[row sep=3ex]
	{F(X)\ot FD\big(D'(Z)\ot D'(Y)\big)} & {F(X)\ot DF\big(D'(Z)\ot D'(Y)\big)}\\
    {F\Big(X \ot D\big(D'(Z)\ot D'(Y)\big)\Big)} & {F(X)\ot D(FD'(Z)\ot FD'(Y))}\\
    {FD\big(D'(Z)\ot D'(X \ot Y)\big)} & {F(X)\ot D\big(D'F(Z)\ot D'F(Y)\big)}\\
    {DF\big(D'(Z)\ot D'(X \ot Y)\big)} & {D\Big(D'F(Z)\ot D'\big(F(X)\ot F(Y)\big)\Big)}\\
    {D\big(FD'(Z)\ot FD'(X \ot Y)\big)} & {D\big(D'F(Z)\ot D'F(X\ot Y)\big).}
    \arrow["{F(X)}\,\ot\, {\xi^{r,F}}"{yshift=3pt, xshift=2pt}, from=1-1, to=1-2]
    \arrow["\varphi^{2,F}"'{yshift=1pt, xshift=-3pt}, from=1-1, to=2-1]
    \arrow["{F(X)}\,\ot\, {D(\varphi^{2,F})}"{yshift=1pt, xshift=3pt}, from=1-2, to=2-2]
    \arrow["F(\distl)"'{yshift=1pt, xshift=-3pt}, from=2-1, to=3-1]
    \arrow["F(X)\,\ot\, D(\xi^{l,F}\,\ot\, \xi^{l,F})^{-1}"{yshift=1pt, xshift=3pt}, from=2-2, to=3-2]
    \arrow["\xi^{r,F}"'{yshift=1pt, xshift=-3pt}, from=3-1, to=4-1]
    \arrow["\distl"{yshift=1pt, xshift=3pt}, from=3-2, to=4-2]
    \arrow["D(\varphi^{2,F})"'{yshift=1pt, xshift=-3pt}, from=4-1, to=5-1]
    \arrow["D\big(D'F(Z)\,\ot\, D'(\varphi^{2,F})\big)"{yshift=1pt, xshift=3pt}, from=4-2, to=5-2]
    \arrow["D(\xi^{l,F} \,\ot\, \xi^{l,F})^{-1}"'{yshift=-5pt,xshift=2pt}, from=5-1, to=5-2]
 \end{tikzcd}
}
\end{equation}

Using Equation~\eqref{eq: GV functor 2} from Lemma~\ref{lemma:duality transfos relations} twice and naturality repeatedly, one can show that the commutativity of diagram~\eqref{dgm: GV functor is LD1} is equivalent to that of the following diagram:

\begin{equation}\label{dgm: GV functor is LD2}
  \adjustbox{max width=\textwidth}{
  \begin{tikzcd}
	{F(X)\ot F(Y \multimapinv Z)} && {F(X)\ot \big(F(Y) \multimapinv F(Z)\big)}\\
    {F\big(X \ot (Y \multimapinv Z)\big)} && {\big(F(X)\ot F(Y)\big) \multimapinv F(Z)}\\
    {F\big((X \ot Y) \multimapinv Z\big)} && {F(X \ot Y) \multimapinv F(Z).}
    \arrow["{F(X)}\,\ot\, \tau^{r,F}_{Y,Z}"{yshift=3pt}, from=1-1, to=1-3]
    \arrow["\varphi^{2,F}_{X,Y\multimapinv Z}"'{xshift=-3pt,yshift=1pt}, from=1-1, to=2-1]
    \arrow["{\widetilde{\distl}_{F(X),F(Y),F(Z)}}"{xshift=3pt, yshift=1pt}, from=1-3, to=2-3]
    \arrow["F(\widetilde{\distl}_{X,Y,Z})"'{xshift=-3pt,yshift=1pt}, from=2-1, to=3-1]
    \arrow["\varphi^{2,F}_{X,Y}\,\multimapinv\, F(Z)"{yshift=1pt, xshift=3pt}, from=2-3, to=3-3]
    \arrow["\tau^{r,F}_{X\ot Y,Z}"'{yshift=-3pt}, from=3-1, to=3-3]
 \end{tikzcd}
}
\end{equation}
Here, \(\widetilde{\distl}\) denotes the natural transformation from Equation~\eqref{eq: distl via inhoms} in Remark~\ref{rem:InHoms and distributors}.

To establish the commutativity of diagram~\eqref{dgm: GV functor is LD2}, we rewrite the composite of the top horizontal and right vertical morphisms using the definitions of \(\tau^{r,F}\) and \(\widetilde{\distl}\) (Definition~\ref{def: closed monoidal functor} and Remark~\ref{rem:InHoms and distributors}), obtaining:
\allowdisplaybreaks
\begin{align*}
    & {\Big(\varphi^{2,F}_{X,Y}\multimapinv F(Z)\Big)} \circ {\Big(\big(F(X)\ot \higheroverline{\operatorname{ev}}^{F(Z)}_{F(Y)}\big)\multimapinv F(Z)\Big)}\circ {\Big(\aoi_{F(X),F(Y)\,\multimapinv\, F(Z),F(Z)} \multimapinv F(Z)\Big)}\\ 
    & \; \circ {\higheroverline{\operatorname{coev}}^{F(Z)}_{F(X)\,\ot\,(F(Y)\,\multimapinv\,F(Z))}} \circ {\Big(F(X)\ot \big(F(\higheroverline{\operatorname{ev}}^{Z}_{Y})\multimapinv F(Z)\big)\Big)}\\
    &\; \circ \Big(F(X)\ot \big(\varphi^{2,F}_{Y\multimapinv Z,Z} \multimapinv F(Z)\big)\Big) \circ {\Big(F(X)\ot \higheroverline{\operatorname{coev}}^{F(Z)}_{F(Y \multimapinv Z)}\Big)}\\[.4em]
    &\quad\eqabove{(1)}\;
    {\Big(F(X\ot \higheroverline{\operatorname{ev}}^Z_Y)\multimapinv F(Z)\Big)}\circ {\Big(\varphi^{2,F}_{X,(Y\,\multimapinv\, Z)\,\ot\,Z}\multimapinv F(Z)\Big)} \circ {\Big(\big(F(X)\ot \varphi^{2,F}_{Y\multimapinv Z,Z}\big)\multimapinv F(Z)\Big)}\\
    &\quad \qquad \circ {\Bigg(\bigg(F(X)\ot \Big(\higheroverline{\operatorname{ev}}^{F(Z)}_{F(Y\multimapinv Z)\,\ot\, F(Z)}\circ \big(\higheroverline{\operatorname{coev}}^{F(Z)}_{F(Y\multimapinv Z)}\ot F(Z)\big)\Big)\bigg)\multimapinv F(Z)\Bigg)}\\
    & \quad \qquad \circ {\Big(\aoi_{F(X),F(Y)\,\multimapinv\, F(Z),F(Z)} \multimapinv F(Z) \Big)} \circ {\higheroverline{\operatorname{coev}}^{F(Z)}_{F(X)\,\ot\, F(Y\multimapinv Z)}}\\[.4em]
    &\quad\eqabove{(2)}\; 
    {\Big(F(X\ot \higheroverline{\operatorname{ev}}^Z_Y)\multimapinv F(Z)\Big)}\circ {\Big(\varphi^{2,F}_{X,(Y\multimapinv Z)\,\ot\, Z}\multimapinv F(Z)\Big)} \circ {\Big(\big(F(X)\ot \varphi^{2,F}_{Y\multimapinv Z,Z}\big)\multimapinv F(Z)\Big)}
    \\& \quad \qquad \circ {\Big(\aoi_{F(X),F(Y)\,\multimapinv\, F(Z),F(Z)} \multimapinv F(Z) \Big)} \circ {\higheroverline{\operatorname{coev}}^{F(Z)}_{F(X)\,\ot\, F(Y\multimapinv Z)}}\\[.4em]
    &\quad\eqabove{(3)}\; 
    {\Big(F(X\ot \higheroverline{\operatorname{ev}}^Z_Y)\multimapinv F(Z)\Big)}\circ {\Big(F(\aoi_{X,Y\multimapinv Z,Z})\multimapinv F(Z)\Big)} \circ {\Big(\varphi^{2,F}_{X\,\ot\, (Y\,\multimapinv\, Z),Z}\multimapinv F(Z)\Big)}
    \\& \quad \qquad \circ {\Big(\big(\varphi^{2,F}_{X,Y\multimapinv Z} \ot F(Z)\big)\multimapinv F(Z)\Big)} \circ {\higheroverline{\operatorname{coev}}^{F(Z)}_{F(X)\,\ot\, F(Y\multimapinv Z)}}\\[.4em]
    &\quad\eqabove{(4)}\; 
    {\Big(F(X\ot \higheroverline{\operatorname{ev}}^Z_Y)\multimapinv F(Z)\Big)}\circ {\Big(F(\aoi_{X,Y\multimapinv Z,Z})\multimapinv F(Z)\Big)} \circ {\Big(\varphi^{2,F}_{X\,\ot\, (Y\multimapinv Z),Z}\multimapinv F(Z)\Big)}
    \\& \quad \qquad \circ {\higheroverline{\operatorname{coev}}^{F(Z)}_{F(X\,\ot\, (Y\,\multimapinv\, Z))}} \circ {\varphi^{2,F}_{X,Y\multimapinv Z}}\\[.4em]
    &\quad\eqabove{(5)}\; 
    {\Big(F(X\ot \higheroverline{\operatorname{ev}}^Z_Y)\multimapinv F(Z)\Big)}\circ {\Big(F(\aoi_{X,Y\multimapinv Z,Z})\multimapinv F(Z)\Big)}\circ {\tau^{r,F}_{(X\,\ot\, (Y\,\multimapinv\, Z)  )\,\ot\, Z,Z}}
    \\&\quad \qquad \circ {F(\higheroverline{\operatorname{coev}}^Z_{X\,\ot\, (Y\,\multimapinv\, Z)})} \circ {\varphi^{2,F}_{X,Y\multimapinv Z}}\\[.4em]
    &\quad\eqabove{(6)}\; 
    {\tau^{r,F}_{X\ot Y,Z}}\circ {F\big((X\ot \higheroverline{\operatorname{ev}}^{Z}_Y)\multimapinv Z\big)} \circ {F(\aoi_{X,Y\multimapinv Z,Z}\multimapinv Z)} \circ {F(\higheroverline{\operatorname{coev}}^Z_{X\,\ot\, (Y\multimapinv Z)})} \circ {\varphi^{2,F}_{X,Y\multimapinv Z}}.
\end{align*} 
Equation (1) follows from naturality; (2) from a snake equation for \(({?}\,\ot\, F(Z)) \,\dashv\, ({?}\multimapinv F(Z))\); (3) from associativity of \(\varphi^{2,F}\); (4) from naturality of \(\higheroverline{\operatorname{coev}}^{F(Z)}\); (5) from Lemma~\ref{lemma: compatibility of multi and left internal hom transfo}; and (6) from the naturality of \(\tau^{r,F}\). By definition of \(\widetilde{\distl}\), the last line equals the composite of the left vertical and lower horizontal morphisms in diagram~\eqref{dgm: GV functor is LD2}, proving Frobenius relation~\eqref{eq:F1}. Relation~\eqref{eq:F2} follows analogously using Equation~\eqref{eq: GV functor 1} from Lemma~\ref{lemma:duality transfos relations}.
\end{proof}

\smallskip

\begin{proof}[Proof of Lemma~\ref{lemma: morphism of GV is morphism of Frob}]
By the defining equation of \(\upsilon^{2,F}\) and \(\upsilon^{2,G}\) (see Equation~\eqref{eq: def of comultiplication morphism of GV functor} in Construction~\ref{constr: GV to LDN}), the compatibility of \(f\colon F \to G\) with the comultiplication morphisms of \(F\) and \(G\) amounts to the commutativity of the following outer diagram for all \(X,Y\in \cC\):
\begin{equation*}
  \adjustbox{max width=\textwidth}{
  \begin{tikzcd}
	{FD\big(D'(Y)\ot D'(X)\big)} &[8em] {GD\big(D'(Y)\ot D'(X)\big)}\\
    {DF\big(D'(Y)\ot D'(X)\big)} &[8em] {DG\big(D'(Y)\ot D'(X)\big)}\\
    {D\big(FD'(Y)\ot FD'(X)\big)} &[8em] {D\big(GD'(Y)\ot GD'(X)\big)}\\
    {D\big(D'F(Y)\ot D'F(X)\big)} &[8em] {D\big(D'G(Y)\ot D'G(X)\big).}
    \arrow["f"{yshift=1pt}, from=1-1, to=1-2]
    \arrow["\xi^{r,F}"'{yshift=1pt, xshift=-3pt}, from=1-1, to=2-1]
    \arrow["\xi^{r,G}"{yshift=1pt, xshift=3pt}, from=1-2, to=2-2]
    \arrow["D(\varphi^{2,F})"'{yshift=1pt, xshift=-3pt}, from=2-1, to=3-1]
    \arrow["D(\varphi^{2,G})"{yshift=1pt, xshift=3pt}, from=2-2, to=3-2]
    \arrow["D(\xi^{l,F}\ot \xi^{l,F})^{-1}"'{yshift=1pt, xshift=-3pt}, from=3-1, to=4-1]
    \arrow["D(\xi^{l,G}\ot \xi^{l,G})^{-1}"{yshift=1pt, xshift=3pt}, from=3-2, to=4-2]
    \arrow["D(f)"'{yshift=3pt}, from=2-2, to=2-1]
    \arrow["D(f \ot f)"'{yshift=1pt}, from=3-2, to=3-1]
    \arrow["D\big(D'(f)\ot D'(f)\big)"'{yshift=-4pt}, from=4-1, to=4-2]
    \arrow[phantom,"\textup{(I)}"{yshift=2pt}, from=1-1, to=2-2]
    \arrow[phantom,"\textup{(II)}"{yshift=1pt,xshift=2pt}, from=2-1, to=3-2]
    \arrow[phantom,"\textup{(III)}"{yshift=0pt}, from=3-1, to=4-2]
 \end{tikzcd}
}
\end{equation*}
Here, indices are omitted for improved readability. Diagrams (I) and (III) commute by Equation~\eqref{eq:xirF-relation-xirG} in Lemma~\ref{lemma:pre-GV-morph-duality-transfo}, and (II) commutes because \(f\) is monoidal. 

The compatibility of \(f\colon F \to G\) with the counit morphisms \(\upsilon^{0,F}\) and \(\upsilon^{0,G}\) holds by the defining equation of a morphism of GV-functors, namely Equation~\eqref{eq: morphism of GV-functors} of Definition~\ref{def: morphism of GV-functors}.
\end{proof}

\begin{proof}[Proof of Lemma~\ref{prop: GV to LDN is functorial}] The assignment \(\mathsf{GV}\to \mathsf{LDN}\) clearly preserves identity \(1\)-cells. To check that it strictly respects composition, let \(F\colon \cC \to \cD\) and \(G\colon \cD \to \cE\) be GV-functors, and let \(X,Y\in \cC\). By Equation~\eqref{eq: def of comultiplication morphism of GV functor} and the definition of composition of Frobenius LD-functors (see the remark preceding \cite[Prop. 2.23]{DeS}), the comultiplication morphism of the Frobenius LD-functor \(G \circ F\) equals the first morphism in the following calculation:
\allowdisplaybreaks
\begin{align*}
    & {D\big(\xi_{F(Y)}^{l,G}\ot \xi_{F(X)}^{l,G}\big)^{-1}} \circ {D\big(\varphi^{2,G}_{D'F(Y),D'F(X)}\big)}\circ {\xi^{r,G}_{D'F(Y)\ot D'F(X)}}\\ 
    & \; \circ {GD\big(\xi_{Y}^{l,F} \ot \xi_{X}^{l,F}\big)^{-1}} \circ {GD\big(\varphi^{2,F}_{D'(Y),D'(X)}\big)} \circ {G\big(\xi^{r,F}_{D(D'(Y)\ot D'(X))}\big)}\\[.4em]
    &\quad\eqabove{(1)}\;{D\big(\xi_{F(Y)}^{l,G} \ot \xi_{F(X)}^{l,G}\big)^{-1}} \circ {D\big(\varphi^{2,G}_{D'F(Y),D'F(X)}\big)}\circ {DG\big(\xi_{Y}^{l,F} \ot \xi_{X}^{l,F}\big)^{-1}}\\ 
    & \quad \qquad \circ {DG\big(\varphi^{2,F}_{D'(Y),D'(X)}\big)} \circ {\xi^{r,G}_{F(D'(Y)\ot D'(X))}} \circ {G\big(\xi^{r,F}_{D(D'(Y)\ot D'(X))}\big)}\\[.4em]
    &\quad\eqabove{(2)}\;{D\big(\xi_{F(Y)}^{l,G} \ot \xi_{F(X)}^{l,G}\big)^{-1}} \circ {D\big(G(\xi_{Y}^{l,F})\ot G(\xi_{X}^{l,F})\big)^{-1}} \circ  {D\big(\varphi^{2,G}_{FD'(Y),FD'(X)}\big)}\\ 
    & \quad \qquad \circ {DG\big(\varphi^{2,F}_{D'(Y),D'(X)}\big)} \circ {\xi^{r,G}_{F(D'(Y)\ot D'(X))}} \circ {G\big(\xi^{r,F}_{D(D'(Y)\ot D'(X))}\big)}.
\end{align*}
Equation~(1) follows from the naturality of \(\xi^{r,G}\), and Equation~(2) from the naturality of \(\varphi^{2,G}\). By Equation~\eqref{eq: def of comultiplication morphism of GV functor}, the final morphism coincides with the comultiplication of the composite \(G\circ F\) as GV-functors.
\end{proof}

\subsection{From \(\mathsf{LDN}\) to \(\mathsf{GV}\)}\label{sec:LD-cat-to-GV}

\begin{proof}[Proof of Lemma~\ref{prop:LD-categories closed}]
	Given \(X,Y\in \cC\), we define morphisms
	\allowdisplaybreaks
	\begin{align*}
		\operatorname{coev}^{X}_Y\colon& Y \xrightarrow{\lou^{-1}_Y} 1 \otimes Y \xrightarrow{\eta^X \otimes Y} \big(\rD (X) \parLL X\big) \otimes Y \xrightarrow{\distr_{\rD (X),X,Y}} \rD (X) \parLL {(X \otimes Y)} \!\eqdef\! X \multimap (X \otimes Y),\\
		\operatorname{ev}^{X}_Y\colon& X \otimes (X \multimap Y) \!\eqdef\! X \otimes \big(\rD (X) \parLL Y\big) \xrightarrow{\distl_{X,\rD (X),Y}} \big(X \otimes \rD (X)\big) \parLL Y \xrightarrow{\epsilon^X \parLL Y} K \parLL Y \xrightarrow{\lpu_Y} Y.
	\end{align*}
	These are natural in \(Y \in \cC\) and satisfy the snake equations, expressing that the functor \((X \otimes {?})\) is left adjoint to \((X \multimap {?})\). Left LD-duals are treated analogously.
\end{proof}

\vspace{0.1pt}

\begin{proof}[Proof of Proposition~\ref{prop:FrobLDfunctors are closed}]
	We freely use the surface-diagrammatic calculus from \cite{DeS} (see \cite[\S 2]{DeS}): The $\parLL$-monoidal structure is depicted in red, the $\otimes$-monoidal structure in black, opposite categories are shaded light blue, and the functors $D^{\prime}$ are drawn in dark blue. With these conventions, Equation~\eqref{eq: Upsilon=tau} follows from the following calculation:
	\begin{figure}[H]
		\centering
		\includegraphics[width=0.98\textwidth]{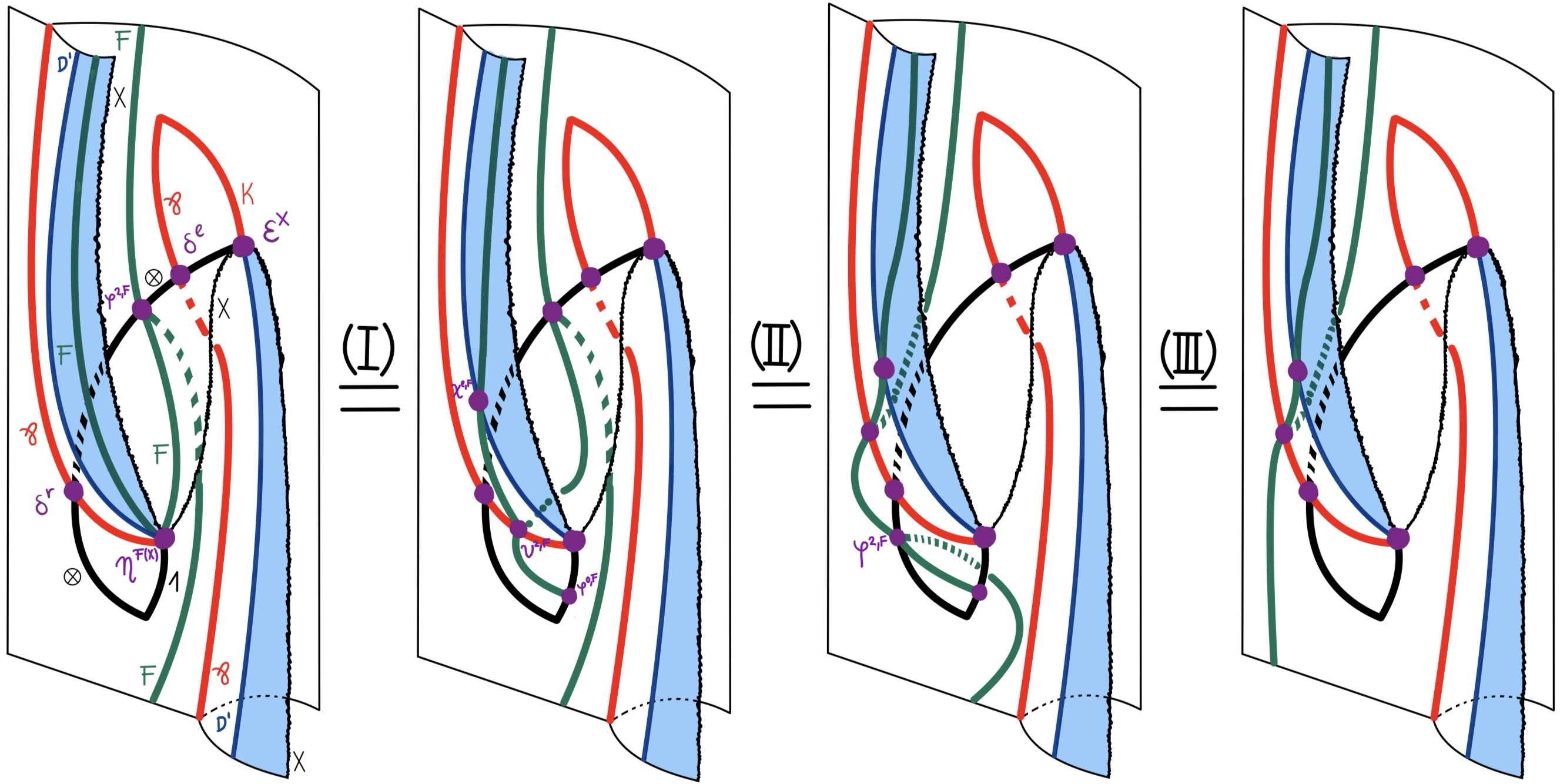}.
	\end{figure}
	The first diagram represents the natural transformation $\tau^{l,F}_{X,?}$, using the definitions of the unit and counit of the adjunction $({X}\otimes {?})\,\dashv\, ({X} \multimap {?})$ in the proof of Lemma~\ref{prop:LD-categories closed}. Equation (I) follows directly from the defining property~\eqref{charact property duality transfo} of $\chi^{l,F}$; (II) from the Frobenius relation~\eqref{eq:F2} for the Frobenius LD-functor $F$ (see Appendix~\ref{app: coherence diagrams} and \cite[Fig. 16]{DeS}); and (III) from the unitality of the lax $\otimes$-monoidal structure $(\varphi^{2,F},\varphi^{0,F})$. Finally, applying the snake equation~\eqref{secondzigzag} for the unit-counit pair $(\eta^X,\epsilon^X)$ (see Appendix~\ref{app: coherence diagrams} and and \cite[Fig. 17]{DeS}) to the rightmost surface diagram yields the natural transformation $\Upsilon^{l,F}_{X,?}$. 
\end{proof}

\vspace{0.1pt}

\begin{proof}[Proof of Lemma~\ref{lemma: LDN to LDN is equivalent to identity}]
	Let \(\cC = (\cC, \otimes, 1, \parLL, K)\) be an LD-category with negation.  
	Its image under the composite \(\mathsf{LDN} \to \mathsf{GV} \to \mathsf{LDN}\) is the LD-category with negation 
	\begin{equation}
		\widetilde{\cC} \;=\; (\cC, \otimes, 1, \widetilde{\parLL}, K),
	\end{equation}
	where, for all $X,Y\in \cC$, we define 
	\begin{equation}
		{X}\, \widetilde{\parLL}\, {Y} \;:=\; D \big(D'(Y)\otimes D'(X)\big),
	\end{equation}
	and the \(\otimes\)-monoidal structure and \(\parLL\)-unit remain unchanged.  
	The identity functor on the underlying category \(\cC\) carries a strong Frobenius LD-structure
	\begin{equation}\label{eq: Frob LD yields natural 2-isomorphism}
		\cC \;\longrightarrow\; \widetilde{\cC},
	\end{equation}
	for which all coherence morphisms are identities, except the \(\parLL\)-comultiplication morphism
	\begin{align}
		&&\upsilon^{\cC}_{X,Y}\colon\; X \parLL Y \;\xlongrightarrow{\;\simeq\;}\; {X} \,\widetilde{\parLL}\, {Y}. &&(X,Y\in \cC)
	\end{align}
	To define \(\upsilon^{\cC}_{X,Y}\), observe that \(D'(Y) \otimes D'(X)\) is the right LD-dual of \(X\parLL Y\) (see Definition~\ref{def:LD-with-negation}): More precisely, the evaluation $\epsilon^{X \parLL Y}\colon (X \parLL Y) \otimes (D'(Y)\otimes D'(X)) \,\longrightarrow\, K$ and the coevaluation \(\eta^{X \parLL Y}\colon 1 \,\longrightarrow\, \big(D'(Y)\otimes D'(X)\big) \parLL {(X \parLL Y)}\) are given by
	\begin{align*}
		\epsilon^{X \parLL Y} &=\, \epsilon^X \circ \Big(\big(\rpu_X \circ (X \parLL \epsilon^Y) \circ \distr \big) \otimes D'(X)\Big) \circ \alpha,\\
		\eta^{X \parLL Y} &=\, \api \circ\Big(\big(\distl \circ \big(D'(Y)\otimes \eta^X\big)\circ (\rho^{-1}_{D'(Y)})\big)\parLL Y\Big)\circ \eta^Y.
	\end{align*}
	Since LD-duals are unique up to unique isomorphism \cite[Lem. A.6]{LD-functors}, this yields an isomorphism \(X \parLL Y \xrightarrow{\simeq} {X} \,\widetilde{\parLL}\, {Y}\). Explicitly, we set
	\begin{align}\label{eq: comultiplication LDN-LDN}
		& \upsilon^{\cC}_{X,Y} \;:=\; \lpu_{X \,\widetilde{\parLL}\, Y} \circ \big(\epsilon^{X\parLL Y} \parLL {(X \, \widetilde{\parLL}\, Y)} \big) \circ \distl \circ \big((X\parLL Y)\otimes {\underline{\eta}}^{D'(Y)\otimes D'(X)}\big)\circ \rho^{-1}_{X\parLL Y}.
	\end{align}
	
	Consider the family of Frobenius LD-equivalences in~\eqref{eq: Frob LD yields natural 2-isomorphism}, indexed by LD-categories with negation $\cC$, equipped with the \(\parLL\)-comultiplications \(\upsilon^{\cC}\). To verify that it defines a strict \(2\)-natural transformation, it suffices to check that for every Frobenius LD-functor \(F\colon \cC \to \cD\) between LD-categories with negation $\cC$ and $\cD$, the comultiplication morphism \(\upsilon^{2,F}\) satisfies
	\begin{equation}\label{eq: Frob LD yields natural 2-isomorphism 2}
		\upsilon^{\cD}_{F(X),F(Y)} \circ \upsilon^{2,F}_{X,Y} \;=\; \widetilde{\upsilon}^{2,F}_{X,Y} \circ F(\upsilon^{\cC}_{X,Y}),
	\end{equation}
	where \(\widetilde{\upsilon}^{2,F}_{X,Y}\) is the comultiplication morphism of the GV-functor associated to $F$, as defined in Equation~\eqref{eq: GV to LDN} of Construction~\ref{constr: GV to LDN}.  The verification of Equation~\eqref{eq: Frob LD yields natural 2-isomorphism 2} relies on repeated use of the naturality and coherence axioms for the structure morphisms of \(F\).
\end{proof}

\smallskip

\subsection{Extension to the braided setting}\label{sec:Braided-GV-categories} 

\begin{proof}[Proof of Lemma~\ref{lemma: tildec ev and coev}]
  For Equation~\eqref{eq: ctilde compat coev}, we compute:
 \allowdisplaybreaks
\begin{align*}
	(\widetilde{c}^{\,\pm}_{Y,X\ot Y})^{-1}\circ \higheroverline{\operatorname{coev}}^Y_X 
	&\;\eqabove{(1)}\; \big(Y \multimap \higheroverline{\operatorname{ev}}^Y_{X \ot Y}\big)\circ \big(Y \multimap c^{\mp}_{Y,(X\ot Y)\multimapinv Y}\big)\circ {{\operatorname{coev}^Y_{(X\ot Y)\multimapinv Y}}\circ {\higheroverline{\operatorname{coev}}^Y_X}}\\
	&\;\eqabove{(2)}\; \big(Y \multimap \higheroverline{\operatorname{ev}}^Y_{X \ot Y}\big)\circ \big(Y \multimap c^{\mp}_{Y,(X\ot Y)\multimapinv Y}\big) \circ \big(Y \multimap (Y\ot \higheroverline{\operatorname{coev}}^Y_X)\big)\circ {\operatorname{coev}^Y_X}\\
	&\;\eqabove{(3)}\; \big(Y \multimap \higheroverline{\operatorname{ev}}^Y_{X \ot Y}\big) \circ \big(Y \multimap (\higheroverline{\operatorname{coev}}^Y_X \ot Y)\big)\circ \big(Y \multimap c^{\mp}_{Y,X}\big)\circ {\operatorname{coev}^Y_X}\\
	&\; \eqabove{(4)}\;  \big(Y \multimap c^{\mp}_{Y,X}\big)\circ \operatorname{coev}^Y_X,
\end{align*} 
where, (1) follows from the definition of \(\widetilde{c}^{\,\pm}\) (Equation~\eqref{def: c tilde}); (2) from naturality of \(\operatorname{coev}^Y\); (3) from naturality of \(c^{\mp}\); and (4) from a snake identity for the adjunction $({?} \otimes Y) \dashv ({?} \multimapinv {Y})$. Equation~\eqref{eq: ctilde compat ev} is proved analogously.
\end{proof}

\medskip

\begin{proof}[Proof of Lemma~\ref{lemma: relationship c overline and c tilde}]
	By Yoneda's Lemma, it suffices to verify that the following diagram
	\begin{equation}\label{dgm: relationship c overline and c tilde2}
		\begin{tikzcd}
			{\operatorname{Hom}_{\cC}\!\big(Z,\,D\big(D'(Y)\otimes D'(X)\big)\big)}&&{\operatorname{Hom}_{\cC}\!\big(Z,\,D\big(D'(X)\otimes D'(Y)\big)\big)}\\
			{\operatorname{Hom}_{\cC}\!\big(Z,\,D'\big(D(Y)\otimes D(X)\big)\big)}&&{\operatorname{Hom}_{\cC}\!\big(Z,\,D\big(D(X)\otimes D'(Y)\big)\big)}
			\arrow["\big(\higheroverline{c}^{\,\pm}_{X,Y}\big)_{\ast}"{yshift=3pt},"\simeq"'{yshift=-1pt}, from=1-1, to=1-3]
			\arrow["\big(\widetilde{c}^{\,\mp}_{D(X),Y}\big)_{\ast}"'{yshift=-3pt},"\simeq"{yshift=1pt}, from=2-1, to=2-3]
			\arrow["D\big(\widetilde{c}^{\,\pm}_{X,K}\otimes D'(Y)\big)_{\ast}"'{xshift=4pt},"\simeq"{xshift=-2pt}, from=2-3, to=1-3]
			\arrow[dash,"\eqref{eq:leftInHom}"'{xshift=-4pt},"\simeq"{xshift=2pt}, from=1-1, to=2-1]
		\end{tikzcd}
	\end{equation}
	 commutes for all \(X,Y,Z\in \cC\). Using the definition of \(\widetilde{c}^{\,\pm}\) (Equation~\eqref{def: c tilde}) and repeated applications of the natural isomorphisms \(\operatorname{Hom}_{\cC}(X \otimes Y,K)\cong \operatorname{Hom}_{\cC}(X,D(Y))\) from Definition~\ref{def:GV-category}, one checks that the commutativity of diagram~\eqref{dgm: relationship c overline and c tilde2} is equivalent to the equation
	 	\begin{equation}\label{dgm: relationship c overline and c tilde3}
	 		\big(c^{\pm}_{D(X),D(Y)}\otimes Z\big)^{\ast} \;=\; \big(c^{\pm}_{D(X),D(Y)\otimes Z}\big)^{\ast} \circ \big(D(Y)\otimes c^{\mp}_{Z,D(X)}\big)^{\ast}
	 \end{equation}
	 as maps 
	 \begin{equation}
	 {\operatorname{Hom}_{\cC}\!\big(D(Y)\otimes D(X)\otimes Z,K\big)}\;\longrightarrow\; {\operatorname{Hom}_{\cC}\!\big(D(X)\otimes D(Y) \otimes Z,K\big)}.
	\end{equation} 
	Associativity constraints are suppressed for readability. Equation~\eqref{dgm: relationship c overline and c tilde3} follows directly from one of the hexagon identities satisfied by the (inverse) braiding \(c^{\pm}\).
\end{proof}

\smallskip

\begin{proof}[Proof of Proposition~\ref{prop: braided GV is braided LDN}]
The two hexagon equations involving only the \(\parLL\)-monoidal structure (those ensuring that \(\higheroverline{c}\) is a braiding) follow directly from the hexagon equations for \(c\), together with the definition of the \(\parLL\)-associator given in Equation~\eqref{def: par associator}.

We now verify the hexagon relation~\eqref{eq: braiding and Frobenius relation 1}; the second relation~\eqref{eq: braiding and Frobenius relation 2} follows analogously. The two diagrams below commute by naturality, together with the definitions of $\distl$ and $\distr$ (Equations~\eqref{eq: def left distributor} and~\eqref{eq: def right distributor} in Remark~\ref{rem:InHoms and distributors}), for all $X,Y,Z\in \cC$: 
\begin{equation}\label{dgm: distr c}
		\begin{tikzcd}
			{X \ot \big(D(Y) \multimap Z\big)}&{\big(D(Y) \multimap Z\big) \ot X}&{D(Y) \multimap (Z \ot X)}\\
			{X \ot (Y \parLL Z)}&{(Y \parLL Z)\ot X}&{Y \parLL {(Z \ot X)},}
			\arrow["c_{X,D(Y)\multimap Z}"{yshift=4pt}, "\simeq"', from=1-1, to=1-2]
			\arrow["\widetilde{\distr}_{D(Y),Z,X}"{yshift=4pt}, from=1-2, to=1-3]
			\arrow["c_{X,Y\parLL Z}"'{yshift=-4pt}, "\simeq", from=2-1, to=2-2]
			\arrow["\distr_{Y,Z,X}"'{yshift=-4pt}, from=2-2, to=2-3]
			\arrow[dash,"\eqref{eq:leftInHom}"'{xshift=-2pt},"\simeq"{xshift=2pt}, from=1-1, to=2-1]
			\arrow[dash,"\eqref{eq:leftInHom}"'{xshift=-2pt},"\simeq"{xshift=2pt}, from=1-3, to=2-3]
			\arrow[dash,"\eqref{eq:leftInHom}"'{xshift=-2pt},"\simeq"{xshift=2pt}, from=1-2, to=2-2]
		\end{tikzcd}
\end{equation}
\begin{equation}\label{dgm: c distl}
		\begin{tikzcd}
			{X \ot (Z \parLL Y)}&{(X \ot Z) \parLL Y}&{(Z \ot X) \parLL Y}\\
			{X \ot \big(Z \multimapinv D'(Y)\big)}&{(X \ot Z)\multimapinv D'(Y)}&{(Z \ot X) \multimapinv D'(Y).}
			\arrow["\distl_{X,Z,Y}"{yshift=2pt}, from=1-1, to=1-2]
			\arrow["c_{X,Z} \parLL Y"{yshift=2pt}, "\simeq"'{yshift=-2pt}, from=1-2, to=1-3]
			\arrow["\widetilde{\distl}_{X,Z,D'(Y)}"'{yshift=-5pt}, from=2-1, to=2-2]
			\arrow["c_{X,Z} \multimapinv D'(Y)"'{yshift=-5pt}, "\simeq"{yshift=2pt}, from=2-2, to=2-3]
			\arrow[dash,"\eqref{eq:rightInHom}"'{xshift=-2pt},"\simeq"{xshift=2pt}, from=1-1, to=2-1]
			\arrow[dash,"\eqref{eq:rightInHom}"'{xshift=-2pt},"\simeq"{xshift=2pt}, from=1-3, to=2-3]
			\arrow[dash,"\eqref{eq:rightInHom}"'{xshift=-2pt},"\simeq"{xshift=2pt}, from=1-2, to=2-2]
		\end{tikzcd}
\end{equation}

Diagrams~\eqref{dgm: distr c} and~\eqref{dgm: c distl}, together with Lemma~\ref{lemma: relationship c overline and c tilde} and the naturality of \(\widetilde{\distl}\), imply that hexagon relation~\eqref{eq: braiding and Frobenius relation 1} is equivalent to the commutativity of the following diagram:

\begin{equation}\label{dgm: F1' InHom}
	\begin{tikzcd}
		{X \ot (Y \multimapinv Z)}&{(X \ot Y) \multimapinv Z}&{(Y \ot X) \multimapinv Z}\\
		{X \ot (Z \multimap Y)}&{(Z \multimap Y)\ot X}&{Z \multimap (Y \ot X),}
		\arrow["\widetilde{\distl}_{X,Y,Z}"{yshift=4pt}, from=1-1, to=1-2]
		\arrow["c_{X,Y} \multimapinv Z"{yshift=4pt}, "\simeq"'{yshift=-2pt}, from=1-2, to=1-3]
		\arrow["c_{X,Z\multimap Y}"'{yshift=-5pt},"\simeq", from=2-1, to=2-2]
		\arrow["\widetilde{\distr}_{Z,Y,X}"'{yshift=-5pt}, from=2-2, to=2-3]
		\arrow["{X} {\ot} {\widetilde{c}^{\,-1}_{Z,Y}}"'{xshift=-4pt},"\simeq"{xshift=2pt}, from=1-1, to=2-1]
		\arrow["\widetilde{c}^{\,-1}_{Z,Y\ot X}"{xshift=4pt},"\simeq"'{xshift=-2pt}, from=1-3, to=2-3]
	\end{tikzcd}
\end{equation}
where \(\widetilde{c}^{\,-1}_{X,Y}:=(\widetilde{c}^{+}_{X,Y})^{-1}\).
We show that Diagram~\eqref{dgm: F1' InHom} commutes by rewriting the composite of the top horizontal and right vertical arrows:
\allowdisplaybreaks
\begin{align*}
	& \widetilde{c}^{\,-1}_{Z,Y \ot X}\circ (c_{X,Y}\multimapinv Z)\circ \widetilde{\distl}_{X,Y,Z}\\[.4em]
	& \;\eqabove{(1)}\; \widetilde{c}^{\,-1}_{Z,Y \ot X}\circ (c_{X,Y}\multimapinv Z)\circ \big((X\ot \higheroverline{\operatorname{ev}}^Z_Y)\multimapinv Z\big)\circ \big(\aoi_{X,Y\multimapinv Z,Z}\multimapinv Z\big) \circ \higheroverline{\operatorname{coev}}^Z_{X\ot (Y \multimapinv Z)}\\[.4em]
	& \;\eqabove{(2)}\; \widetilde{c}^{\,-1}_{Z,Y \ot X}\circ {(c_{X,Y}\multimapinv Z)}\circ {\big((X\ot \operatorname{ev}^Z_Y)\multimapinv Z\big)}\\ 
	&\qquad \circ \big((X\ot (Z \ot \widetilde{c}^{\,-1}_{Z,Y}))\multimapinv Z\big) \circ \big((X\ot c_{Y \multimapinv Z,Z})\multimapinv Z\big)\circ \big(\aoi_{X,Y\multimapinv Z,Z}\multimapinv Z\big) \circ \higheroverline{\operatorname{coev}}^Z_{X\ot (Y \multimapinv Z)}\\[.4em]
	& \;\eqabove{(3)}\; \widetilde{c}^{\,-1}_{Z,Y \ot X}\circ {\big(({\operatorname{ev}^Z_Y}\ot X)\multimapinv Z\big)} \circ {(c_{X,Z\ot (Z \multimap Y)}\multimapinv Z)}\\ 
	&\qquad \circ \big((X\ot (Z \ot \widetilde{c}^{\,-1}_{Z,Y}))\multimapinv Z\big) \circ \big((X\ot c_{Y \multimapinv Z,Z})\multimapinv Z\big)\circ \big(\aoi_{X,Y\multimapinv Z,Z}\multimapinv Z\big) \circ \higheroverline{\operatorname{coev}}^Z_{X\ot (Y \multimapinv Z)}\\[.4em]
	& \;\eqabove{(4)}\;{\big(Z \multimap ({\operatorname{ev}^Z_Y}\ot X)\big)} \circ \widetilde{c}^{\,-1}_{Z,(Z\ot (Z \multimap Y)) \ot X} \circ {(c_{X,Z\ot (Z \multimap Y)}\multimapinv Z)}\\ 
	&\qquad \circ \big((X\ot (Z \ot \widetilde{c}^{\,-1}_{Z,Y}))\multimapinv Z\big) \circ \big((X\ot c_{Y \multimapinv Z,Z})\multimapinv Z\big)\circ \big(\aoi_{X,Y\multimapinv Z,Z}\multimapinv Z\big) \circ \higheroverline{\operatorname{coev}}^Z_{X\ot (Y \multimapinv Z)}\\[.4em]
	& \;\eqabove{(5)}\;{\big(Z \multimap ({\operatorname{ev}^Z_Y}\ot X)\big)} \circ {(Z \multimap c_{X,Z\ot (Z \multimap Y)})} \circ \widetilde{c}^{\,-1}_{Z,X \ot (Z\ot (Z \multimap Y))}\\ 
	&\qquad \circ \big((X\ot (Z \ot \widetilde{c}^{\,-1}_{Z,Y}))\multimapinv Z\big) \circ \big((X\ot c_{Y \multimapinv Z,Z})\multimapinv Z\big)\circ \big(\aoi_{X,Y\multimapinv Z,Z}\multimapinv Z\big) \circ \higheroverline{\operatorname{coev}}^Z_{X\ot (Y \multimapinv Z)}\\[.4em]
	& \;\eqabove{(6)}\;{\big(Z \multimap ({\operatorname{ev}^Z_Y}\ot X)\big)} \circ {(Z \multimap c_{X,Z\ot (Z \multimap Y)})} \circ \widetilde{c}^{\,-1}_{Z,X \ot (Z\ot (Z \multimap Y))}\\ 
	& \qquad \circ \big((X\ot c_{Z \multimap Y,Z})\multimapinv Z\big)\circ \big((X\ot (\widetilde{c}^{\,-1}_{Z,Y} \ot Z))\multimapinv Z\big) \circ \big(\aoi_{X,Y\multimapinv Z,Z}\multimapinv Z\big) \circ \higheroverline{\operatorname{coev}}^Z_{X\ot (Y \multimapinv Z)}\\[.4em]
	& \;\eqabove{(7)}\;{\big(Z \multimap ({\operatorname{ev}^Z_Y}\ot X)\big)} \circ {(Z \multimap c_{X,Z\ot (Z \multimap Y)})} \circ \big(Z\multimap (X\ot c_{Z \multimap Y,Z})\big)\\ 
	& \qquad \circ \widetilde{c}^{\,-1}_{Z,X \ot ((Z \multimap Y)\ot Z)} \circ \big((X\ot (\widetilde{c}^{\,-1}_{Z,Y} \ot Z))\multimapinv Z\big) \circ \big(\aoi_{X,Y\multimapinv Z,Z}\multimapinv Z\big) \circ \higheroverline{\operatorname{coev}}^Z_{X\ot (Y \multimapinv Z)}\\[.4em]
	& \;\eqabove{(8)}\;{\big(Z \multimap ({\operatorname{ev}^Z_Y}\ot X)\big)} \circ {(Z \multimap (c_{Z \multimap Y,Z}\ot X))} \circ (Z\multimap c_{X, (Z \multimap Y)\ot Z})\\ 
	& \qquad \circ \widetilde{c}^{\,-1}_{Z,X \ot ((Z \multimap Y)\ot Z)} \circ \big((X\ot (\widetilde{c}^{\,-1}_{Z,Y} \ot Z))\multimapinv Z\big) \circ \big(\aoi_{X,Y\multimapinv Z,Z}\multimapinv Z\big) \circ \higheroverline{\operatorname{coev}}^Z_{X\ot (Y \multimapinv Z)}\\[.4em]
	& \;\eqabove{(9)}\;{\big(Z \multimap ({\operatorname{ev}^Z_Y}\ot X)\big)} \circ {(Z \multimap (c_{Z \multimap Y,Z}\ot X))} \circ \widetilde{c}^{\,-1}_{Z,((Z \multimap Y)\ot Z)\ot X}\\ 
	& \qquad \circ (c_{X,(Z\multimap Y)\ot Z}\multimapinv Z)\circ \big((X\ot (\widetilde{c}^{\,-1}_{Z,Y} \ot Z))\multimapinv Z\big) \circ \big(\aoi_{X,Y\multimapinv Z,Z}\multimapinv Z\big) \circ \higheroverline{\operatorname{coev}}^Z_{X\ot (Y \multimapinv Z)}.
\end{align*} 
Equation (1) follows from Equation~\eqref{eq: distl via inhoms} in Remark~\ref{rem:InHoms and distributors}; (2) from Equation~\eqref{eq: ctilde compat ev} in Lemma~\ref{lemma: tildec ev and coev}; (3) from the naturality of the braiding \(c\); equations (4)--(7) and (9) from the naturality of \(\widetilde{c}^{\,-1}\); and (8) from the naturality of \(c\).

Next, we rewrite the composite of the left vertical and bottom horizontal arrows in Diagram~\eqref{dgm: F1' InHom} as follows:
\allowdisplaybreaks
\begin{align*}
	& \widetilde{\distr}_{Z,Y,X}\circ c_{X,Z\multimap Y}\circ (X \ot {\widetilde{c}^{\,-1}_{Z,Y}})\\[.4em]
	& \;\eqabove{(1)}\; \big(Z\multimap ({\operatorname{ev}^Z_Y} \ot X)\big)\circ \big(Z \multimap \ao_{Z,Z\multimap Y,X}\big) \circ {\operatorname{coev}^Z_{(Z \multimap Y)\ot X}} \circ {c_{X,Z\multimap Y}}\circ (X \ot {\widetilde{c}^{\,-1}_{Z,Y}})\\[.4em]
	& \;\eqabove{(2)}\; \big(Z\multimap ({\operatorname{ev}^Z_Y} \ot X)\big)\circ \big(Z \multimap \ao_{Z,Z\multimap Y,X}\big) \circ {(Z \multimap c_{(Z \multimap Y)\ot X,Z})}\\ 
	& \qquad \circ \widetilde{c}^{\,-1}_{Z,((Z \multimap Y)\ot X)\ot Z}\circ \higheroverline{\operatorname{coev}}^Z_{(Z \multimap Y)\ot X} \circ {c_{X,Z\multimap Y}}\circ (X \ot {\widetilde{c}^{\,-1}_{Z,Y}})\\[.4em]
	& \;\eqabove{(3)}\; \big(Z\multimap ({\operatorname{ev}^Z_Y} \ot X)\big)\circ \big(Z \multimap \ao_{Z,Z\multimap Y,X}\big) \circ {(Z \multimap c_{(Z \multimap Y)\ot X,Z})}\\ 
	& \qquad \circ \widetilde{c}^{\,-1}_{Z,((Z \multimap Y)\ot X)\ot Z}\circ \big((c_{X,Z\multimap Y} \ot Z)\multimapinv Z\big) \circ \big(((X\ot \widetilde{c}^{\,-1}_{Z,Y}) \ot Z)\multimapinv Z\big) \circ \higheroverline{\operatorname{coev}}^Z_{X \ot (Y \multimapinv Z)}
	\\[.4em]
	& \;\eqabove{(4)}\; \big(Z\multimap ({\operatorname{ev}^Z_Y} \ot X)\big)\circ (Z\multimap (c_{Z \multimap Y,Z}\ot X))\\
	& \qquad \circ {(Z\multimap \ao_{Z\multimap Y,Z,X})} \circ {\big(Z \multimap ((Z\multimap Y)\ot c_{X,Z})\big)} \circ {(Z\multimap \aoi_{Z\multimap Y,X,Z})}\\ 
	& \qquad \circ \widetilde{c}^{\,-1}_{Z,((Z \multimap Y)\ot X)\ot Z}\circ \big((c_{X,Z\multimap Y} \ot Z)\multimapinv Z\big) \circ \big(((X\ot \widetilde{c}^{\,-1}_{Z,Y}) \ot Z)\multimapinv Z\big) \circ \higheroverline{\operatorname{coev}}^Z_{X \ot (Y \multimapinv Z)}\\[.4em]
	& \;\eqabove{(5)}\; \big(Z\multimap ({\operatorname{ev}^Z_Y} \ot X)\big)\circ (Z\multimap (c_{Z \multimap Y,Z}\ot X))\\
	& \qquad \circ {\widetilde{c}^{\,-1}_{Z,((Z\multimap Y)\ot Z)\ot X}} \circ {(\ao_{Z\multimap Y,Z,X}\multimapinv Z)} \circ {\big(((Z\multimap Y)\ot c_{X,Z})\multimapinv Z \big)}\\ 
	& \qquad \circ {(\aoi_{Z\multimap Y,X,Z} \multimapinv Z)} \circ \big((c_{X,Z\multimap Y} \ot Z)\multimapinv Z\big) \circ \big(((X\ot \widetilde{c}^{\,-1}_{Z,Y}) \ot Z)\multimapinv Z\big) \circ \higheroverline{\operatorname{coev}}^Z_{X \ot (Y \multimapinv Z)}\\[.4em]
	& \;\eqabove{(6)}\; \big(Z\multimap ({\operatorname{ev}^Z_Y} \ot X)\big)\circ (Z\multimap (c_{Z \multimap Y,Z}\ot X))\circ \widetilde{c}^{\,-1}_{Z,((Z\multimap Y)\ot Z)\ot X} \circ {(c_{X,(Z\multimap Y)\ot Z} \multimapinv Z)}\\ 
	& \qquad \circ {(\aoi_{X,Z\multimap Y, Z} \multimapinv Z)} \circ \big(((X\ot \widetilde{c}^{\,-1}_{Z,Y}) \ot Z)\multimapinv Z\big) \circ \higheroverline{\operatorname{coev}}^Z_{X \ot (Y \multimapinv Z)}\\[.4em]
	& \;\eqabove{(7)}\;{\big(Z \multimap ({\operatorname{ev}^Z_Y}\ot X)\big)} \circ {(Z \multimap (c_{Z \multimap Y,Z}\ot X))} \circ \widetilde{c}^{\,-1}_{Z,((Z \multimap Y)\ot Z)\ot X}\\ 
	& \qquad \circ (c_{X,(Z\multimap Y)\ot Z}\multimapinv Z)\circ \big((X\ot (\tilde{c}^{-1}_{Z,Y} \ot Z))\multimapinv Z\big) \circ \big(\aoi_{X,Y\multimapinv Z,Z}\multimapinv Z\big) \circ \higheroverline{\operatorname{coev}}^Z_{X\ot (Y \multimapinv Z)}.
\end{align*}
Equation (1) follows from Equation~\eqref{eq: distr via inhoms} in Remark~\ref{rem:InHoms and distributors}; (2) from Equation~\eqref{eq: ctilde compat coev} in Lemma~\ref{lemma: tildec ev and coev}; (3) from the naturality of \(\higheroverline{\operatorname{coev}}^Z\); (4) from one of the hexagon axioms for the braiding \(c\); (5) from the naturality of \(\tilde{c}\); (6) from the other hexagon axiom for \(c\); and (7) from the naturality of the inverse associator \(\aoi\).

Comparing the final lines of both computations shows that Diagram~\eqref{dgm: F1' InHom} commutes.
\end{proof}

\begin{proof}[Proof of Proposition~\ref{prop: braided GVf is braided LDN f}]
	The claim follows from the commutativity of the diagram
	\begin{equation*}
		\adjustbox{max width=\textwidth}{
			\begin{tikzcd}
				{FD(D'(Y)\ot D'(X))} &[8em] {FD(D'(X)\ot D'(Y))}\\
				{DF(D'(Y)\ot D'(X))} &[8em] {DF(D'(X)\ot D'(Y))}\\
				{D(FD'(Y)\ot FD'(X))} &[8em] {D(FD'(X)\ot FD'(Y))}\\
				{D(D'F(Y)\ot D'F(X))} &[8em] {D(D'F(X)\ot D'F(Y)),}
				\arrow["FD(c)"{yshift=3pt}, from=1-1, to=1-2]
				\arrow["\xi^{r,F}"'{yshift=1pt, xshift=-4pt}, from=1-1, to=2-1]
				\arrow["\xi^{r,F}"{yshift=1pt, xshift=4pt}, from=1-2, to=2-2]
				\arrow["D(\varphi^{2,F})"'{yshift=1pt, xshift=-4pt}, from=2-1, to=3-1]
				\arrow["D(\varphi^{2,F})"{yshift=1pt, xshift=4pt}, from=2-2, to=3-2]
				\arrow["D(\xi^{l,F}\ot\, \xi^{l,F})^{-1}"'{yshift=1pt, xshift=-4pt}, from=3-1, to=4-1]
				\arrow["D(\xi^{l,F}\ot\, \xi^{l,F})^{-1}"{yshift=1pt, xshift=4pt}, from=3-2, to=4-2]
				\arrow["DF(c)"{yshift=0.2pt}, from=2-1, to=2-2]
				\arrow["D(c)"{yshift=0.2pt}, from=3-1, to=3-2]
				\arrow["D(c)"'{yshift=-4pt}, from=4-1, to=4-2]
				\arrow[phantom,"\textup{(I)}"{yshift=2pt}, from=1-1, to=2-2]
				\arrow[phantom,"\textup{(II)}"{yshift=1pt,xshift=2pt}, from=2-1, to=3-2]
				\arrow[phantom,"\textup{(III)}"{yshift=0pt}, from=3-1, to=4-2]
			\end{tikzcd}
		}
	\end{equation*}
	for all \(X,Y\in \cC\). Diagram (I) commutes by the naturality of \(\xi^{r,F}\); (II) by the naturality of the $\otimes$-braiding \(c\); and (III) because \(F\) is a braided lax \(\ot\)-monoidal functor.
\end{proof}

\begin{proof}[Proof of Lemma~\ref{lemma: BrLDN to BrLDN is equivalent to identity}]
	We adopt the notation from the proof of Lemma~\ref{lemma: LDN to LDN is equivalent to identity}. It remains to verify that each component of the strict \(2\)-natural isomorphism~\eqref{eq: Frob LD yields natural 2-isomorphism} constructed in that proof is a braided Frobenius LD-functor. By the definition of the Frobenius LD-structure on these components, this reduces to showing that, for any braided LD-category with negation \(\cC\),
	\begin{align}\label{eq: identity is braided}
	 \upsilon^{\cC}_{Y,X} \circ \higheroverline{c}_{X,Y} &= D\big(c_{D'(X),D'(Y)}\big)\circ \upsilon^{\cC}_{X,Y},
	\end{align}
	where \(\higheroverline{c}\) is the \(\parLL\)-braiding from Construction \ref{constr: BrGV to BrLDN}. Unwinding the definition of \(\upsilon^{\cC}\) (see Equation~\eqref{eq: comultiplication LDN-LDN}), Equation~\eqref{eq: identity is braided} follows directly from the naturality of the unitors \(\rou\), \(\lpu\), the naturality of the distributor \(\distl\), and the extranaturality of the (co)evaluations \(\underline{\eta}\) and \(\epsilon\).
\end{proof}

\begin{proof}[Proof of Proposition~\ref{prop: Drinfeld iso is morphism of Frob LD}]
	By Theorem~\ref{thm: GV equiv LDN}, it suffices to show that the transformations $\varphi^{\pm}$ are morphisms of GV-functors. According to Remark~\ref{rem: comparing duality transfos in braided GV} and \cite[Prop.~6.10]{BoDrinfeld}, they are monoidal, where \(D'\circ J^{\pm 1}\) and \(D\) are viewed as monoidal functors $(\cC,\otimes,1)\rightarrow (\cC^{\operatorname{op}},\parLL^{\operatorname{rev}},K).$
	
	Using the Frobenius forms on $D$ and $D'$ described in Example~\ref{ex: duality functors are GV}, the defining condition~\eqref{eq: morphism of GV-functors} from Definition~\ref{def: morphism of GV-functors} for $\varphi^{\pm}$ is equivalent to the identity
	\begin{align}\label{eq: morph GV-functors Drinfeld iso. counit condition}
		\varphi_K^{\pm} \circ (\higheroverline{\gamma}^{-1}_K\multimap K) \circ d_1 &\;=\; (K \multimapinv \gamma^{-1}_K) \circ \widetilde{d}_1,
	\end{align}
	where \(d\) and \(\widetilde{d}\) are the unit and counit defined in Remark~\ref{rem: double dualization morphism}. By the definitions of \(\gamma\) and \(\higheroverline{\gamma}\),
	\begin{align}\label{eq: morph GV-functors Drinfeld iso. counit condition3}
		\higheroverline{\operatorname{ev}}_K^1 \;=\; \higheroverline{\gamma}_K \circ \rho_{K\multimapinv 1}
		\qquad
		\text{and}
		\qquad
		\operatorname{ev}_K^1 \;=\; \gamma_K \circ \lambda_{1\multimap K}.
	\end{align}
	
	Unpacking the definitions of \(d\) and \(\widetilde{d}\), and using Equation~\eqref{eq: morph GV-functors Drinfeld iso. counit condition3} together with the extranaturality of \(\operatorname{coev}\) and \(\higheroverline{\operatorname{coev}}\), we see that Equation~\eqref{eq: morph GV-functors Drinfeld iso. counit condition} is equivalent to 
	\begin{align}\label{eq: morph GV-functors Drinfeld iso. counit condition2}
		\varphi_K^{\pm} \circ (K\multimap \rho_K) \circ \operatorname{coev}_1^K &\;=\; (\lambda_K \multimapinv K) \circ \higheroverline{\operatorname{coev}}_1^K.
	\end{align}
	
	Finally, by the definition of \(\varphi^{\pm}\) and the naturality of the adjunction isomorphisms~\eqref{AdjMor1} and~\eqref{AdjMor2}, Equation~\eqref{eq: morph GV-functors Drinfeld iso. counit condition2} follows from the compatibility of the braidings \(c^{\pm}\) with the unitors, namely from the equation \(\rho_K \circ c^{\pm}_{1,K}=\lambda_K.\)
\end{proof}

\subsection{Algebras, bimodules, and local modules}\label{sec:Bimodules-GV-categories} 

\begin{proof}[Proof of Lemma~\ref{lemma: rewrite inHom action}]
	Omitting associators, we compute:
	\allowdisplaybreaks
	\begin{align*}
		& {\operatorname{ev}^A_{M\multimap N}}\circ {\big(A\otimes \beta_{M,A,N}\big)}\circ {\big(A\otimes (r^M\multimap N)\big)}\\
		&\qquad\eqabove{(1)}\; {\operatorname{ev}^A_{M\multimap N}}\circ{\big(A\otimes\big(A\multimap (M\multimap \operatorname{ev}^{M\otimes A}_N)\big)\big)} \circ {\big(A\otimes \big(A\multimap \operatorname{coev}^M_{A\otimes((M\otimes A)\multimap N)}\big)\big)}
		\\& \qquad \qquad \circ {\big(A\otimes \operatorname{coev}^A_{(M\otimes A)\multimap N}\big)}\circ {\big(A\otimes (r^M\multimap N)\big)}\\
		&\qquad\eqabove{(2)}\; {\big(M\multimap \operatorname{ev}^{M\otimes A}_N\big)} \circ {\operatorname{coev}^M_{A\otimes ((M\otimes A)\multimap N)}} \circ {\operatorname{ev}^A_{A\otimes((M\otimes A)\multimap N)}} 
		\\& \qquad \qquad \circ {\big(A\otimes \operatorname{coev}^A_{(M\otimes A)\multimap N}\big)}\circ {\big(A\otimes (r^M\multimap N)\big)}\\
		&\qquad\eqabove{(3)}\; {\big(M\multimap \operatorname{ev}^{M\otimes A}_N\big)} \circ {\operatorname{coev}^M_{A\otimes ((M\otimes A)\multimap N)}} \circ {\big(A\otimes (r^M \multimap N)\big)}\\
		&\qquad\eqabove{(4)}\; {\operatorname{comp}^l_{M,M\otimes A,N}}\circ {\big({\operatorname{coev}^M_A} \otimes ((M\otimes A)\multimap N)\big)} \circ {\big(A\otimes (r^M\multimap N)\big)}\\
		&\qquad\eqabove{(5)}\; {\operatorname{comp}^l_{M,M\otimes A,N}} \circ {\big((M\multimap (M\otimes A))\otimes (r^M\multimap N)\big)} \circ {\big({\operatorname{coev}^M_A} \otimes (M\multimap N)\big)}\\
		&\qquad\eqabove{(6)}\; {\operatorname{comp}^l_{M,M,N}} \circ {\big((M\multimap r^M)\otimes (M\multimap N)\big)} \circ {\big({\operatorname{coev}^M_A} \otimes (M\multimap N)\big)}\\
		&\qquad\eqabove{(7)}\; l^{M\multimap N}.
	\end{align*} 
	Equation (1) follows from the definition of $\beta$ (Remark~\ref{rem:canonical isos}); (2) from the naturality of $\operatorname{ev}^A$; (3) from a snake equation for $\operatorname{coev}^A$ and $\operatorname{ev}^A$; (4) from the snake equation for $\operatorname{coev}^M$ and $\operatorname{ev}^M$, together with the definition of internal composition in Remark~\ref{rem:internal composition}; (5) from the functoriality of $\otimes$; (6) from Lemma~\ref{lem:internal composition is extranatural}; and (7) from the definition of $\underline{r}^M$ (Equation~\eqref{eq: r underline morphism}) and $l^{M\multimap N}$ (Equation~\eqref{eq: left action on internal hom}).
\end{proof}

\medskip

\begin{proof}[Proof of Lemma~\ref{lemma: rewrite coreflexive pair}]
Omitting associators, we compute:
\allowdisplaybreaks
\begin{align*}
	& {\big((A\otimes M)\multimap l^{L\multimap N}\big)} 
	\circ {\underline{A\otimes}_{M,L\multimap N}}
	\circ \beta_{L,M,N}
	\\
	&\qquad\eqabove{(1)}\; 
	{\big((A\otimes M)\multimap l^{L\multimap N}\big)}
	\circ {\big((A\otimes M)\multimap (A \otimes \operatorname{ev}^M_{L\multimap N})\big)} 
	\circ {\operatorname{coev}^{A\otimes M}_{M\multimap (L \multimap N)}}
	\circ \beta_{L,M,N}
	\\
	&\qquad\eqabove{(2)}\; 
	{\big((A\otimes M)\multimap ({\operatorname{ev}^A_{L\multimap N}}
	\circ {(A\otimes \beta_{L,A,N})})\big)} 
	\circ {\big((A\otimes M)\multimap (A\otimes (r^L\multimap N))\big)}
	\\
	&\qquad \qquad 
	\circ {\big((A\otimes M)\multimap (A \otimes \operatorname{ev}^M_{L\multimap N})\big)} 
	\circ {\operatorname{coev}^{A\otimes M}_{M\multimap (L \multimap N)}}\circ \beta_{L,M,N}
	\\
	&\qquad\eqabove{(3)}\;
	{\big((A\otimes M)\multimap {\operatorname{ev}^A_{L\multimap N}}\big)} 
	\circ {\big((A\otimes M)\multimap (A \otimes \operatorname{ev}^{M}_{A\multimap (L\multimap N)})\big)}\\
	&\qquad \qquad 
	\circ {\operatorname{coev}^{A\otimes M}_{M\multimap (A\multimap (L\multimap N))}} \circ {(M\multimap \beta_{L,A,N})} 
	\circ {\big(M \multimap (r^L\multimap N)\big)} 
	\circ \beta_{L,M,N}
	\\
	&\qquad\eqabove{(4)}\;
	{\big((A\otimes M)\multimap {\operatorname{ev}^A_{L\multimap N}}\big)} 
	\circ {\big((A\otimes M)\multimap (A \otimes \operatorname{ev}^{M}_{A\multimap (L\multimap N)})\big)}\\
	&\qquad \qquad 
	\circ {\operatorname{coev}^{A\otimes M}_{M\multimap (A\multimap (L\multimap N))}} 
	\circ {(M\multimap \beta_{L,A,N})} 
	\circ {\beta_{L\otimes A,M,N}} 
	\circ {\big((r^L\otimes M) \multimap N)\big)}
	\\
	&\qquad\eqabove{(5)}\; 
	{\big((A\otimes M)\multimap {\operatorname{ev}^A_{L\multimap N}}\big)} 
	\circ {\big((A\otimes M)\multimap (A \otimes \operatorname{ev}^{M}_{A\multimap (L\multimap N)})\big)}\\
	&\qquad \qquad 
	\circ {\operatorname{coev}^{A\otimes M}_{M\multimap (A\multimap (L\multimap N))}} 
	\circ {\beta_{A,M,L \multimap N}} 
	\circ {\beta_{L,A \otimes M,N}} 
	\circ {\big((r^L\otimes M) \multimap N)\big)}
	\\
	&\qquad\eqabove{(6)}\;
	{\big((A\otimes M)\multimap ({\operatorname{ev}^A_{L\multimap N}} 
	\circ (A \otimes \operatorname{ev}^{M}_{A\multimap (L\multimap N)}))\big)}
	\circ {\big((A \otimes M) \multimap (A\otimes M \otimes \beta_{A,M,L \multimap N})\big)}
	\\
	&\qquad \qquad \circ {\operatorname{coev}^{A\otimes M}_{(A\otimes M)\multimap (L\multimap N)}}  
	\circ {\beta_{L,A \otimes M,N}} 
	\circ {\big((r^L\otimes M) \multimap N)\big)}
	\\
	&\qquad\eqabove{(7)}\;
	{\big((A\otimes M)\multimap \operatorname{ev}^{A\otimes M}_{L\multimap N}\big)}
	\circ {\operatorname{coev}^{A\otimes M}_{(A\otimes M)\multimap (L\multimap N)}}  
	\circ {\beta_{L,A \otimes M,N}} 
	\circ {\big((r^L\otimes M) \multimap N)\big)}
	\\
	&\qquad\eqabove{(8)}\;
	{\beta_{L,A \otimes M,N}} 
	\circ {\big((r^L\otimes M) \multimap N)\big)}.
\end{align*} 
Equation (1) follows from the definition of $\underline{A\otimes}_{M,L\multimap N}$ (Remark~\ref{rem: internal hom tensorality}); (2) from Equation~\eqref{eq: rewrite inHom action} in Lemma~\ref{lemma: rewrite inHom action}; (3) from the naturality of $\operatorname{ev}^M$ and $\operatorname{coev}^{A \otimes M}$; (4) from the naturality of \(\beta\); (5) from the pentagon diagram for the associator $\alpha$, together with the definition of $\beta$ (Remark \ref{rem:canonical isos}); (6) from the naturality of \(\operatorname{coev}^{A \otimes M}\); (7) from Equation~\eqref{eq: ev compatible with beta} in Lemma~\ref{lemma: beta  compatibility}; and (8) from a snake equation for \(\operatorname{coev}^{A \otimes M}\) and \(\operatorname{ev}^{A\otimes M}\).
\end{proof}

\smallskip

\begin{proof}[Proof of Proposition~\ref{prop: bimodule structure on dual of algebra}]
	To simplify notation, we suppress associators. By Lemma~\ref{lemma: rewrite inHom action}, it suffices to show that the following two composites $A\otimes D'(A) \otimes A\to D'(A)$ coincide: 
	\begin{align}
		{\operatorname{ev}_{D'(A)}^A}\circ {\big(A\otimes \big(\beta_{A,A,K} \circ D'(\mu)\circ f^{-1}\circ \higheroverline{\operatorname{ev}}_{D(A)}^A\big)\big)}\circ {\big(A\otimes \big(\higheroverline{\beta}_{A,A,K}\circ D(\mu)\circ f\big) \otimes A\big)},\label{pf:bimodule LHS}\\
		f^{-1}\circ {\higheroverline{\operatorname{ev}}_{D(A)}^A}\circ {\big( \big(\higheroverline{\beta}_{A,A,K} \circ D(\mu)\circ f\circ \operatorname{ev}_{D'(A)}^A\big)\otimes A\big)}\circ {\big(A\otimes \big({\beta}_{A,A,K}\circ D'(\mu)\big) \otimes A\big)}.\label{pf:bimodule RHS}
	\end{align}
	By the fact that $f$ is $\parLL$-comultiplicative (Equation~\eqref{def:f is comultiplicative}) and the naturality of $\higheroverline{\operatorname{ev}}^A$, the composite~\eqref{pf:bimodule LHS} coincides with the following morphism, where indices are omitted for readability:
	\begin{align}
		\begin{split}
			&{\operatorname{ev}_{D'(A)}^A} \circ {\big(A\otimes \big(\beta \circ D'(\mu)\circ \higheroverline{\operatorname{ev}}^A_{D'(A)}\big)\big)}\circ {\big(A\otimes \big(\iota_{A,K,A}^{-1} \circ (A\multimap f)\circ {\beta}\circ D'(\mu)\big) \otimes A\big)}.\label{eq: bimod LHS2}
		\end{split}
	\end{align}
	
	A similar argument shows that the composite~\eqref{pf:bimodule RHS} agrees with
	\begin{align}
		\begin{split}
			&{\higheroverline{\operatorname{ev}}_{D'(A)}^A}\circ {\big(\big(\iota_{A,K,A}^{-1}\circ (A\otimes f)\circ \beta\circ D'(\mu)\circ \operatorname{ev}_{D'(A)}^A\big)\otimes A\big)} \circ {\big(A\otimes \big({\beta}\circ D'(\mu)\big) \otimes A\big)}.\label{eq: bimod RHS2}
		\end{split}
	\end{align}
	
	We now show that the morphisms~\eqref{eq: bimod LHS2} and~\eqref{eq: bimod RHS2} coincide. Repeated applications of the naturality of $\beta$ and $\operatorname{ev}^A$, together with the associativity of $\mu$, identify~\eqref{eq: bimod RHS2} with:
	\begin{align}
		\begin{split}
			&{\higheroverline{\operatorname{ev}}}\circ \Big(\!\operatorname{ev}\otimes A\Big) 
			\circ {\Big(A\otimes \Big(\big(A\multimap(\iota^{-1}\circ (A\multimap f)\circ \beta)\big)\circ {\beta}\circ D'(\mu\otimes A)\circ D'(\mu)\Big) \otimes A\Big)}.\label{eq: bimod RHS3}
		\end{split}
	\end{align}
	
	By naturality of $\beta$ and functoriality of the right internal hom $\multimap$, composite \eqref{eq: bimod RHS3} equals
	\begin{align}
		\begin{split}
			&{\higheroverline{\operatorname{ev}}}\circ \Big(\!\operatorname{ev}\otimes A\Big) 
			\circ\Big(A\otimes\big((A\multimap\iota^{-1})\circ \beta\circ (\mu\multimap f)\circ \beta\circ D'(\mu)\big) \otimes A\Big).\label{eq: bimod RHS4}
		\end{split}
	\end{align}
	
Using Equation~\eqref{eq: ev compatible with iota} from Lemma~\ref{lemma: iota compatibility}, the morphism~\eqref{eq: bimod RHS4} can be rewritten as
	
	\begin{align}
		\begin{split}
			&{\operatorname{ev}}\circ \Big(A\otimes \higheroverline{\operatorname{ev}}\Big) 
			\circ\Big(A\otimes\big(\iota^{-1}\circ(A\multimap\iota^{-1})\circ \beta\circ (\mu\multimap f)\circ \beta\circ D'(\mu)\big) \otimes A\Big).\label{eq: bimod RHS5}
		\end{split}
	\end{align}
	Finally, Equation~\eqref{eq: pentagon iota 1} of Lemma~\ref{lemma: iota compatibility}, together with the naturality of $\higheroverline{\operatorname{ev}}^A$ and $\iota^{-1}$, shows that \eqref{eq: bimod RHS5} (and thus \eqref{eq: bimod RHS2}) coincides with \eqref{eq: bimod LHS2}. This proves that the composites \eqref{pf:bimodule LHS} and \eqref{pf:bimodule RHS} are equal.
\end{proof}

\smallskip

\begin{proof}[Proof of Lemma~\ref{lemma: coalgebra iso via braiding}]
Part (i) follows directly from Proposition~\ref{prop: Drinfeld iso is morphism of Frob LD}.

\smallskip

For part (ii), we need to show that the following equality of morphisms holds:
\begin{align}
	{\varphi^{\pm}_A} \circ {l^{D'(A)}} \circ c^{\pm}_{D'(A),A} \;=\; {r^{D(A)}} \circ (\varphi^{\pm}_{A} \otimes A),
\end{align}
where $r^{D(A)}\colon D(A) \otimes A \,\to\, D(A)$ is the right $A$-action on $(K\multimapinv A)\,=\, D(A)$ in Lemma~\ref{lemma: module structure on left internal hom}, induced by the multiplication $\mu$ of $A$. Suppressing associators, we compute:
\allowdisplaybreaks
\begin{align*}
	& \varphi^{\pm}_A 
	\circ {l^{D'(A)}} 
	\circ c^{\pm}_{D'(A),A}
	\\
	&\quad\eqabove{(1)}\; 
	\varphi^{\pm}_A
	\circ {\operatorname{ev}^A_{D'(A)}}
	\circ {(A\otimes \beta_{A,A,K})} 
	\circ {\big(A\otimes D'(\mu)\big)} 
	\circ c^{\pm}_{D'(A),A}
	\\
	&\quad\eqabove{(2)}\; 
	\varphi^{\pm}_A
	\circ {\operatorname{ev}^A_{D'(A)}}
	\circ c^{\pm}_{A \multimap D'(A),A}
	\circ {\big(\beta_{A,A,K} \otimes A\big)} 
	\circ {\big(D'(\mu) \otimes A\big)} 
	\\
	&\quad\eqabove{(3)}\;
	\varphi^{\pm}_A
	\circ {\higheroverline{\operatorname{ev}}^A_{D'(A)}}
	\circ {c^{\mp}_{A,D'(A)\multimapinv A}}
	\circ {\big(A\otimes {\widetilde{c}^{\,\pm}_{A,D'(A)}}\big)}
	\circ c^{\pm}_{A \multimap D'(A),A}
	\circ {(\beta_{A,A,K} \otimes A)} 
	\circ {\big(D'(\mu) \otimes A\big)} 
	\\
	&\quad\eqabove{(4)}\;
	\varphi^{\pm}_A
	\circ {\higheroverline{\operatorname{ev}}^A_{D'(A)}}
	\circ {\big({\widetilde{c}^{\,\pm}_{A,D'(A)}} \otimes A\big)}
	\circ {(\beta_{A,A,K} \otimes A)} 
	\circ {\big(D'(\mu) \otimes A\big)} 
	\\
	&\quad\eqabove{(5)}\; 
	{\higheroverline{\operatorname{ev}}^A_{D(A)}}
	\circ \big((\varphi^{\pm}_A\multimapinv A) \otimes A\big)
	\circ {\big({\widetilde{c}^{\,\pm}_{A,D'(A)}} \otimes A\big)}
	\circ {(\beta_{A,A,K} \otimes A)} 
	\circ {\big(D'(\mu) \otimes A\big)} 
	\\
	&\quad\eqabove{(6)}\;
	{\higheroverline{\operatorname{ev}}^A_{D(A)}}
	\circ (\higheroverline{\beta}_{A,A,K}\otimes A)
	\circ \big(D(\mu)\otimes A\big)
	\circ (\varphi^{\pm}_{A}\otimes A)
	\\
	&\quad\eqabove{(7)}\;
	{r^{D(A)}} 
	\circ (\varphi^{\pm}_{A} \otimes A).
\end{align*} 
Equation (1) follows from Lemma~\ref{lemma: rewrite inHom action}; (2) and (4) follow from the naturality of $c^{\pm}$; (3) from Equation~\eqref{eq: ctilde compat ev} in Lemma~\ref{lemma: tildec ev and coev}; (5) from the naturality of $\higheroverline{\operatorname{ev}}^{A}$; and (7) from Lemma~\ref{lemma: rewrite inHom action}. 

\smallskip

It remains to establish Equation (6). By Yoneda's lemma, it suffices to show that for every $X\in \cC$, the two maps 
\(\operatorname{Hom}(X, D'(A) \otimes A) \to \operatorname{Hom}(X, D(A))\) obtained by applying \(\operatorname{Hom}_{\cC}(X,-)\) to both sides of Equation (6) coincide. 

Unwinding the definitions of \(\beta\), \(\higheroverline{\beta}\), and \(\widetilde{c}^{\,\pm}\), this reduces to verifying
\begin{align}\label{eq: cplusminus mu}
	(\mu \otimes X) \circ (A \otimes c^{\pm}_{X,A}) \circ c^{\pm}_{X \otimes A,A} \;=\; c^{\pm}_{X,A} \circ (X \otimes \mu).
\end{align}
Equation \eqref{eq: cplusminus mu} follows directly from the hexagon axioms and naturality of the braiding $c^{\pm}$, together with the commutativity of $\mu$.
\end{proof}

\medskip

\subsection{Applications}\label{app: Applications}\hfill\\[0.3cm]
For the proof of Proposition~\ref{prop: category of bimodules is gv}, we need the following result.
\begin{lemma}\label{lemma: equalizer of left module into dual coalgebra}
	Let \(A\) be a GV-algebra in a GV-category \(\cC\), and let \(M\in {_A\cC}\). The diagram
	\begin{equation}
		\begin{tikzcd}
			{D'(M)}&&&{M\multimap D'(A)} &&&& {(A\otimes M)\multimap D'(A)}
			\arrow[->, "{\beta_{A,M,K}} \,\circ\, {D'(l^M)}"{yshift=1.5pt}, from=1-1, to=1-4]
			\arrow[->,shift left=.7ex, "l^M\,\multimap\, D'(A)"{yshift=1.5pt}, from=1-4, to=1-8]
			\arrow[->,shift left=-.7ex, "{\big((A\,\otimes \,M)\,\multimap \,l^{D'(A)}\big)}\,\circ\,{\underline{A\otimes }_{M,D'(A)}}"'{yshift=-1.5pt}, from=1-4, to=1-8]
		\end{tikzcd}
	\end{equation}
	is an equalizer diagram in \(\cC\).
\end{lemma}
\begin{proof}
	By specializing Lemma~\ref{lemma: rewrite coreflexive pair} to the case \(L=A\) and \(N=K\), we obtain
	\begin{align}
		{\big((A\otimes M)\multimap l^{D'(A)}\big)} \circ {\underline{A\otimes}_{M,D'(A)}} &\,=\, \beta_{A,A\otimes M,K} \circ D'(\mu\otimes M) \circ \beta^{-1}_{A,M,K}.
	\end{align}
	Using the naturality of \(\beta\) and the fact that \(D'\) is an antiequivalence, the claim now follows from the observation that the following diagram is a coequalizer diagram:
	\begin{equation}
		\begin{tikzcd}
			{A\otimes A \otimes M} & {A\otimes M} & {M.}
			\arrow[->,shift left=.7ex, "A\otimes l^M"{yshift=1.5pt}, from=1-1, to=1-2]
			\arrow[->,shift left=-.7ex, "{\mu \otimes M}"'{yshift=-1.5pt}, from=1-1, to=1-2]
			\arrow[->, "l^M"{yshift=1.5pt}, from=1-2, to=1-3]
		\end{tikzcd}
		\vspace{-0.4cm}
	\end{equation}
\end{proof}

\begin{proof}[Proof of Proposition~\ref{prop: category of bimodules is gv}]
	We verify the hypotheses of Theorem~\ref{main thm}. By Proposition~\ref{lemma: bimodules closed}, the monoidal category ${_A}\cC_A$ is closed. The forgetful functor $U_A\colon {_A}\cC_A\rightarrow \cC$ is lax monoidal (see Remark~\ref{rem: monoidal category of bimodules}) and conservative. It thus suffices to show that the form
	\begin{equation}
		\upsilon^{0,U_A}\colon \; D'(A)\, \xrightarrow{D'(\eta)}\, D'(1) \xrightarrow{\gamma_K}\, K
	\end{equation}
	for \(U_A\) is Frobenius.
	
	Let $M \in {_A}\cC_A$.
	Consider the canonical monomorphism 
	\begin{equation}
		i_{M,D'(A)}\colon\, M \multimap_A D'(A)\;\hookrightarrow\; D \multimap D'(A).
	\end{equation}
	By Lemma~\ref{lemma: equalizer of left module into dual coalgebra}, there exists a unique isomorphism $g\colon\; M\multimap_A D'(A) \;\xlongrightarrow{\simeq}\; D'(M)$ satisfying
	\begin{align}\label{eq: A-bimod 1}
		D'(l^M)\circ g &\;=\; \beta^{-1}_{A,M,K} \circ i_{M,D'(A)}.
	\end{align}
	Postcomposing~\eqref{eq: A-bimod 1} with the morphism
	\begin{equation}
		D'(A\otimes M)\; \xrightarrow{D'(\eta \otimes M)}\; D'(1\otimes M)\; \xrightarrow{D'(\lambda^{-1}_M)}\; D'(M),
	\end{equation}
	and using the unitality of the left $A$-action \(l^M\colon A\otimes M \to M\), we find
	\begin{align}\label{eq: A-bimod 2}
		g &\;=\; D'(\lambda_M)^{-1}\circ D'(\eta \otimes M) \circ \beta^{-1}_{A,M,K} \circ i_{M,D'(A)}.
	\end{align}
	By naturality of \(\beta\) and Equation~\eqref{eq: relation beta and gamma1} in Lemma~\ref{lemma: relation beta and gamma}, the right-hand side of~\eqref{eq: A-bimod 2} becomes
	\begin{align}
		{(M\multimap \gamma_K)} \circ {(M\multimap D'(\eta))} \circ i_{M,D'(A)},
	\end{align}
	which by Remark~\ref{rem: comparator for U} is equal to
	\begin{align}
		{(M\multimap \gamma_K)} \circ {(M\multimap D'(\eta))} \circ \tau^{l,U_A}_{M,D'(A)}.
	\end{align}
	 Thus, \(g\) coincides with the duality transformation of Definition~\ref{def: duality transformations} associated with \(\upsilon^{0,U_A}\):
	\begin{equation}
		\xi^{l,U_A}_M \colon\;  U_A(M\multimap_A D'(A))\, \longrightarrow\, D'U_A(M)
	\end{equation}
	In particular, \(\xi^{l,U_A}_M\) is invertible.
	
	An analogous argument shows that the composite $\big((\higheroverline{\gamma}_K\circ D(\eta)) \multimapinv M\big)\circ \tau^{r,U_A}_{D(A),M}$ is also invertible. With the following computation,
	\allowdisplaybreaks
	\begin{align*}
		\big((\higheroverline{\gamma}_K\circ D(\eta)) \multimapinv M\big)\circ \tau^{r,U_A}_{D(A),M}
		&\quad \eqabove{(1)}\;
		\big((\higheroverline{\gamma}_K\circ D(\eta)) \multimapinv M\big)\circ (f^{-1} \multimapinv M) \circ \tau^{r,U_A}_{D'(A),M}\circ (f\multimap_A M)\\[0.4em]
		&\quad\eqabove{(2)}\;
		\big((\gamma_K\circ D'(\eta)) \multimapinv M\big) \circ \tau^{r,U_A}_{D'(A),M}\circ (f\multimap_A M)\\[0.4em]
		&\quad\eqabove{def}\;
		\xi^{r,U_A}_M \circ (f\multimap_A M),
	\end{align*}
	we conclude that the right duality transformation \(\xi^{r,U_A}_M\) associated to \(\upsilon^{0,U_A}\) is also invertible. Here, Equation (1) follows from the naturality of \(\tau^{r,U_A}\); and (2) from the counitality of \(f\) (see Equation~\eqref{def:f is counital} in Remark~\ref{rem: coalgebra and in homs}). All in all, \(\upsilon^{0,U_A}\) is a Frobenius form. 
\end{proof}

\section{Coherence axioms}\label{app: coherence diagrams}
\subsection{LD-categories}
The distributors are required to be compatible with the unitors 
\begin{align*}
	\lou\colon {\ot} \circ (1\tim \idC)&\sxlongrightarrow{\simeq} \idC,\\[0.3em]
	\rou \colon {\ot} \circ (\idC \tim 1)&\sxlongrightarrow{\simeq} \idC,\\[0.3em]
	\lpu\colon {\parLL} \circ (K\tim \idC)&\sxlongrightarrow{\simeq} \idC,\\[0.3em]
	\rpu \colon {\parLL} \circ (\idC \tim K)&\sxlongrightarrow{\simeq} \idC,
\end{align*}
in that, for all \(X,Y\in \cC\), the following four triangle diagrams have to commute:
\begin{align*}
	(\lou_X \parLL Y) \circ \distl_{1,X,Y} &\eq \lou_{X\parLL Y}.\label{eq:A1}\tag{A1}\\[0.4em]
	(X \parLL \rou_Y) \circ \distr_{X,Y,1} &\eq \rou_{X\parLL Y}.\label{eq:A2}\tag{A2}\\[0.4em]
	\lpu_{X \otimes Y} \circ \distr_{K,X,Y} &\eq \lpu_{X} \ot Y.\label{eq:A3}\tag{A3}\\[0.4em]
	\rpu_{X \otimes Y} \circ \distl_{X,Y,K} &\eq X \ot \rpu_{Y}.\label{eq:A4}\tag{A4}
\end{align*}
The distributors are required to be compatible with the associators 
\begin{align*}
	\ao \colon {\otimes} {\,\circ\,} {(\operatorname{id_\cC} \times \ot)} &\sxlongrightarrow{\simeq} {\ot} {\,\circ\,} {(\otimes \times \operatorname{id_\cC})},\\[0.3em]
	\ap \colon {\parLL} {\,\circ\,} {(\operatorname{id_\cC} \times \parLL)} &\sxlongrightarrow{\simeq} {\parLL} {\,\circ\,} {(\parLL \times \operatorname{id_\cC})},
\end{align*}
in that, for all \(W,X,Y,Z\in \cC\), the following six pentagon diagrams have to commute:
\begin{align}
	\distl_{W \ot X,Y,Z} \circ \ao_{W,X,Y \parLL Z} &\eq (\ao_{W,X,Y}\parLL Z)\circ (\distl_{W,X\ot Y,Z}) \circ (W \ot \distl_{X,Y,Z}). \label{eq:A5}\tag{A5}\\[0.4em]
	(W \parLL \ao_{X,Y,Z}) \circ \distr_{W,X,Y \ot Z} &\eq \distr_{W,X\ot Y,Z} \circ (\distr_{W,X,Y}\ot Z) \circ \ao_{W\parLL X,Y,Z}.\label{eq:A6}\tag{A6}\\[0.4em]
	\distr_{W \parLL X,Y,Z} \circ (\ap_{W,X,Y}\ot Z) &\eq \ap_{W,X,Y \ot Z} \circ (W \parLL \distr_{X,Y,Z}) \circ \distr_{W,X \parLL Y,Z}.\label{eq:A7}\tag{A7}\\[0.4em]
	\ap_{W\ot X,Y,Z} \circ \distl_{W,X,Y\parLL Z} &\eq (\distl_{W,X,Y}\parLL Z) \circ \distl_{W,X \parLL Y,Z} \circ (W \ot \ap_{X,Y,Z}). \label{eq:A8}\tag{A8}\\[0.4em]
	\distl_{W,X,Y \ot Z} \circ (W\ot \distr_{X,Y,Z}) &\eq \distr_{W \ot X,Y,Z}\circ (\distl_{W,X,Y} \ot Z) \circ \ao_{W,X \parLL Y,Z}.\label{eq:A9}\tag{A9}\\[0.4em]
	(\distr_{W,X,Y} \parLL Z) \circ \distl_{W\parLL X,Y,Z} &\eq \ap_{W,X\ot Y,Z} \circ (W \parLL \distl_{X,Y,Z}) \circ \distr_{W, X,Y \parLL Z}.\label{eq:A10}\tag{A10}
\end{align}

\bigskip

\subsection{Right LD-dualizability}
The following \emph{snake equations}
\begin{align}\label{firstzigzag}\tag{S1}
	({\epsilon^X} \parLL \Eev\, X) \circ \distl_{{\,\Eev\, X},X,{\Eev\, X}} \circ (\Eev\, X \otimes {\eta^X}) & \;=\; {(\lpu_{\Eev\, X})^{-1}}\circ {\rou_{\,\Eev\, X}},\\[0.4em] \label{secondzigzag}\tag{S2}
	({X} \parLL {\epsilon^X})\circ\distr_{\,X,{\Eev\, X},X}\circ({\eta^X}\otimes {X}) & \;=\;{(\rpu_X)^{-1}}\circ {\lou_X},
\end{align}
are required to hold for $X\in \cC$.

\bigskip

\subsection{Frobenius LD-functors}
The following \emph{Frobenius relations}
\begin{align}
	\upsilon^{2,F}_{X\otimes Y,Z}\circ F(\distl_{X,Y,Z})\circ \varphi^{2,F}_{X,Y\parLL Z}\:=\:\big(\varphi^{2,F}_{X,Y}\parLL F(Z)\big)\circ \distl_{F(X),F(Y),F(Z)}\circ \big(F(X)\otimes \upsilon^{2,F}_{Y,Z}\big),\label{eq:F1}\tag{F1}\\[0.4em]
	\upsilon^{2,F}_{X,Y\otimes Z}\circ F(\distr_{X,Y,Z})\circ \varphi^{2,F}_{X \parLL Y,Z}\:=\:\big(F(X)\parLL{\varphi^{2,F}_{Y,Z}}\big)\circ \distr_{F(X),F(Y),F(Z)}\circ \big(\upsilon^{2,F}_{X,Y}\otimes F(Z)\big),\label{eq:F2}\tag{F2}
\end{align}
are required to hold for all \(X,Y,Z\in \cC\).

\vfill

{\centering\textsc{Acknowledgments}\\[0.6em]}
I am grateful to the anonymous referee for many valuable comments and suggestions, in particular for proposing the new title of the paper and for suggesting Examples~\ref{ex: non-con GV} and \ref{ex: fun wo Frob}. I thank my advisor, Christoph Schweigert, for his insightful comments and careful reading of much of this manuscript. I am grateful to Raschid Abedin, Sam Bauer, Gabriella Böhm, Aaron Hofer, Ulrich Krähmer, JS Lemay, \mbox{Paul-André Melliès}, Catherine Meusburger, Chris Raymond, Matti Stroiński, and Tony Zorman for correspondence and valuable discussions.

The author is funded by the DFG through the CRC 1624 \emph{Higher Structures, Moduli Spaces and Integrability}, project number 506632645. He also acknowledges support from the DFG under Germany’s Excellence Strategy, EXC 2121 \emph{Quantum Universe}, project number 390833306.

\bibliographystyle{alphaurl}
\bibliography{references}

\vspace{0.8cm}

\nid\footnotesize\textsc{Fachbereich Mathematik, Universität Hamburg, Bundesstraße 55, 20146 Hamburg, Germany}

\nid\textit{Email address:} \texttt{max.demirdilek@uni-hamburg.de}
\end{document}

%% file: preamble.tex
\usepackage{amsmath,amsthm,amssymb,amsfonts}
\usepackage{mathtools}
\usepackage{mathrsfs}

\usepackage{graphicx}
\usepackage{quiver}

\usepackage{enumitem}
\usepackage{float}
\usepackage[countmax]{subfloat} 

\usepackage{titletoc}
\usepackage[titletoc]{appendix}
\usepackage{caption}

\usepackage{tabularx}
\usepackage{array}

\usepackage{geometry}
\usepackage{scalerel}
\usepackage{adjustbox}

\usepackage{leftindex}
\usepackage[multiple]{footmisc}

\let\oldbibliography\thebibliography
\renewcommand{\thebibliography}[1]{%
	\oldbibliography{#1}%
	\setlength{\itemsep}{4pt}
}

\usepackage{bbm}  
\usepackage{stmaryrd}

\usepackage{etoolbox}  
\usepackage{verbatim}

\usepackage[hyperfootnotes=false,
draft=false,
pdftex, 
colorlinks=true, 
linkcolor=black, 
citecolor=black, 
urlcolor=black, 
hypertexnames=false
]{hyperref}

\setlist[enumerate]{itemsep=0.3em, partopsep=0.3em}
\setlist[itemize]{itemsep=0.3em, partopsep=0.3em}

\newenvironment{myitemize}
{ \begin{itemize}
		\setlength{\itemsep}{0pt}
		\setlength{\partopsep}{0pt}     }
	{ \end{itemize}                  }

\numberwithin{equation}{section}      

\newcommand{\eq}{\,=\,}

\newcommand\eqdef{\mathrel{\overset{\makebox[0pt]{\mbox{\normalfont\tiny\sffamily def}}}{\,=\,}}}

\newcommand\eqabove[1]{\mathrel{\overset{\makebox[0pt]{\mbox{\normalfont\tiny #1}}}{=}}}

\newcommand\congabove[1]{\mathrel{\overset{\makebox[0pt]{\mbox{\normalfont\tiny #1}}}{\;\cong\;}}}

\newcommand{\highoverline}[1]{\overline{\vphantom{#1^{\rule{0pt}{0.9ex}}}#1}}
\newcommand{\higheroverline}[1]{\overline{\vphantom{#1^{\rule{0pt}{0.3ex}}}#1}}

\newcommand{\id}{{\operatorname{id}}}

\newcommand{\nid}{\noindent}

\newcommand{\idC}{{\operatorname{id}_{\cC}}}

\newcommand{\tim}{\times}
\newcommand{\rD}{D^{\prime}}
\newcommand{\ot}{\otimes}

\newcommand{\tikztriangleright}[1][gray,fill=gray]{
	\mathbin{
		\scalerel*{
			\tikz \draw[rounded corners=0.1pt,#1]
			(0,-2.5pt)--++(0,5pt)--++(-30:5pt)--cycle;
		}{\triangleright}
	}
}

\newcommand{\rotmultimap}[2]{\mathbin{\rotatebox[origin=c]{180}{$\multimap$}}}

\makeatletter 
\newcommand*{\multimapinv}{\mathrel{\mathpalette\multimap@inv\relax}}
\newcommand*{\multimap@inv}[2]{\reflectbox{$#1\multimap$}}
\makeatother

\newcommand\Eev[1] {{{}^{\vee\!}}\!{#1}}
\newcommand{\parll}{{\mathpalette\rotamp\relax}}
\newcommand{\rotamp}[2]{\rotatebox[origin=c]{180}{$#1\&$}}
\DeclareMathOperator{\parLL}{\parll}

\newcommand{\distl}{\delta^l}
\newcommand{\distr}{\delta^r}

\newcommand{\ao}{\alpha}
\newcommand{\aoi}{\alpha^{-1}}
\newcommand{\ap}{\higheroverline{\alpha}} 
\newcommand{\api}{\higheroverline{\alpha}^{-1}}

\newcommand{\lou}{\lambda} 
\newcommand{\rou}{\rho} 
\newcommand{\lpu}{\higheroverline{\lambda}} 
\newcommand{\rpu}{\higheroverline{\rho}}

\newcommand{\ra}{\rightarrow}

\newcommand{\sxlongrightarrow}[1]{\,\xlongrightarrow{#1}\,}

\def\cC{\mathcal C}
\def\cD{\mathcal D}
\def\cE{\mathcal E}

\theoremstyle{plain}
\newtheorem{theorem}{Theorem}[section]
\newtheorem{lemma}[theorem]{Lemma}
\newtheorem{corollary}[theorem]{Corollary}          
\newtheorem{proposition}[theorem]{Proposition}              
          
\newtheorem{prop}[theorem]{Proposition}      
\newtheorem*{liftingthm}{Theorem A}
\newtheorem*{sndthm}{Theorem B}
\newtheorem*{thdthm}{Theorem C}
\newtheorem{theorem*}[]{Theorem}

\theoremstyle{definition} 
\newtheorem{definition}[theorem]{Definition}

\theoremstyle{remark}
\newtheorem{remark}[theorem]{Remark}
\newtheorem{example}[theorem]{Example}

\newtheorem{construction}[theorem]{Construction}